\documentclass[10pt,reqno]{amsart}

\usepackage[affil-it]{authblk}
\usepackage{titling}
\usepackage[latin1]{inputenc}
\usepackage[english]{babel}
\usepackage{calligra}
\usepackage[OT1]{fontenc}
\usepackage{amsfonts}
\usepackage{amsmath}
\usepackage{amssymb}
\usepackage{amsthm}
\usepackage{faktor}
\usepackage{float}
\usepackage{enumerate}
\usepackage{color}
\usepackage{pdfpages}
\usepackage{esint}
\usepackage{hyperref}
\usepackage{cite}
\usepackage{fancyhdr}
\usepackage{ulem}
\usepackage{mathtools}
\usepackage{mathrsfs}

\usepackage{etoolbox}
\patchcmd{\section}{\scshape}{\bfseries}{}{}
\makeatletter
\renewcommand{\@secnumfont}{\bfseries}
\makeatother
\newcommand\testname{\textbf{ Abstract}}
\makeatletter
\newenvironment{abs}{%
    \small
    \begin{center}%
        {\textsc \testname\vspace{-.2em}\vspace{\z@}}%
    \end{center}%
    \quote
    }
   {\endquote}
\makeatother

\advance\hoffset by 5mm

\allowdisplaybreaks[4]

\DeclareMathOperator*{\Id}{Id}

\newcommand{\tr}{{}^\mathrm{t} }
\newcommand{\Div}{\mathrm{div}}
\newcommand{\trc}{\mathrm{tr}}

\newcommand{\ee}{\varepsilon}

\newcommand{\Aa}{\mathcal{A}}
\newcommand{\CC}{\mathcal{C}}
\newcommand{\Bb}{\mathcal{B}}
\newcommand{\dd}{\mathrm{d}}

\newcommand{\J}{\mathcal{J}}
\newcommand{\I}{\mathcal{I}}

\newcommand{\EE}{\mathcal{E}}

\newcommand{\FF}{\mathcal{F}}

\newcommand{\RR}{\mathbb{R}}
\newcommand{\ZZ}{\mathbb{Z}}

\newcommand{\BB}{\dot{B}}
\newcommand{\Hh}{\dot{H}}
\newcommand{\Dd}{\dot{\Delta}}
\newcommand{\Sd}{\dot{S}}
\newcommand{\Rd}{\dot{R}}
\newcommand{\R}{\mathbb{R}}
\newcommand{\non}{\nonumber}

\newtheorem{theorem}{Theorem}[section]

\newtheorem{proposition}[theorem]{Proposition}
\newtheorem{lemma}[theorem]{Lemma}

\newtheorem{remark}[theorem]{Remark}

\definecolor{grey}{rgb}{0.85,0.85,0.85}
\date{}

\title{\Large 
	\textbf{\uppercase{Uniqueness of weak solutions of the full\\coupled Navier-Stokes and Q-tensor system in 2D}}}

\author{Francesco De Anna$^1\quad$}
\author{$\quad$Arghir Zarnescu$^2,^3$}

\affil{	\textsc{$\,^1$Universit\'e de Bordeaux} \\ 
		\small\textsc{Institut de Math\'ematiques de Bordeaux}\\ 
		\small{F-33405 Talence Cedex, France}\\
		\small\textnormal{Francesco.Deanna@math.u-bordeaux1.fr}
		\vspace{0.3cm}}

\affil{	\textsc{$\,^2$University of Sussex} \\ 
		\small\textsc{Pevensey 3 5C17}\\ 
		\small{Brighton BN1 9QF, UK}\\
		\small\textnormal{A.Zarnescu@sussex.ac.uk}
		}
\affil{\textsc{$\,^3$ ``Simion Stoilow" Institute of Mathematics}\\
     \textsc{ of the Romanian Academy}\\
        \small\textnormal{ 21 Calea Grivitei Street}\\
        \small\textnormal{  010702 Bucharest, Romania}
}		
		
\date{\today}		
 
\markleft{Francesco De Anna $\quad$ Arghir  Zarnescu}
\begin{document}

\maketitle
\begin{abs}
	This paper is devoted to the full system of incompressible liquid crystals, as modeled in the Q-tensor framework. The main purpose is to establish the uniqueness of weak solutions 
	in a two dimensional setting, without imposing an extra regularity on the solutions themselves. 
	This result only requires the initial data to fulfill the features which allow the existence of a weak solution. 
	Thus, we also present a revisit of the global existence result in dimension two and three.  
\end{abs}

\section{Introduction}
The main aim of this article is to prove the uniqueness of weak solutions for a type of  coupled Navier-Stokes and Q-tensor systems proposed in \cite{berised} and studied numerically and analytically in \cite{pz2,equationsELreduction, abels, gonzalez, gonzalez2, chineseseveral}. This type of system models nematic liquid crystals and provides in a certain sense an extension of the classical Ericksen-Leslie model \cite{equationsELreduction}, whose uniqueness of weak solutions was proved in \cite{Wang2016919}. In the remainder of this introduction we will briefly present the equations and state our main  result. 

\bigskip
The system models the evolution of liquid crystal molecules together with the underlying flow, through a parabolic-type system coupling an incompressible Navier-Stokes system with a nonlinear convection-diffusion system. The  local orientation of the molecules is described through a function $Q$ taking values from  $\R_+\times\Omega\subset\R_+\times\R^d$,$d=2,3$ into the set of so-called $d$-dimensional $Q$-tensors that is  $$S_0^{(d)}\stackrel{\rm{def}}{=} \left\{Q \in \mathbb{M}^{d\times d};
Q_{ij}=Q_{ji},\textrm{tr}(Q) = 0, i,j=1,\dots,d \right\}$$ (the most relevant physical situations being $d=2,3$). The evolution of the $Q$'s is driven by  a gradient flow of the free energy of the molecules as well as the transport, distortion and alignment effects caused by the flow. 
 The flow field  $u:\R_+\times\Omega\to\mathbb{R}^{d}$  satisfies a forced incompressible Navier-Stokes system, with the forcing provided by the additional, non-Newtonian stress caused by the molecules orientations, thus expressed in terms of $Q$. 
We restrict ourselves to the case $\Omega=\R^d$ and work with non-dimensional quantities. The evolution of $Q$ is given by:

\begin{equation}
\partial_t Q+u\cdot \nabla Q-S(\nabla u,Q)=-\Gamma \frac{\partial \mathcal{F}_e}{\partial Q}
\end{equation} with $\Gamma>0$. Here 
\begin{equation}
\mathcal{F}_e(Q)=\int_{\mathbb{R}^d} \frac{L}{2}|\nabla Q|^2+\big(\frac{a}{2}\textrm{tr}(Q^2)-\frac{b}{3}\textrm{tr}(Q^3)+\frac{c}{4}\textrm{tr}^2(Q^2))\,dx
\label{freeenergy}
\end{equation} is the free energy of the liquid crystal molecules and $ \frac{\partial \mathcal{F}_e}{\partial Q}$ denotes the variational derivative. The $L,a,b,c$ constants are specific to the material with:
\begin{equation}\label{c+}
L>0\textrm{ and }a,b,c\in \mathbb{R}, \,c>0
\end{equation}

  If $u=0$ the $Q$-tensor equation would simply be a gradient flow of the free energy. For $u\not=0$ the molecules are transported by the flow (as indicated by the convective derivative $\partial_t+u\cdot\nabla$) as well as being tumbled and aligned by the flow, fact described by the term 

\begin{equation}
 S(\nabla u,Q)\stackrel{\rm{def}}{=}(\xi D+\Omega)(Q+\frac{1}{d}Id)+(Q+\frac{1}{d}Id)(\xi D-\Omega)-2\xi (Q+\frac{1}{d}Id)\textrm{tr}(Q\nabla u)
 \label{rel:defS}
\end{equation} where $D\stackrel{def}{=}\frac{1}{2}\left(\nabla u+(\nabla u)^T\right)$ and $\Omega\stackrel{def}{=}\frac{1}{2}
\left(\nabla u-(\nabla u)^T\right)$ are, respectively, the symmetric part and the antisymmetric part, of the velocity gradient matrix $\nabla u$.  The constant $\xi$ is specific to the liquid crystal material.

The flow satisfies the forced Navier-Stokes system:

\begin{align*}
\partial_t u+(u\cdot\nabla)u&=\nu   \Delta u+\nabla p+ \lambda\nabla\cdot(\tau+\sigma)\\
          \nabla\cdot u&=0
\end{align*} where $\nu,\lambda>0$ with $\lambda$ measuring the ratio of the elastic effects (produced by the liquid crystal molecules) to that of the diffusive effects. The forcing is provided by the additional stress caused by the presence of the liquid crystal molecules,  more specifically we have the symmetric part of the additional stress tensor:
\begin{equation}
\begin{aligned}
\tau\stackrel{\rm{def}}{=}\Big[-\xi \Big(Q+\frac{1}{d}Id\big)H-\xi H\big(Q+\frac{1}{d}Id\big)+2\xi \big(Q&+\frac{1}{d}Id\big)QH
-L \nabla  Q\odot  \nabla Q\Big]
\end{aligned}
\end{equation} 
and the antisymmetric part:
\begin{equation}
\sigma\stackrel{\rm{def}}{=}QH-HQ
\end{equation}  where we denoted
\begin{equation}
 H\stackrel{\rm{def}}{=}-\frac{\partial \mathcal{F}_e}{\partial Q}=L\Delta Q-aQ+b[Q^2-\frac{\textrm{tr}(Q^2)}{d}Id]-cQ\textrm{tr}(Q^2)
\label{rel:defH}
\end{equation}

Summarising we have the coupled system:

\begin{align}      (\partial_t+u\cdot \nabla)Q-S(\nabla u,Q)=&\Gamma(L\Delta Q-aQ+b[Q^2-\frac{\textrm{tr}(Q^2)}{d}Id]-cQ\textrm{tr}(Q^2))\non \\
         \partial_t u+(u\cdot\nabla)u=&\nu   \Delta u+\nabla p+ \lambda\nabla\cdot(QH-HQ)\non\\
            &+\lambda\nabla\cdot \Big[-\xi \Big(Q+\frac{1}{d}Id\big)H-\xi H\big(Q+\frac{1}{d}Id\big)+2\xi \big(Q+\frac{1}{d}Id\big)QH
-L \nabla  Q\odot \nabla Q\Big]\non\\
          \nabla\cdot u=&0
\label{system}          
     \end{align}

where $\Gamma, L,\nu,c>0$, $a,b\in\R$. Let us observe that this is a slight extension of the system considered in \cite{pz2}, where $\lambda=1$. However, this does not create any major difficulties compared to equations in \cite{pz2} but it is more relevant from a physical point of view. 

The main result of the paper is the uniqueness of weak solutions, which are defined in a rather standard manner:

\bigskip
\noindent\par{\bf Definition of weak solutions} { \it    A pair $(Q,u)$ is called a weak solution  of the system (\ref{system}), subject to initial data
\begin{equation}
Q(0,x)=\bar Q(x)\in H^1(\mathbb{R}^d;S_0^{(d)}),\, u(0,x)=\bar u(x)\in L^2(\mathbb{R}^d;\R^d), \nabla\cdot \bar u=0\,\textrm{ in }\mathcal{D}'(\mathbb{R}^d)
\label{initialdata}
\end{equation}
if   $Q\in L^\infty_{loc}(\mathbb{R}_+;H^1)\cap L^2_{loc}(\mathbb{R}_+;H^2)$, $u\in L^\infty_{loc}(\mathbb{R}_+;L^2)\cap L^2_{loc}(\mathbb{R}_+;H^1)$  and for every compactly supported $\varphi\in C^\infty([0,\infty)\times \mathbb{R}^d;  S_0^{(d)})$, $\psi\in C^\infty ([0,\infty)\times\mathbb{R}^d;\mathbb{R}^d)$ with $\nabla\cdot\psi=0$ we have
\begin{eqnarray}\int_0^\infty\int_{\mathbb{R}^d}(-Q\cdot    \partial_t\varphi-\Gamma L\Delta Q\cdot \varphi)- Q\cdot u\nabla_x \varphi dx\,dt\nonumber\\-\int_0^\infty\int_{\mathbb{R}^d}(\xi D+\Omega)(Q+\frac{1}{d}Id)\cdot\varphi+(Q+\frac{1}{d}Id)(\xi D-\Omega)\cdot\varphi-2\xi (Q+\frac{1}{d}Id)\textrm{tr}(Q\nabla u)\cdot\varphi\,dx\,dt\nonumber\\
=\int_{\mathbb{R}^d}\bar Q(x)\cdot \varphi(0,x)\,dx+\Gamma\int_0^\infty\int_{\mathbb{R}^d}\Big\{-aQ+b[Q^2-\frac{\textrm{tr}(Q^2)}{d}Id]-cQ\textrm{tr}(Q^2)\Big\}\cdot\varphi\,\,dx\,dt
\label{weaksol1}
\end{eqnarray} 
and
\begin{eqnarray}\int_0^\infty\int_{\mathbb{R}^d}-u   \partial_t\psi-u_\alpha u_\beta   \partial_\alpha\psi_\beta+\nu\nabla u\nabla\psi\,\, dt\,dx-\int_{\mathbb{R}^d}\bar u(x)\psi(0,x)\,dx\nonumber\\
=L\lambda\int_0^\infty\int_{\mathbb{R}^d} 
 Q_{\gamma\delta,\alpha} Q_{\gamma\delta,\beta}\psi_{\alpha,\beta} -Q_{\alpha\gamma}\Delta Q_{\gamma\beta}\psi_{\alpha,\beta}+\Delta Q_{\alpha\gamma}Q_{\gamma\beta}\psi_{\alpha,\beta}\,\,dx\,dt\nonumber\\
+\xi\lambda\int_0^\infty\int_{\mathbb{R}^d}\left(Q_{\alpha\gamma}+\frac{\delta_{\alpha\gamma}}{d}\right)H_{\gamma\beta}\psi_{\alpha,\beta}+H_{\alpha\gamma}\left(Q_{\gamma\beta}+\frac{\delta_{\gamma\beta}}{d}\right)\psi_{\alpha,\beta}-2(Q_{\alpha\beta}+\frac{\delta_{\alpha\beta}}{d})Q_{\gamma\delta}H_{\gamma\delta}\psi_{\alpha,\beta}\,dx\,dt
\label{weaksol2}
\end{eqnarray} 
}
We can now state our main result, which is the existence and uniqueness of weak solutions:

\begin{theorem}\label{thm: uniqueness} Let $d=2, 3$ and take 
$$Q(0,x)=\bar Q(x)\in H^1(\mathbb{R}^d;S_0^{(d)}),\, u(0,x)=\bar u(x)\in L^2(\mathbb{R}^d), \nabla\cdot \bar u=0\,\textrm{ in }\mathcal{D}'(\mathbb{R}^d)$$
Then the system \eqref{system} admits a global weak solution. Moreover if $d=2$, then uniqueness holds.
\end{theorem}
\begin{remark}
With minor modifications to the proof, that are left to the interested reader, the result also holds when the system is $d=2$ in the domain but $d=3$ in the target, which physically corresponds to a situation where there is no dependence in one of the three spatial directions.  
\end{remark}

The main part of the theorem is about uniqueness, as the existence part is just a fairly straightforward revisit of the arguments in \cite{pz2}. The main difficulties associated with treating the system \eqref{system} are related to the presence of the Navier-Stokes part. One can essentially think of the system as a highly non-trivial perturbation of a Navier-Stokes system. It is known that for Navier-Stokes alone the uniqueness of weak solutions in $2D$ can be achieved through rather standard arguments, while in $3D$ it is a major open problem.

The extended system that we deal with has an intermediary position, as the perturbation produced by the presence of the additional stress-tensor generates significant technical difficulties related in the first place to the weak norms available for the $u$ term. A rather common way of dealing with this issue is by using a weak norm for estimating the difference between the two weak solutions, a norm that is  below the natural spaces in which the weak solutions are defined. This approach was used before in the context of the related Leslie-Ericksen model \cite{titiuniqueness} as well as for the usual Navier-Stokes system in \cite{MR1813331} and \cite{MR2309504}.

 In our case, for technical convenience we  use a homogeneous Sobolev space, namely $\dot H^{-\frac{1}{2}}$. The fact that the initial data for the difference is zero (i.e. $(\delta u,\,\delta Q)_{t=0}=0$) helps in controlling the difference in such a low regularity space. However, one of the main reasons for chosing the homogeneous setting is a specific product law, see Theorem~\ref{theorem_product_homogeneous_sobolev_spaces} in the Appendix. The mentioned theorem shows that the product is a bounded operator from $\Hh^s(\RR^2)\times \Hh^t(\RR^2)$ into $\Hh^{s+t-1}(\RR^2)$, for any $|s|,\,|t|\leq 1$ such that  $s+t$ is positive. We note that evaluating the difference at regularity level $s=0$ i.e. in $L^2$, would only allow to prove a weak-strong uniqueness result, along the lines of \cite{MR2864407}. Working in a negative Sobolev space, $
\dot H^s$ with $s\in (-1,0)$ allows to capture the uniqueness of weak solutions.
 We expect that a similar proof would work in any $\dot H^s$ with $s\in (-1,0)$ and our choice $s=-\frac{1}{2}$ is just for convenience.
 
 Our main work is to obtain the delicate double-logarithmic type estimates that lead to an Osgood lemma, a generalization of the 
Gronwall inequality (see \cite{MR2768550}, Lemma 3.4). Indeed the uniqueness reduces to an estimate of the following type:

\begin{equation*}
	\Phi'(t) \leq \chi(t) \Big\{ \Phi(t) + \Phi(t)\ln\Big( 1+e +\frac{1}{\Phi(t)}\Big)
	+\Phi(t)
	\ln\Big( 1+e+  \frac{1}{\Phi(t)}\Big)
	 \ln \ln\Big( 1+e+ \frac{1}{\Phi(t)}\Big)\Big\} ,
\end{equation*}
where $\Phi(t)$ is a norm of the difference between two solutions and $\chi$ is apriori in $L^1_{loc}$.

\smallskip
In addition to these there are some difficulties that are specific to this system. These are  of two different types, being related to:

\begin{itemize}
\item controlling the ``extraneous" maximal derivatives: that is the highest derivatives in $u$ that appear in the $Q$ equation and the highest derivatives in $Q$ that appear in the $u$ equation,
\item controlling the high powers of $Q$ , such as $Q\textrm{tr}(Q^2)$ in particular those that interact with $u$ terms (such as $Q\textrm{tr}(Q\nabla u)$). 
\end{itemize}

\medskip
The first difficulty is dealt with  by taking into account the specific feature of the coupling that allows for the {\it cancellation of the worst terms},  when considering certain physically meaningful combinations of terms. This feature is explored in the next section where we revisit and revise the existence proof from \cite{pz2}. In what concerns the second difficulty, this is overcome by delicate harmonic analysis arguments leading to the double logarithmic estimates mentioned before. 

\medskip The paper is organised as follows: in the next section we revisit the existence arguments done in cite \cite{pz2}, providing a slight adaptation to our case and a minor correction to one of the estimates used there. The main benefit of this section is that it exhibits in a simple setting a number of cancellations that are later-on crucial for the uniqueness argument. In the third section we start by introducing a number of technical harmonic analysis tools related to the Littlewood-Paley theory and then use them in the proof of our main result. Some standard but perhaps less-known tools, toghether with some more technical estimates are postponed into the appendices.

\subsection*{Notations and conventions}  Let
$S_0^{(d)}\subset \mathbb{M}^{d\times d}$ denote the space of Q-tensors in dimension $d$,  i.e.
 $$S_0^{(d)}\stackrel{\rm{def}}{=} \left\{Q \in \mathbb{M}^{d\times d};
Q_{ij}=Q_{ji},\textrm{tr}(Q) = 0, i,j=1,\dots,d \right\}$$
We  use the Einstein summation convention, that is we  assume summation over repeated  indices. 

      We define the Frobenius norm of a matrix $
 \left| Q \right|\stackrel{\rm{def}}{=}\sqrt{\textrm{tr}Q^2} =\sqrt{ Q_{\alpha\beta}
Q_{\alpha\beta}}$ and define Sobolev spaces of $Q$-tensors in terms of this norm. For instance $H^1(\mathbb{R}^d,S_0^{(d)})\stackrel{\rm{def}}{=}\{Q:\mathbb{R}^d\to S_0^{(d)}, \int_{\mathbb{R}^d} |\nabla Q(x)|^2+|Q(x)|^2\,dx<\infty\}$ where  $|\nabla Q|^2(x)\stackrel{\rm{def}}{=}Q_{\alpha\beta,\gamma}(x)Q_{\alpha\beta,\gamma}(x)$ with $Q_{\alpha\beta,\gamma}\stackrel{\rm{def}}{=}   \partial_\gamma Q_{\alpha\beta}$. For $A,B\in S_0^{(d)}$ we denote $A:B=\textrm{tr}(AB)$, $|A|=\sqrt{\textrm{tr}(A^2)}$ and $\| (A,\,B) \|_X = \| A \|_X + \| B \|_X$, for any suitable Banach space $X$. We also denote $\Omega_{\alpha\beta}\stackrel{\rm{def}}{=}\frac{1}{2}\left(   \partial_\beta u_\alpha-   \partial_\alpha u_\beta\right)$,$u_{\alpha,\beta}\stackrel{\rm{def}}{=}   \partial_\beta u_\alpha$ and $(\nabla Q\odot \nabla Q)_{ij}=Q_{\alpha\beta,i}Q_{\alpha\beta,j}$.
\section{The  energy decay, apriori estimates { and scaling}}

      In the absence of the flow, when $u=0$ in the equations (\ref{system}), the free energy is a Lyapunov functional of the system. If $u\not=0$ we still have a Lyapunov functional for (\ref{system}) but this time one that includes the kinetic energy of the system. These estimates provide as usually the basis for obtaining apriori estimates for the system. The propositions in this section show this and their proofs follow closely the ones of the similar propositions in \cite{pz2} where they were  done for the case $\lambda=1$. The reason for including them is to display in relatively simple setting  the cancellations that will appear again in the proof of the uniqueness theorem but in a much more complicated framework.  We have:

\begin{proposition}
The system \eqref{system} has a Lyapunov functional:
\begin{equation}
E(t)\stackrel{\rm{def}}{=}\frac{1}{2}\int_{\mathbb{R}^d}|u|^2(t,x)\,dx+\int_{\mathbb{R}^d}\frac{L\lambda}{2} |\nabla Q|^2(t,x)+\lambda(\frac{a}{2}\textrm{tr}(Q^2(t,x))-\frac{b}{3}\textrm{tr}(Q^3(t,x))+\frac{c}{4}\textrm{tr}^2(Q^2(t,x)))\,dx
\end{equation}
      If $d=2,3$ and $(Q,u)$ is a smooth solution of (\ref{system}) such that $Q\in L^\infty(0,T; H^1(\mathbb{R}^d))\cap L^2(0,T;H^2(\mathbb{R}^d))$  and $u\in L^\infty(0,T;L^2(\mathbb{R}^d))\cap L^2(0,T;H^1(\mathbb{R}^d))$ then, for all $t<T$, we have:
\begin{equation}
\frac{d}{dt}E(t)=-\nu\int_{\mathbb{R}^d}|\nabla u|^2\,dx-\Gamma \lambda\int_{\mathbb{R}^d} \textrm{tr}\left(L\Delta Q-aQ+b[Q^2-\frac{\textrm{tr}(Q^2)}{d}Id]-cQ\textrm{tr}(Q^2)\right)^2\,dx\le 0
\label{energydecay}
\end{equation}
\label{prop:Lyapunov}
\end{proposition}

\smallskip     {\bf Proof.} We multiply the first equation in (\ref{system}) to the right by $-\lambda H$, take the trace, integrate over $\mathbb{R}^d$ and by parts and sum with the second equation multiplied by $u$ and integrated over $\mathbb{R}^d$ and by parts (let us observe that because of our assumptions on $Q$ and $u$ we do not have boundary terms, when integrating by parts). We obtain:

\begin{eqnarray}
\frac{d}{dt}\int_{\mathbb{R}^d}\frac{1}{2}|u|^2+\frac{L\lambda}{2}|\nabla Q|^2+\lambda(\frac{a}{2}\textrm{tr}(Q^2)-\frac{b}{3}\textrm{tr}(Q^3)+\frac{c}{4}\textrm{tr}^2(Q^2))\,dx\nonumber\\+\nu\int_{\mathbb{R}^d}|\nabla u|^2\,dx+\Gamma\lambda\int_{\mathbb{R}^d}\textrm{tr}\left(L\Delta Q L-aQ+b[Q^2-\frac{\textrm{tr}(Q^2)}{d}Id]-cQ\textrm{tr}(Q^2)\right)^2\,dx\nonumber\\=\underbrace{\lambda\int_{\mathbb{R}^d}u\cdot\nabla Q_{\alpha\beta}\left(-aQ_{\alpha\beta}+b[Q_{\alpha\gamma}Q_{\gamma\beta}-\frac{\delta_{\alpha\beta}}{d}\textrm{tr}(Q^2)]-cQ_{\alpha\beta}\textrm{tr}(Q^2))\right)\,dx}_{\stackrel{\rm{def}}{=}\mathcal{I}}\nonumber\\+\underbrace{\lambda\int_{\mathbb{R}^d}\left(-\Omega_{\alpha\gamma} Q_{\gamma\beta}+Q_{\alpha\gamma}\Omega_{\gamma\beta}\right)\left(-aQ_{\alpha\beta}+b[Q_{\alpha\delta}Q_{\delta\beta}-\frac{\delta_{\alpha\beta}}{d}\textrm{tr}(Q^2)]-cQ_{\alpha\beta}\textrm{tr}(Q^2))\right)\,dx}_{\stackrel{\rm{def}}{=}\mathcal{II}}
\nonumber
\end{eqnarray}
\begin{eqnarray}-\lambda\xi\underbrace{\int_{\mathbb{R}^d} \big(Q_{\alpha\gamma}+\frac{\delta_{\alpha\gamma}}{d}\big)D_{\gamma\beta}H_{\alpha\beta}dx}_{\stackrel{\rm{def}}{=}\mathcal{J}_1}-\lambda\xi\underbrace{\int_{\mathbb{R}^d} D_{\alpha\gamma}\big(Q_{\gamma\beta}+\frac{\delta_{\gamma\beta}}{d}\big)H_{\alpha\beta}dx}_{\stackrel{\rm{def}}{=}\mathcal{J}_2}
+2\lambda\xi\underbrace{\int_{\mathbb{R}^d}\big(Q_{\alpha\beta}+\frac{\delta_{\alpha\beta}}{d}\big)H_{\alpha\beta}\textrm{tr}(Q\nabla u)dx}_{\stackrel{\rm{def}}{=}\mathcal{J}_3}\nonumber\\
+L\lambda\underbrace{\int_{\mathbb{R}^d}u_{\gamma} Q_{\alpha\beta,\gamma}\Delta Q_{\alpha\beta}\,dx}_{\stackrel{\rm{def}}{=}\mathcal{A}}\underbrace{
-\frac{L\lambda}{2}\int_{\mathbb{R}^d} u_{\alpha,\gamma}Q_{\gamma\beta}\Delta Q_{\alpha\beta}\,dx}_{\stackrel{\rm{def}}{=}\mathcal{B}}
\nonumber
\end{eqnarray}
\begin{eqnarray}
\underbrace{+\frac{L\lambda}{2}\int_{\mathbb{R}^d}u_{\gamma,\alpha}Q_{\gamma\beta}\Delta Q_{\alpha\beta}}_{\stackrel{\rm{def}}{=}\mathcal{C}}\,dx
\underbrace{+\frac{L\lambda}{2}\int_{\mathbb{R}^d}Q_{\alpha\gamma}u_{\gamma,\beta}\Delta Q_{\alpha\beta}\,dx}_{\mathcal{C}}\underbrace{-\frac{L\lambda}{2}\int_{\mathbb{R}^d}Q_{\alpha\gamma}u_{\beta,\gamma}\Delta Q_{\alpha\beta}\,dx}_{\mathcal{B}}\nonumber\\
+L\lambda\underbrace{\int_{\mathbb{R}^d}Q_{\gamma\delta,\alpha}Q_{\gamma\delta,\beta}u_{\alpha,\beta}\,dx}_{\stackrel{\rm{def}}{=}\mathcal{AA}}\underbrace{-L\lambda\int_{\mathbb{R}^d} Q_{\alpha\gamma}\Delta Q_{\gamma\beta}u_{\alpha,\beta}\,dx}_{\stackrel{\rm{def}}{=}\mathcal{CC}}\underbrace{+L\lambda\int_{\mathbb{R}^d}\Delta Q_{\alpha\gamma}Q_{\gamma\beta}u_{\alpha,\beta}\,dx}_{\stackrel{\rm{def}}{=}\mathcal{BB}}\nonumber\\
+\lambda\xi\underbrace{\int_{\mathbb{R}^d} \big(Q_{\alpha\gamma}+\frac{\delta_{\alpha\gamma}}{d}\big)H_{\gamma\beta}u_{\alpha,\beta}\,dx}_{\stackrel{\rm{def}}{=}\mathcal{JJ}_1}+\lambda\xi\underbrace{\int_{\mathbb{R}^d}H_{\alpha\gamma}\big(Q_{\gamma\beta}+\frac{\delta_{\gamma\beta}}{d}\big)u_{\alpha,\beta}\,dx}_{\stackrel{\rm{def}}{=}\mathcal{JJ}_2}-2\lambda\xi\underbrace{\int_{\mathbb{R}^d} \big(Q_{\alpha\beta}+\frac{\delta_{\alpha\beta}}{d}\big)u_{\alpha,\beta}\textrm{tr}(QH)\,dx}_{\stackrel{\rm{def}}{=}\mathcal{JJ}_3}\nonumber
\end{eqnarray}
\begin{eqnarray}
=\underbrace{-L\lambda\int_{\mathbb{R}^d} u_{\alpha,\gamma}Q_{\gamma\beta}\Delta Q_{\alpha\beta}\,dx}_{2\mathcal{B}}\underbrace{+L\lambda\int_{\mathbb{R}^d}u_{\gamma,\alpha}Q_{\gamma\beta}\Delta Q_{\alpha\beta}\,dx}_{2\mathcal{C}}\nonumber\\ \underbrace{-L\lambda\int_{\mathbb{R}^d} Q_{\alpha\gamma}\Delta Q_{\gamma\beta}u_{\alpha,\beta}\,dx}_{\mathcal{CC}}\underbrace{+L\lambda\int_{\mathbb{R}^d}\Delta Q_{\alpha\gamma}Q_{\gamma\beta}u_{\alpha,\beta}\,dx}_{\mathcal{BB}}=0
\label{Lyapunovcancellation}
\end{eqnarray} where $\mathcal{I}=0$ (since $\nabla\cdot u=0$), $\mathcal{II}=0$ (since $Q_{\alpha\beta}=Q_{\beta\alpha}$) and for the second equality we used

\begin{eqnarray}\underbrace{\int_{\mathbb{R}^d}u_\gamma Q_{\alpha\beta,\gamma}\Delta Q_{\alpha\beta}\,dx}_{\mathcal{A}}\underbrace{+\int_{\mathbb{R}^d}Q_{\gamma\delta,\alpha}Q_{\gamma\delta,\beta}u_{\alpha,\beta}\,dx}_{\mathcal{AA}}=\int_{\mathbb{R}^d}u_\gamma Q_{\alpha\beta,\gamma}\Delta Q_{\alpha\beta}\,dx\nonumber\\-\int_{\mathbb{R}^d}Q_{\gamma\delta,\alpha}Q_{\gamma\delta,\beta\beta}u_\alpha\,dx-
\int_{\mathbb{R}^d}Q_{\gamma\delta,\alpha\beta}Q_{\gamma\delta,\beta}u_\alpha\,dx=\int_{\mathbb{R}^d}\frac{1}{2}Q_{\gamma\delta,\beta}Q_{\gamma\delta,\beta}u_{\alpha,\alpha}\,dx=0\nonumber
\end{eqnarray} together with $Q_{\alpha\alpha}=H_{\alpha\alpha}=u_{\alpha,\alpha}=0$, $\mathcal{J}_3=\mathcal{JJ}_3$ and

\begin{eqnarray}
\mathcal{J}_1+\mathcal{J}_2= \int_{\mathbb{R}^d} \frac{1}{2}Q_{\alpha\gamma}u_{\gamma,\beta}H_{\alpha\beta}+\frac{1}{2}Q_{\alpha\gamma}u_{\beta,\gamma}H_{\alpha\beta}+\frac{1}{2}u_{\alpha,\gamma}Q_{\gamma\beta}H_{\alpha\beta}+\frac{1}{2}u_{\gamma,\alpha}Q_{\gamma\beta}H_{\alpha\beta}\,dx\nonumber\\+\frac{2}{d}\int_{\mathbb{R}^d}D_{\alpha\beta}H_{\alpha\beta}
=\int_{\mathbb{R}^d}\frac{1}{2}\big(Q_{\alpha\gamma}u_{\gamma,\beta}H_{\alpha\beta}+u_{\gamma,\alpha}Q_{\gamma\beta}H_{\alpha\beta}\big)+\frac{1}{2}\big(Q_{\alpha\gamma}u_{\beta,\gamma}H_{\alpha\beta}+u_{\alpha,\gamma}Q_{\gamma\beta}H_{\alpha\beta}\big)\,dx\nonumber\\+\frac{1}{d}\int_{\mathbb{R}^d}(u_{\alpha,\beta}+u_{\beta,\alpha})H_{\alpha\beta}\,dx
=\int_{\mathbb{R}^d} H_{\beta\alpha}Q_{\alpha\gamma}u_{\gamma,\beta}+ Q_{\gamma\alpha}H_{\alpha\beta}u_{\beta,\gamma}\,dx+\frac{2}{d}\int_{\mathbb{R}^d} u_{\alpha,\beta}H_{\alpha\beta}\,dx=\mathcal{JJ}_1+\mathcal{JJ}_2.\nonumber
\end{eqnarray} 
\par Finally,  the last equality in (\ref{Lyapunovcancellation}) is a consequence of the straightforward identities $2\mathcal{B}+\mathcal{BB}=2\mathcal{C}+\mathcal{CC}=0.$ 
$\Box$

{ It can be easily checked that the system has a scaling, namely we have:

\begin{lemma} Let $(Q,u,p)$ be a solution of (\ref{system}). Then letting 
\begin{equation}
u_\delta\stackrel{\rm def}{=}\delta u(\delta x,\delta^2 t),\,Q_\delta\stackrel{\rm def}{=} Q(\delta x,\delta^2 t), p_\delta(x,t)\stackrel{\rm def}{=}\delta^2 p(\delta x,\delta^2)
\end{equation} we have that $(Q_\delta,u_\delta,p_\delta)$ satisfy (\ref{system}) with $F(Q)=-aQ+b[Q^2-\frac{\textrm{tr}(Q^2)}{3}Id]-cQ\textrm{tr}(Q^2)$ replaced by $F_\delta(Q_\delta)=\delta^2\left[ -aQ_\delta+b[(Q_\delta)^2-\frac{\textrm{tr}(Q_\delta^2)}{3}-cQ_\delta\textrm{tr}(Q_\delta)^2\right]$. We note that, in dimension two, the space $\dot H^1(\RR^2)\times L^2(\RR^2)$ is invariant by the scaling. 
\end{lemma}

\par In the following we assume that there exists a smooth solution of (\ref{system}) and  obtain estimates on the behaviour of various norms. 

\begin{proposition} Let $(Q,u)$ be a smooth solution of \eqref{system} in dimension $d=2$ or $d=3$, with restriction \eqref{c+}, and smooth initial data $(\bar Q(x),\bar u(x))$, that decays fast enough at infinity so that we can integrate by parts in space (for any $t\ge 0$) without boundary terms. We assume that $|\xi|<\xi_0$ where $\xi_0$  is an explicitly computable constant, scale invariant, depending on $a,b,c,d,\Gamma,\nu,\lambda$.

 \par    For $(\bar Q,\bar u)\in H^1\times L^2_x$,we have
\begin{equation}
\|Q(t,\cdot)\|_{H^1}\le C_1+\bar C_1 e^{\bar C_1t}\|\bar Q\|_{H^1}, \forall t\ge 0
\label{apriorih1}
\end{equation} with $C_1,\bar C_1$ depending on $(a,b,c,d,\Gamma,L, \nu,\bar Q,\bar u)$. Moreover

\begin{equation}
\|u(t,\cdot)\|_{L^2_x}^2+\nu\int_0^t\|\nabla u\|_{L^2_x}^2\le C_1
\label{apriorih2}
\end{equation}

\label{prop:apriorismallstrain}
 \end{proposition}

\smallskip     {\bf Proof.} We denote:
\begin{equation}
X_{\alpha\beta}\stackrel{\rm def}{=}L\Delta Q_{\alpha\beta}-cQ_{\alpha\beta}\textrm{tr}(Q^2),\,\alpha,\beta=1,2,3
\end{equation}
\par Then equation (\ref{energydecay}) becomes:
\begin{equation}\label{energydecay+}
\begin{aligned}
	&\frac{d}{dt}E(t)+\nu\|\nabla u\|_{L^2_x}^2	+\Gamma\lambda L^2	\|\Delta Q\|_{L^2_x}^2+\Gamma\lambda c^2\|Q\|_{L^6}^6-2cL\Gamma\lambda\int_{\mathbb{R}^d}\Delta Q_{\alpha\beta}Q_{\alpha\beta}\textrm{tr}(Q^2)\,dx+a^2\Gamma\lambda\|Q\|_{L^2_x}^2+\\&+b^2\Gamma\lambda\int_{\mathbb{R}^d}\textrm{tr}\left(Q^2-\frac{\textrm{tr}(Q^2)}{d}\right)^2\,dx\le 2a\Gamma\lambda\underbrace{\int_{\mathbb{R}^d}\textrm{tr}(XQ)\,dx}_{\stackrel{\rm def}{=}\mathcal{I}}-
2b\Gamma\lambda\underbrace{\int_{\mathbb{R}^d}\textrm{tr}(XQ^2)\,dx}_{\stackrel{\rm def}{=}\mathcal{J}}+2ab\Gamma\lambda\int_{\mathbb{R}^d}\textrm{tr}(Q^3)\,dx	
\end{aligned}
\end{equation}
\par Integrating by parts we have:
\begin{eqnarray}
-2cL\Gamma\lambda\int_{\mathbb{R}^d}\Delta Q_{\alpha\beta}Q_{\alpha\beta}\textrm{tr}(Q^2)dx=2cL\Gamma\lambda\int_{\mathbb{R}^d}Q_{\alpha\beta,k}Q_{\alpha\beta,k}\textrm{tr}(Q^2)dx+2cL\Gamma\lambda\int_{\mathbb{R}^d}Q_{\alpha\beta,k}Q_{\alpha\beta}\partial_k\left(\textrm{tr}(Q^2)\right)dx\nonumber\\
=2cL\Gamma\lambda\int_{\mathbb{R}^d}|\nabla Q|^2\textrm{tr}(Q^2)\,dx+cL\Gamma\lambda\int_{\mathbb{R}^d}|\nabla\left(\textrm{tr}(Q^2)\right)|^2\,dx\ge 0
\label{eq:positivitycrosstermsc}
\end{eqnarray}
(where for the last inequality we used the assumption (\ref{c+}) and $L,\Gamma,\lambda>0$). One can easily see that 
\begin{equation}
\mathcal{I}=-\frac{L}{2}\|\nabla Q\|_{L^2_x}^2-c\|Q\|_{L^4}^4
\end{equation}
\par On the other hand, for any $\varepsilon>0$ and $\tilde C=\tilde C(\varepsilon,c)$ an explicitly computable constant, we have:
\begin{eqnarray}
\mathcal{J}=L\int_{\mathbb{R}^d}Q_{\alpha\beta,kk}Q_{\alpha\gamma}Q_{\gamma\beta}\,dx-c\int_{\mathbb{R}^d}\textrm{tr}(Q^2)\textrm{tr}(Q^3)\,dx\nonumber\\
\le-L\int_{\mathbb{R}^d}Q_{\alpha\beta,k}Q_{\alpha\gamma,k}Q_{\gamma\beta}\,dx-L\int_{\mathbb{R}^d}Q_{\alpha\beta,k}Q_{\alpha\gamma}Q_{\gamma\beta,k}+ \int_{\mathbb{R}^d}\textrm{tr}(Q^2)\left(\frac{\tilde C}\varepsilon\textrm{tr}(Q^2)+\varepsilon\textrm{tr}^2(Q^2)\right)\,dx\nonumber\\
\le L\varepsilon \int_{\mathbb{R}^d}|\nabla Q|^2\textrm{tr}(Q^2)\,dx+\frac{\tilde C}{\varepsilon}\|\nabla Q\|_{L^2_x}^2+\int_{\mathbb{R}^d}\textrm{tr}(Q^2)\left(\frac{\tilde C}\varepsilon\textrm{tr}(Q^2)+\varepsilon\textrm{tr}^2(Q^2)\right)\,dx\non
\end{eqnarray}
\par Using the last three relations in (\ref{energydecay+}) we obtain:
\begin{align*}
\frac{d}{dt}E(t)+\nu \|\nabla u\|_{L^2_x}^2+\Gamma\lambda L^2\|\Delta Q\|_{L^2_x}^2+c^2\Gamma\lambda\|Q\|_{L^6}^6+a^2\Gamma\lambda\|Q\|_{L^2_x}^2
+2cL\Gamma\lambda\int_{\mathbb{R}^d}|\nabla Q|^2\textrm{tr}(Q^2)\,dx\\+cL\Gamma\lambda\int_{\mathbb{R}^d}|\nabla\left(\textrm{tr}(Q^2)\right)|^2\,dx\le 2|a|\Gamma\lambda(\frac{L}{2}\|\nabla Q\|_{L^2_x}^2+c\|Q\|_{L^4}^4)+2|b|\Gamma\lambda L\varepsilon \int_{\mathbb{R}^d}|\nabla Q|^2\textrm{tr}(Q^2)\,dx\nonumber\\+2|b|\Gamma\lambda\frac{\tilde C}{\varepsilon}\|\nabla Q\|_{L^2_x}^2+2|b|\Gamma\lambda\int_{\mathbb{R}^d}\textrm{tr}(Q^2)\left(\frac{\tilde C}\varepsilon\textrm{tr}(Q^2)+\varepsilon\textrm{tr}^2(Q^2)\right)\,dx+2|ab|\Gamma\lambda(\varepsilon\|Q\|_{L^2_x}^2+ \frac{\tilde C}{\varepsilon}\|Q\|_{L^4}^4)	
\end{align*}
\par Taking $\varepsilon$ small enough we can absorb all the terms with an epsilon coefficient on the right into the left hand side, and  we are left with 
\begin{eqnarray}
\frac{d}{dt}E(t)+\nu \|\nabla u\|_{L^2_x}^2+\Gamma\lambda L^2\|\Delta Q\|_{L^2_x}^2+\Gamma\lambda c^2\|Q\|_{L^6}^6+\Gamma\lambda a^2\|Q\|_{L^2_x}^2\nonumber\\+2cL\Gamma\lambda\int_{\mathbb{R}^d}|\nabla Q|^2\textrm{tr}(Q^2)\,dx+cL\Gamma\lambda\int_{\mathbb{R}^d}|\nabla\left(\textrm{tr}(Q^2)\right)|^2\,dx
\le \bar C\left(\|\nabla Q\|_{L^2_x}^2+\|Q\|_{L^4}^4\right)
\label{lphalf1}
\end{eqnarray} with $\bar C=\bar C(a,b,c)$.

\par The last relation is not yet enough because the $Q$ terms  without derivatives in $E(t)$ are not summing to a positive number. However, let us note that, if $a>0$ we obtain the a-priori estimates  by using the inequality
$\trc(Q^3)\leq \frac 38 \trc(Q^2)+\trc(Q^2)^2$.  If $a\leq 0$ we have to estimate separately $\|Q\|_{L^2_x}$ and this ask for a smallness condition for $\xi$.

  We need to control in some sense low frequencies of $Q$. To this end   we multiply the first equation in (\ref{system}) by $Q$, take the trace, integrate over $\mathbb{R}^d$ and by parts and we obtain:

\begin{eqnarray}
\frac{1}{2} \frac{d}{dt}\int_{\mathbb{R}^d} |Q|^2(t,x)\,dx=\Gamma\Big(-L\int_{\mathbb{R}^d} |\nabla Q|^2\,dx-a\int_{\mathbb{R}^d}|Q(x)|^2\,dx+b\int_{\mathbb{R}^d}\textrm{tr}(Q^3)\,dx-c\int_{\mathbb{R}^d}|Q|^4\,dx\Big)\nonumber\\
+\underbrace{\int_{\mathbb{R}^d}\textrm{tr}(\Omega Q^2-Q\Omega Q)\,dx}_{\stackrel{\rm{def}}{=}\mathcal{I}}\nonumber\\
+\underbrace{\xi\int_{\mathbb{R}^d} D_{\alpha\gamma}(Q_{\gamma\beta}+\frac{\delta_{\gamma\beta}}{d})Q_{\alpha\beta}+(Q_{\alpha\gamma}+\frac{\delta_{\alpha\gamma}}{d})D_{\gamma\beta}Q_{\alpha\beta}-2(Q_{\alpha\beta}+\frac{\delta_{\alpha\beta}}{d})Q_{\alpha\beta}\textrm{tr}(Q\nabla u)\,dx}_{\stackrel{\rm{def}}{=}\mathcal{II}}\nonumber
\end{eqnarray}

      Recalling that $Q$ is symmetric we have $\mathcal{I}=0$. Also:

\begin{displaymath}
|\mathcal{II}|=|2\xi||\int_{\mathbb{R}^d} \frac{1}{d}D_{\alpha\beta}Q_{\alpha\beta}+D_{\alpha\gamma}Q_{\gamma\beta}Q_{\beta\alpha}-Q_{\alpha\beta}Q_{\alpha\beta}\textrm{tr}(Q\nabla u)\,dx|\le C(d)\int_{\mathbb{R}^d}\varepsilon |\nabla u|^2+\int_{\mathbb{R}^d}\frac{|\xi|^2}{\varepsilon}(|Q|^2+|Q|^6)\,dx
\end{displaymath} 

\par Thus we get:

\begin{equation}
   \frac{d}{dt} \int_{\mathbb{R}^d}|Q|^2\,dx\le C(d)\varepsilon\int_{\mathbb{R}^d} |\nabla u|^2\,dx+\frac{|\xi|^2}{\varepsilon}\int_{\mathbb{R}^d}|Q|^2+|Q|^6\,dx+\hat C\int_{\mathbb{R}^d}|Q|^2+|Q|^4\,dx
 \label{lphalf2}
\end{equation} with $\hat C=\hat C(a,b)>0$.
\par Let us observe now that there exists $M=M(a,b,c)$ large enough,  so that 
\begin{equation}
 \frac{M}{2} \textrm{tr}(Q^2)+\frac{c}{8}\textrm{tr}^2(Q^2)\le(M+\frac{a}{2})\textrm{tr}(Q^2)-\frac{b}{3}\textrm{tr}(Q^3)+\frac{c}{4}\textrm{tr}^2(Q^2)
 \label{eq:M}
 \end{equation} for any $Q\in S_0$.
 
 \par Multiplying the equation (\ref{lphalf2}) by $M$ and adding to (\ref{lphalf1}) we obtain:
 
 \begin{eqnarray}\frac{d}{dt}(E(t)+M\|Q\|_{L^2_x}^2)+\nu \|\nabla u\|_{L^2_x}^2+\Gamma\lambda L^2\|\Delta Q\|_{L^2_x}^2+\Gamma\lambda c^2\|Q\|_{L^6}^6+a^2\|Q\|_{L^2_x}^2\nonumber\\
 +2cL\Gamma\lambda\int_{\mathbb{R}^d}|\nabla Q|^2\textrm{tr}(Q^2)\,dx+cL\Gamma\lambda\int_{\mathbb{R}^d}|\nabla\left(\textrm{tr}(Q^2)\right)|^2\,dx\nonumber\\
\le \bar C\left(\|\nabla Q\|_{L^2_x}^2+\|Q\|_{L^4}^4\right)+MC(d)\varepsilon\int_{\mathbb{R}^d} |\nabla u|^2\,dx+\frac{M|\xi|^2}{\varepsilon}\int_{\mathbb{R}^d}|Q|^2+|Q|^6\,dx+M\hat C\int_{\mathbb{R}^d}|Q|^2+|Q|^4\,dx\label{estimate_control_H1}
\end{eqnarray}

\par We chose $\varepsilon$ small enough so that $MC(d)\varepsilon<\nu$. Finally we make the assumption that $|\xi|$ is small enough, depending on $a,b,c,d,\nu$ so that 

$$\frac{M|\xi|^2}{\varepsilon}\le \Gamma\lambda c^2$$

\par Then taking into account equation (\ref{eq:M}) we obtain the claimed relation (\ref{apriorih1}).$\Box$

}
 We note that the $\xi$ small hypothesis is necessary because we are in infinite domain, for example, in the periodic domain, we can add a constant to the functional and get the apriori $L^p$ estimates without any smallness condition on $\xi$.

\section{The existence of weak solutions}

The next proposition follows closely the similar result in \cite{pz2} where it was done for $\lambda=1$. The purpose for including it here is to provide an alternative approximation system thus correcting the proof in \cite{pz2} and also to show how the cancellations that appeared previously in the derivation of the energy law still survive at the approximate level but with some differences, phenomenon which will appear in a much more complex setting in the proof of uniqueness in the next section.

\begin{proposition} For $d=2,3$ there exists a weak solution $(Q,u)$ of the system \eqref{system} subject to initial conditions \eqref{initialdata}. The solution $(Q,u)$ is such that $Q\in L^\infty_{loc}(\mathbb{R}_+;H^1)\cap L^2_{loc}(\mathbb{R}_+;H^2)$ and $u\in L^\infty_{loc}(\mathbb{R}_+;L^2)\cap L^2_{loc}(\mathbb{R}_+;H^1)$.
\label{prop:weak}
\end{proposition} 

{
\smallskip     {\bf Proof.}  As  first step of the construction of weak solutions for the system (\ref{system}) we construct for any fixed $\varepsilon>0$ a global weak solution 
$$Q_\varepsilon\in L^\infty_{loc}(\mathbb{R}_+; H^1)\cap L^2_{loc}(\mathbb{R}_+;H^2) \quad u_\varepsilon\in L^\infty_{loc}(\mathbb{R}_+;L^2)\cap L^2_{loc}(\mathbb{R}_+; H^1)$$
for the modified system obtained by mollifying the coefficients of the equation for the $Q$ tensor  and by adding to the equation of the velocity a regularizing term. This term is needed in order to estimate some "bad" terms which does not disappear in an energy estimate. For the simplicity of the notations, we drop the indices $\varepsilon$ and we denote the solution $(Q_\varepsilon,u_\varepsilon)$ by $(Q,u)$.
\begin{equation}
\left\{\begin{array}{l}
         \partial_t Q+(R_\varepsilon u) \nabla  Q  -\Big(\big(R_\varepsilon(\xi   D+  \Omega)\big)(Q+\frac{1}{d}Id)\Big)\nonumber\\-\Big((Q+\frac{1}{d}Id)R_\varepsilon (\xi  D-  \Omega)\Big)\nonumber\\+2\xi  \Big((Q+\frac{1}{d}Id)\textrm{tr} \big(Q\nabla  R_\varepsilon u\big)\Big)=\Gamma  H\\
 \partial_t u+ (R_\varepsilon u)\nabla u-\nu\Delta u+\nabla p=  -\varepsilon \mathcal{P} R_\varepsilon\left(\sum_{l,m=1}^d \nabla Q_{lm}\left(R_\varepsilon  u\cdot\nabla Q_{lm}\right)|R_\varepsilon u\nabla Q|\right)\nonumber\\
         +\varepsilon\mathcal{P}\nabla\cdot R_\varepsilon\bigg( \nabla R_\varepsilon   u  |\nabla R_\varepsilon   u|^2 \bigg)  -\lambda\xi\nabla\cdot   R_\varepsilon\bigg(\left(Q+\frac{1}{d}Id\right)  H\bigg)-\xi\mathcal{P}\nabla\cdot R_\varepsilon \bigg(H\left(Q+\frac{1}{d}Id\right)\bigg) \nonumber\\
          +2\lambda\xi \nabla\cdot R_\varepsilon  \bigg(\big(Q+\frac{1}{d}Id\big)\big( Q H\big)\bigg)- 
          L\lambda R_\varepsilon(\nabla\cdot\textrm{tr}(\nabla  Q\nabla  Q))\nonumber\\+L\lambda\mathcal{P}\nabla\cdot  R_\varepsilon \left(Q \Delta Q-\Delta Q  Q\right)\\
          (Q,u)|_{t=0}=(R_\varepsilon \overline Q, R_\varepsilon \overline u).
     \end{array}\right.
\end{equation}
where $R_\varepsilon$ is the convolution operator with the kernel $\epsilon^{-d}\chi(\epsilon^{-1}\cdot)$. 

 In order to construct the global weak solution for this system, we use the classical Friedrich's scheme. We define the mollifying operator
\begin{equation*}
	\widehat{J_nf}(\xi)\stackrel{\rm def}{=}1_{\{2^{-n}\leq|\xi|\leq 2^n\}}\hat f(\xi).
\end{equation*}
\par We consider the approximating system:
\begin{equation}
\left\{\begin{array}{l}
         \partial_t Q^{(n)}+ J_n\Big(R_\varepsilon J_n u^{n} \nabla J_n Q^{(n)}\Big)  -J_n\Big((\xi  J_n  R_\varepsilon D^{(n)}+  J_n R_\varepsilon\Omega^{(n)})(J_n Q^{(n)}+\frac{1}{d}Id)\Big)\\-J_n\Big((J_nQ^{(n)}+\frac{1}{d}Id)(\xi   J_n R_\varepsilon D^{(n)}- J_n R_\varepsilon\Omega^{(n)})\Big)\\+2\xi  J_n\Big(( J_n Q^{(n)}+\frac{1}{d}Id)\textrm{tr} \big( J_n Q^{(n)}\nabla   J_n R_\varepsilon u^{(n)}\big)\Big)=\Gamma \tilde H^{(n)}\\
         \partial_t u^n +\mathcal{P}J_n( \mathcal{P}J_n R_\varepsilon u^n\nabla\mathcal{P}J_n u^n)-\nu\Delta \mathcal{P}J_nu^{(n)}=\\
          -\varepsilon\mathcal{P}J_n R_\varepsilon\bigg(\sum_{l,m=1}^d \nabla J_n Q^{(n)}_{lm}\left(R_\varepsilon J_n u^{n}\cdot\nabla J_nQ^{(n)}_{lm}\right)|R_\varepsilon J_n u^n\nabla J_n Q^{(n)}|\bigg)\\
          +\varepsilon\mathcal{P}\nabla\cdot J_n R_\varepsilon\bigg(  \nabla R_\varepsilon J_n u^{(n)} |\nabla R_\varepsilon  J_n u^{(n)}|^2 \bigg) \\
          -\lambda\xi\mathcal{P}\nabla\cdot J_n   \bigg(\left(  J_n Q^{(n)}+\frac{1}{d}Id\right) \tilde H^{(n)}\bigg)-\lambda\xi\mathcal{P}\nabla\cdot J_n   \bigg(  \tilde H^{(n)}\left( J_n Q^{(n)}+\frac{1}{d}Id\right)\bigg) \\+2\lambda\xi \mathcal{P}\nabla\cdot J_n   \bigg(\big(J_n Q^{(n)}+\frac{1}{d}Id\big)\big( J_nQ^{(n)} \tilde H^{(n)}\big)\bigg)- L\lambda\mathcal{P} J_n    (\nabla\cdot\textrm{tr}( J_n Q^{(n)}\nabla J_n Q^{(n)}))\\+L\lambda\mathcal{P}\nabla\cdot J_n   \left(J_n Q^{(n)} \Delta  J_n Q^{(n)}-\Delta  J_n Q^{(n)}  J_n Q^{(n)}\right)\\
     \end{array}\right.
      \label{approxsystem+}
\end{equation} where $\mathcal{P}$ denotes the Leray projector onto divergence-free vector fields, $M$ is a positive constant, and $ \tilde H^{(n)}\stackrel{def}{=}L\Delta   J_n Q^{(n)} -a J_n Q^{(n)}+bJ_n[(J_n Q^{(n)}  J_n Q^{(n)})-\frac{\textrm{tr}(J_n Q^{(n)} J_n Q^{(n)})}{d}Id]-cJ_n\bigg(J_n Q^{(n)}\big|J_n Q^{(n)}\big|^2\bigg)$. We take as initial data $(J_n R_\varepsilon \bar Q, J_n  R_\varepsilon \bar u)$.

      The system above can be regarded as an ordinary differential equation in $L^2$ verifying the conditions of the Cauchy-Lipschitz theorem. Thus it admits a unique maximal solution $(Q^{(n)}, u^{(n)})\in C^1([0,T_n); L^2(\mathbb{R}^d;\mathbb{R}^{d\times d})\times L^2(\mathbb{R}^d,\mathbb{R}^d))$.  
As we have $(\mathcal{P}J_n)^2=\mathcal{P}J_n$ and $J_n^2=J_n$ the pair $(J_n Q^{(n)}, \mathcal{P}J_n u^{(n)})$ is also a solution of (\ref{approxsystem+}). By uniqueness we have $(J_n Q^{(n)}, \mathcal{P}J_n u^{(n)})=(Q^{(n)}, u^{(n)})$ hence $(Q^{(n)}, u^{(n)})\in C^1([0,T_n),H^\infty)$ and $(Q^{(n)}, u^{(n)})$ satisfy the system:
\begin{equation}
\left\{\begin{array}{l}
         \partial_t Q^{(n)}+J_n \Big(   R_\varepsilon u^n \nabla  Q^{(n)}\Big)  -J_n\Big((\xi    R_\varepsilon D^{(n)}+ R_\varepsilon \Omega^{(n)})( Q^{(n)}+\frac{1}{d}Id)\Big)\\-J_n\Big((Q^{(n)}+\frac{1}{d}Id)(\xi   R_\varepsilon D^{(n)}-  R_\varepsilon\Omega^{(n)})\Big)+2\xi J_n \Big(( Q^{(n)}+\frac{1}{d}Id)\textrm{tr} \big( Q^{(n)}\nabla  R_\varepsilon u^n\big)\Big)=\Gamma  \bar H^{(n)}\\
         \partial_t u^n +\mathcal{P}J_n(  R_\varepsilon u^n\nabla u^n)-\nu\Delta u^{(n)}=\\
          -\varepsilon\mathcal{P}J_n\bigg(\sum_{l,m=1}^d \nabla Q^{(n)}_{lm}\left(R_\varepsilon u^{n}\cdot\nabla Q^{(n)}_{lm}\right)|R_\varepsilon u^n\nabla Q^{(n)}|\bigg)\\
           +\varepsilon\mathcal{P}\nabla\cdot J_n R_\varepsilon\bigg( \nabla R_\varepsilon   u^{(n)} | \nabla R_\varepsilon  u^{(n)}|^2 \bigg) \\
         -\lambda\xi\mathcal{P}\nabla\cdot J_n   \bigg(\left( Q^{(n)}+\frac{1}{d}Id\right) \bar H^{(n)}\bigg)-\lambda\xi\mathcal{P} \nabla\cdot J_n   \bigg( \bar H^{(n)}\left( Q^{(n)}+\frac{1}{d}Id\right)\bigg)\\ +2\lambda\xi \mathcal{P}\nabla\cdot J_n   \bigg(\big(Q^{(n)}+\frac{1}{d}Id\big)\big(Q^{(n)}\bar H^{(n)}\big)\bigg)\\- L\lambda\mathcal{P}J_n     (\nabla\cdot\textrm{tr}(\nabla  Q^{(n)}\nabla  Q^{(n)}))+L\lambda\mathcal{P}\nabla\cdot J_n   \left(Q^{(n)} \Delta  Q^{(n)}-\Delta Q^{(n)}  Q^{(n)}\right)\\
     \end{array}\right.
      \label{approxsystem++}
\end{equation}  where $\bar H^{(n)}\stackrel{def}{=}L\Delta Q^{(n)} -a Q^{(n)}+bJ_n [(Q^{(n)}Q^{(n)})-\frac{\textrm{tr}(J_n(Q^{(n)} Q^{(n)}))}{d}Id]-cJ_n(Q^{(n)}|Q^{(n)}|^2)$. The initial data is $(J_n \bar Q, J_n  \bar u)$. We recall now a few properties of $J_n$ :
\begin{lemma}  The operators $\mathcal{P} $ and $J_n$ are selfadjoint in $L^2$. Moreover $J_n$ and $\mathcal{P}J_n$ are  also idempotent and $J_n$ commutes with distributional derivatives.
\end{lemma}
     We proceed in a manner analogous to the proof of Proposition~\ref{prop:Lyapunov} and multiply the first equation in \eqref{approxsystem++} by $-\lambda \bar H^{(n)}$, take the trace, integrate over $\mathbb{R}^d$ and by parts, and add to the second equation multiplied by $u^{(n)}$. Let us observe that almost all the cancellations in the proof of (\ref{prop:Lyapunov}) hold, except for a few terms that need to be estimated separately. We also have some more new terms that we added in the regularization, terms that control the ones which do not cancel. Thus we have:
 \begin{eqnarray}\frac{d}{dt}\int_{\mathbb{R}^d}\frac{1}{2}|u^n|^2+\frac{L\lambda}{2}|\nabla Q^{(n)}|^2+\lambda\big(\frac{a}{2}|Q^{(n)}|^2-\frac{b}{3}\textrm{tr}(Q^{(n)})^3+\frac{c}{4}|Q^{(n)}|^4\big)\,dx\nonumber\\+\nu\int_{\mathbb{R}^d}|\nabla u^n|^2\,dx+\Gamma\lambda\int_{\mathbb{R}^d}\textrm{tr}\bigg[J_n\left(L\Delta Q^{(n)}-aQ^{(n)}+b[(Q^{(n)})^2-\frac{\textrm{tr}((Q^{(n)})^2)}{3}Id]-cQ^{(n)}\big|Q^{(n)}\big|^2\right)\bigg]^2\,dx\nonumber\\ 
 +\varepsilon\int_{\mathbb{R}^d} |R_\varepsilon u\nabla Q^{(n)}|^3\,dx+\varepsilon \int_{\mathbb{R}^d} |R_\varepsilon\nabla u^n|^4\,dx\nonumber\\ \le \lambda\int_{\mathbb{R}^d}J_n\left(  R_\varepsilon u^n\cdot\nabla Q_{\alpha\beta}^{(n)}\right)J_n\left(bQ_{\alpha\gamma}^{(n)}Q_{\gamma\beta}^{(n)}-cQ_{\alpha\beta}^{(n)}\big|Q^{(n)}\big|^2\right)\,dx\nonumber\\
 +\lambda \int_{\mathbb{R}^d}J_n\left(-R_\varepsilon\Omega_{\alpha\gamma}^{(n)} Q_{\gamma\beta}^{(n)}+Q_{\alpha\gamma}^{(n)}R_\varepsilon\Omega_{\gamma\beta}^{(n)}\right)J_n\left(bQ_{\alpha\delta}^{(n)}Q_{\delta\beta}^{(n)}-cQ_{\alpha\beta}^{(n)}\big| Q^{(n)}\big|^2\right)\,dx
\label{est:energyapprox}
  \end{eqnarray} hence :  
  \begin{align*}
  \frac{d}{dt}\int_{\mathbb{R}^d}\frac{1}{2}|u^n|^2+\frac{L\lambda}{2}|\nabla Q^{(n)}|^2+\lambda\big(\frac{a}{2}|Q^{(n)}|^2-\frac{b}{3}\textrm{tr}(Q^{(n)})^3+\frac{c}{4}|Q^{(n)}|^4\big)\,dx+\nu\int_{\mathbb{R}^d}|\nabla u^n|^2\,dx\\
  +\Gamma\lambda\int_{\mathbb{R}^d}L^2|\Delta Q^{(n)}|^2\,dx+\Gamma\lambda a^2\int_{\mathbb{R}^d}|Q^{(n)}|^2\,dx+C(b^2,d,\Gamma,\lambda)\int_{\mathbb{R}^d} |Q^{(n)}|^4\,dx+\Gamma \lambda c^2\int_{\mathbb{R}^d}\big|J_n(Q^{(n)}|Q^{(n)}|^2)\big|^2\,dx\\
+\varepsilon\int_{\mathbb{R}^d} |R_\varepsilon u\nabla Q^{(n)}|^3\,dx+\varepsilon \int_{\mathbb{R}^d} |R_\varepsilon\nabla u^n|^4\,dx\\
  \le  \underbrace{2\Gamma\lambda c\int_{\mathbb{R}^d} L\Delta Q^{(n)}\cdot Q^{(n)}|Q^{(n)}|^2\,dx}_{\stackrel{\rm def}{=}\mathcal{I}}-2\Gamma\lambda \int_{\mathbb{R}^d}L\Delta Q^{(n)}\cdot\big( -a Q^{(n)}+bJ_n\left[(Q^{(n)})^2-\frac{\textrm{tr}(Q^{(n)})^2}{d}Id)\right]\big)\,dx\\
   -2\Gamma\lambda\int_{\mathbb{R}^d} cQ^{(n)}|Q^{(n)}|^2\cdot \left(aQ^{(n)}-bJ_n\left[(Q^{(n)})^2-\frac{\textrm{tr}(Q^{(n)})^2}{d}Id\right]\right)\,dx\\
  +\lambda\underbrace{\int_{\mathbb{R}^d}J_n\left(   R_\varepsilon u^n\cdot\nabla Q_{\alpha\beta}^{(n)}\right)J_n\left(bQ_{\alpha\gamma}^{(n)}Q_{\gamma\beta}^{(n)}-cQ_{\alpha\beta}^{(n)}\big|Q^{(n)}\big|^2\right)\,dx}_{\stackrel{\rm def}{=}\mathcal{II}}\\ 
  +C\int_{\mathbb{R}^d}|R_\varepsilon \nabla u^n|^2|Q^{(n)}|^2\,dx+\frac{\Gamma c^2}{8}\int_{\mathbb{R}^d} |J_n(Q^{(n)}|Q^{(n)}|^2)|^2\,dx +C\int_{\mathbb{R}^d}|Q^{(n)}|^4\,dx
    \end{align*}   
\par We have that 
\begin{equation}\label{est_II}  
   \begin{aligned}
  \mathcal{II}=\int_{\mathbb{R}^d} &\left(R_\varepsilon u^n\cdot\nabla Q_{\alpha\beta}^{(n)}\right)J_n\left(bQ_{\alpha\gamma}^{(n)}Q_{\gamma\beta}^{(n)}-cQ_{\alpha\beta}^{(n)}\big|Q^{(n)}\big|^2\right)\,dx\\
  &\le \left(\frac{4}{\Gamma c^2}+\frac{1}{4C(b^2,d,\Gamma)}\right)\int_{\mathbb{R}^d} |R_\varepsilon u^n\cdot\nabla Q^{(n)}|^2\,dx\\
  &+\frac{C(b^2,d,\Gamma)}{2}\|Q^{(n)}\|_{L^4}^4+\frac{\Gamma c^2}{8}\int_{\mathbb{R}^d} |J_n(Q^{(n)}|Q^{(n)}|^2)|^2\,dx\\
  &\le\frac{\varepsilon}{2}\int_{\mathbb{R}^d}  |R_\varepsilon u^n\cdot\nabla Q^{(n)}|^3+C(\varepsilon,b^2,c^2,d,\Gamma)\int_{\mathbb{R}^d} \sum_{l,m=1}^d |R_\varepsilon u^n\cdot\nabla Q^{(n)}_{lm}|\,dx\\ 
  & +\frac{C(b^2,c^2,d,\Gamma)}{2}\|Q^{(n)}\|_{L^4}^4+\frac{\Gamma c^2}{8}\int_{\mathbb{R}^d} |J_n(Q^{(n)}|Q^{(n)}|^2)|^2\,dx\\
   &\le\frac{\varepsilon}{2}\int_{\mathbb{R}^d}  |R_\varepsilon u^n\cdot\nabla Q^{(n)}|^3\,dx+C_1(\varepsilon,b^2,c^2,d,\Gamma)\int_{\mathbb{R}^d} |u^n|^2\,dx+C_2(\varepsilon,b,c,d^2,\Gamma)\int_{\mathbb{R}^d} |\nabla Q^{(n)}|^2\,dx\\ 
  & +\frac{C(b^2,d,\Gamma)}{2}\|Q^{(n)}\|_{L^4}^4+\frac{\Gamma c^2}{8}\int_{\mathbb{R}^d} |J_n(Q^{(n)}|Q^{(n)}|^2)|^2\,dx
  \end{aligned}
 \end{equation} 
\par Using the fact that $\mathcal{I}\le 0$ and the estimate for $\mathcal{II}$ shown before, we replace in (\ref{est:energyapprox}) and obtain:  
\begin{align*}
  \frac{d}{dt}\int_{\mathbb{R}^d}\frac{1}{2}|u^n|^2+\frac{L\lambda}{2}|\nabla Q^{(n)}|^2+\lambda\big(\frac{a}{2}|Q^{(n)}|^2-\frac{b}{3}\textrm{tr}(Q^{(n)})^3+\frac{c}{4}|Q^{(n)}|^4\big)\,dx\\
  +\nu\int_{\mathbb{R}^d}|\nabla u^n|^2\,dx+\Gamma\lambda\int_{\mathbb{R}^d}L^2|\Delta Q^{(n)}|^2\,dx
  + \frac{\varepsilon}{2}\int_{\mathbb{R}^d}  |R_\varepsilon u^n\cdot\nabla Q^{(n)}|^3\,dx+\frac{\varepsilon}{2}\int_{\mathbb{R}^d} |\nabla R_\varepsilon u^n|^4\,dx\\
  \le \int_{\mathbb{R}^d} |Q^{(n)}|^2+|Q^{(n)}|^4\,dx+C\int_{\mathbb{R}^d} |\nabla Q^{(n)}|^2\,dx+C(\varepsilon)\int_{\mathbb{R}^d}|u^{(n)}|^2\,dx
 \end{align*}     
 \par This  estimate does not readily provide bounds on $Q^{(n)}$ because the term $\frac{a}{2}|Q^{(n)}|^2-\frac{b}{3}\textrm{tr}(Q^{(n)})^3+\frac{c}{4}|Q^{(n)}|^4$ could be negative. In order to obtain $H^1$ estimates we proceed as in the proof of Proposition~\ref{prop:apriorismallstrain}.  We put the proof in the appendix by Proposition \ref{prop:apriorismallstrainn}. We can continue to proceed as in the proof of Proposition \ref{prop:apriorismallstrain} and in fact in this case because of the first two regularizing terms on the right hand side of the $u^n$ equation in (\ref{approxsystem++}) we do not need the $\xi$ small assumption.
    These estimates allow us to conclude that $T_n=\infty$ and we also get the following apriori bounds:
\begin{equation}\label{weaksolapriori}
\begin{aligned}
\sup_{n}\|\nabla R_\varepsilon u^{n}\|_{L^4(0,T;L^4)}, \sup_n  \|R_\varepsilon u^n\cdot\nabla Q^{(n)}\|_{L^3(0,T;L^3)},dx\le C(\varepsilon)\\
\sup_n \|Q^{(n)}\|_{L^2(0,T;H^2)\cap L^\infty(0,T;H^1)}<\infty,\\
\sup_n \|u^n\|_{L^\infty(0,T;L^2)\cap L^2(0,T;H^1)}<\infty,
\end{aligned}
\end{equation} for any $T<\infty$. By the bounds which can be obtained by using the equation on $\partial_t (Q^{(n)}, u^n)$ in some $L^\infty_{loc}(H^{-N})$ for large enough $N$, we get, by classical local compactness Aubin-Lions lemma, on a subsequence, that:

\begin{eqnarray}
Q^{(n)}\rightharpoonup Q\textrm{ in } L^2(0,T;H^2)\,\textrm{ and }Q^{(n)}\to Q\textrm{ in }L^2(0,T;H_{loc}^{2-\delta}),\forall \delta>0\nonumber\\
Q^{(n)}(t)\rightharpoonup Q(t)\textrm{ in }H^1\textrm{ for all }t\in\mathbb{R}_+\nonumber\\
u^n\rightharpoonup u\textrm{ in }L^2(0,T;H^1)\,\textrm{ and }u^n\to u\textrm{ in }L^2(0,T;H_{loc}^{1-\delta}),\forall\delta>0\nonumber\\
u^n(t)\rightharpoonup u(t)\,\textrm{ in }L^2\textrm{ for all }t\in\mathbb{R}_+\nonumber
\end{eqnarray}
Thus we can pass to the limit and obtain a weak solution of the approximating system:
\begin{equation}
\left\{\begin{array}{l}
         \partial_t Q^{(\varepsilon)}+  R_\varepsilon u^\varepsilon \nabla  Q^{(\varepsilon)}  -(\xi    R_\varepsilon D^\varepsilon+  R_\varepsilon \Omega^\varepsilon)( Q^{(\varepsilon)}+\frac{1}{d}Id)\Big)+\Big((Q^{(\varepsilon)}+\frac{1}{d}Id)(\xi  R_\varepsilon  D^\varepsilon-  R_\varepsilon \Omega^\varepsilon)\Big)\\-2\xi \Big(( Q^{(\varepsilon)}+\frac{1}{d}Id)\textrm{tr} \big( Q^{(\varepsilon)}\nabla   u^\varepsilon\big)\Big)=\Gamma H^\varepsilon\\
         \partial_t u^\varepsilon+\mathcal{P}  R_\varepsilon u^\varepsilon\nabla u^\varepsilon= -\varepsilon \mathcal{P} R_\varepsilon\left(\sum_{l,m=1}^d \nabla Q_{lm}\left(R_\varepsilon  u\cdot\nabla Q_{lm}\right)|R_\varepsilon u\nabla Q|\right)\\+\varepsilon\mathcal{P}\nabla\cdot R_\varepsilon\bigg( R_\varepsilon \nabla  u  |R_\varepsilon \nabla  u|^2 \bigg) 
         -\lambda\xi\mathcal{P}\nabla\cdot R_\varepsilon \bigg(\left( Q^{(\varepsilon)}+\frac{1}{d}Id\right)H^\varepsilon\bigg)\\-\lambda\xi\mathcal{P} \nabla\cdot R_\varepsilon \bigg(H^\varepsilon\left( Q^{(\varepsilon)}+\frac{1}{d}Id\right)\bigg)+2\lambda\xi \mathcal{P}\nabla\cdot R_\varepsilon \bigg(\big(Q^{(\varepsilon)}+\frac{1}{d}\big)\big(Q^{(\varepsilon)} H^\varepsilon\big)\bigg)\\- L\lambda\mathcal{P}  (\nabla\cdot R_\varepsilon \textrm{tr}(\nabla  Q^{(\varepsilon)}\odot\nabla  Q^{(\varepsilon)}))+L\lambda\mathcal{P}\nabla\cdot   R_\varepsilon\left(Q^{(\varepsilon)}\Delta  Q^{(\varepsilon)}-\Delta Q^{(\varepsilon)} Q^{(\varepsilon)}\right)+\nu\Delta u^\varepsilon
     \end{array}\right.
\label{approxsystem3+}
\end{equation}  where we recall that  $H=L\Delta Q^{(\varepsilon)} -a Q^{(\varepsilon)}+b[(Q^{(\varepsilon)})^2-\frac{\textrm{tr}((Q^{(\varepsilon)})^2))}{d}Id]-cQ^{(\varepsilon)}\textrm{tr}((Q^{(\varepsilon)})^2)$. The initial data for the limit system is $( R_\varepsilon \bar Q,  R_\varepsilon \bar u)$. 
\par One can easily see that the solutions of (\ref{approxsystem3+}) are smooth, first by obtaining $C^\infty$ regularity for the first $Q$ equations, by bootstrapping the regularity improvement provided by the linear heat equation, and then the regularity for the $u$ equation, by bootstrapping the regularity improvement provided by a linear advection equation.  For this system we can proceed as in the case of apriori estimates and obtain the same estimates, independent of $\varepsilon$ because the solutions are smooth and all the cancellations that were used in the apriori estimates also hold here.  In particular we obtain:
\begin{eqnarray}
\sup_\varepsilon \|Q^{(\varepsilon)}\|_{ L^\infty(0,T;H^1)\cap L^2(0,T;H^2)}<\infty,\nonumber\\
\sup_\varepsilon \|u^\varepsilon\|_{L^\infty(0,T;L^2)\cap L^2(0,T;H^1)}<\infty
\label{weaksolapriori++}
\end{eqnarray} for any $T<\infty$. Taking into account those bounds  and also the bounds which can be obtained by using the equation on $\partial_t (Q^\varepsilon, u^\varepsilon)$ in some $L^p_{loc}(H^{-N})$ for large enough $N$, we get, by classical local compactness Aubin-Lions lemma and by weak convergence arguments, that there exists a $Q\in L^\infty_{loc}(\mathbb{R}_+;H^1)\cap L^2_{loc}(\mathbb{R}_+;H^2)$ and a $u\in L^\infty_{loc}(\mathbb{R}_+;L^2)\cap L^2_{loc}(\mathbb{R}_+;H^1)$ so that, on a subsequence, we have: 
\begin{eqnarray}
Q^{(\varepsilon)}\rightharpoonup Q\textrm{ in } L^2(0,T;H^2)\,\textrm{ and }Q^{(n)}\to Q\textrm{ in }L^2(0,T;H_{loc}^{2-\delta}),\forall \delta>0\nonumber\\
Q^{(\varepsilon)}(t)\rightharpoonup Q(t)\textrm{ in }H^1\textrm{ for all }t\in\mathbb{R}_+\nonumber\\
u^\varepsilon\rightharpoonup u\textrm{ in }L^2(0,T;H^1)\,\textrm{ and }u^n\to u\textrm{ in }L^2(0,T;H_{loc}^{1-\delta}),\forall\delta>0\nonumber\\
u^\varepsilon(t)\rightharpoonup u(t)\,\textrm{ in }L^2\textrm{ for all }t\in\mathbb{R}_+
\label{convergences}
\end{eqnarray}
 These  convergences allow us to the pass to the limit in the weak solutions of the system
(\ref{approxsystem3+}) to obtain a weak solution of (\ref{system}), namely  (\ref{weaksol1}),(\ref{weaksol2}). Of all the terms there is only one type that is slightly  difficult to treat in passing to the limit, namely:
$$L\int_0^\infty\int_{\mathbb{R}^d} \partial_\beta\left(Q^{(\varepsilon)}_{\alpha\gamma}\Delta Q^{(\varepsilon)}_{\gamma\beta}-\Delta Q^{(\varepsilon)}_{\alpha\gamma}Q^{(\varepsilon)}_{\gamma\beta}\right)\psi_{\alpha}\,\,dx\,dt=-L\int_0^\infty\int_{\mathbb{R}^d}\left(Q^{(\varepsilon)}_{\alpha\gamma}\Delta Q^{(\varepsilon)}_{\gamma\beta}-\Delta Q^{(\varepsilon)}_{\alpha\gamma}Q^{(\varepsilon)}_{\gamma\beta}\right)\cdot  \psi_{\alpha,\beta}\,\,dx\,dt.$$
Taking into account that $\psi$ is compactly supported and the convergences (\ref{convergences}) one can easily pass to the limit the terms $     \psi_{\alpha,\beta} Q^{(\varepsilon)}_{\alpha\gamma}$ and $     \psi_{\alpha,\beta} Q^{(\varepsilon)}_{\gamma\beta}$ strongly in $L^2(0,T;L^2)$.  
       Relations (\ref{convergences}) give that   $\Delta Q^{(\varepsilon)}_{\gamma\beta}$, $\Delta Q^{(\varepsilon)}_{\alpha\gamma}$ converges weakly in $L^2(0,T;L^2)$. Thus  we get convergence to the limit term 

 \begin{eqnarray}
 L\int_0^\infty \int_{\mathbb{R}^d} \partial_\beta( Q_{\alpha\gamma}\Delta Q_{\gamma\beta}) \psi_\alpha dxdt-L\int_0^\infty\int_{\mathbb{R}^d}\partial_\beta (\Delta Q_{\alpha\gamma})Q_{\gamma\beta})\psi_\alpha dxdt\nonumber\\
=-L\int_0^T \int_{\mathbb{R}^d} (\Delta Q_{\gamma\beta})(   \partial_\beta \psi_\alpha Q_{\alpha\gamma}) dxdt+L\int_0^T\int_{\mathbb{R}^d}(\Delta Q_{\alpha\gamma})(   \partial_\beta\psi_\alpha Q_{\gamma\beta})dxdt.
 \end{eqnarray}}
  Using also the uniform bound of $\varepsilon\|R_\varepsilon u^\varepsilon\nabla Q^\varepsilon\|^3_{L^3}$ it is easy to check that $\varepsilon\int |R_\varepsilon u^\varepsilon\nabla Q^\varepsilon |^2 \nabla Q^\varepsilon \cdot R_\varepsilon \mathcal{P}\varphi dx dt $ converges to zero.  A similar observation holds for the $\varepsilon$-regularisation term $\varepsilon\mathcal{P}\nabla\cdot \bigg( R_\varepsilon \nabla  u  |R_\varepsilon \nabla  u|^2$\bigg).
  $\Box$
  
  \bigskip

\section{The uniqueness of weak solutions}

We start with a number of technical tools that are crucial for our proof.

\subsection{Littlewood-Paley theory}
We  define
$\mathcal{C}$ to be the ring of center
$0$, of small radius $1/2$ and great radius $2$. There exist two
nonnegative  radial
functions $\chi$ and $\varphi$ belonging respectively to~${\mathcal{D}} 
(B(0,1)) $ and to
${\mathcal{D}} (\mathcal{C}) $ so that
\begin{equation}
\label{lpfond1}
\chi(\xi) + \sum_{q\geq 0} \varphi (2^{-q}\xi) = 1,\forall \xi\in\mathbb{R}^d
\end{equation}
\begin{equation}
\label{lpfond2}
|p-q|\geq 5
\Rightarrow
{\rm Supp}\,\, \varphi(2^{-q}\cdot)\cap {\rm Supp}\,\, \varphi(2^{-p}\cdot)=\emptyset.
\end{equation}
For instance, one can take $\chi \in \mathcal{D} (B(0,1))$ such that $
\chi  \equiv 1 $ on $B(0,1/2)$ and take
$$
\varphi(\xi) = \chi(\xi/2) -
\chi(\xi).
$$
Then, we are able to define the Littlewood-Paley decomposition. Let us denote
by~$\mathcal{F}$ the Fourier transform on~$\mathbb{R}^d$. Let
$h,\
\tilde h,\  \Dd_q, \Sd_q$ ($q \in \mathbb{Z}$) be defined as follows:
$$\displaylines{
\label{defnotationdyadique}h = {\mathcal F}^{-1}\varphi\quad {\rm and}\quad \tilde h =
{\mathcal{F}}^{-1}\chi, \cr
\Dd_q u = \mathcal{F}^{-1}(\varphi(2^{-q}\xi)\mathcal{F} u) = 2^{qd}\int h(2^qy)u(x-y)dy,\cr
\Sd_qu
=\mathcal{F}^{-1}(\chi(2^{-q}\xi)\mathcal{F} u) =2^{qd} \int \tilde h(2^qy)u(x-y)dy.\cr
}
$$
We recall that for two appropriately smooth functions $a$ and $b$ we have the Bony's paraproduct decomposition \cite{Bony81}:
\begin{equation*}
	ab=\dot T_a b+\dot T_b a+\dot R(a,b)
\end{equation*} 
where 
\begin{equation*}
	\dot T_a b=\sum_{q\in\ZZ}\Sd_{q-1} a	\Dd_q b,\quad 
	\dot T_b a=\sum_{q\in\ZZ}\Sd_{q-1} b	\Dd_q a,\quad
	\text{and}\quad
	\dot R(a,b)=\sum_{\substack{q\in\ZZ,\\ i\in\{0,\pm 1\}} }\Dd_{q} a\Dd_{q+i} b.
\end{equation*}
Then we have $$\Dd_q(ab)=\Dd_q \dot T_a b+\Dd_q \dot T_b a+\Dd_q\dot R(a,b)=\Dd_q \dot T_a b+\Dd_q { \tilde R}(a,b),$$ where 
${\tilde{R}}(a,b)=\dot T_b a+\dot R(a,b)=\sum_{q\in\ZZ} \Sd_{q+2}b\Dd_{q}a$. Moreover:
\begin{equation}\label{bonydecomp}
\begin{aligned}
	\Dd_q (ab)	&=	 \sum_{|q'-q|\le 5}\Dd_q (\Sd_{q'-1}a\Dd_{q'}b)	+	\sum_{q'> q-5}\Dd_q(\Sd_{q'+2}b\Dd_{q'}a)							\\
				&=	 \sum_{|q'-q|\le 5}[\Dd_q,\,\Sd_{q'-1}a]\Dd_{q'}b	+	\sum_{|q'-q|\le 5}\Sd_{q'-1}a\Dd_q\Dd_{q'}b
					+\sum_{q'> q-5}\Dd_q(S_{q'+2}b\Dd_{q'}a)																					\\
				&=	 \sum_{|q'-q|\le 5}[\Dd_{q},\,\Sd_{q'-1} a]\Dd_{q'}b+	\sum_{|q'-q|\le 5}(\Sd_{q'-1} a-\Sd_{q-1}a)\Dd_q\Dd_{q'}b				\\
				&\hspace{4.6cm}
					+\sum_{q'> q-5}\Dd_q (\Sd_{q'+2} b\Dd_{q'}a)+
	\underbrace{		
					 \sum_{|q'-q|\le 5}\Sd_{q-1}a\Dd_q\Dd_{q'}b
	}_{=\Sd_{q-1}a\Dd_q b}
\end{aligned}
\end{equation}
In terms of this decomposition we can express the Sobolev norm  of an element $u$ in the  (nonhomogeneous!) space $H^s$ as:

$$
\|u\|_{H^s}=\big(\|\Sd_0 u\|_{L^2}^2+\sum_{q\in\mathbb{N}}2^{2qs}\|\Dd_q u\|_{L^2}^2\big)^{1/2}
$$

These are a particular case of the general nonhomogeneous Besov spaces $B^s_{p,r}$, for $s\in \R, p,r\in [1,\infty]^2$ consisting of all tempered distributions $u$ such that:

$$\|u\|_{B^s_{p,r}}\stackrel{def}{=}\left\{\begin{array}{ll}\| (\|\Sd_0 u\|_{L^p}^r+\sum_{q\in\mathbb{N}} 2^{rqs}\|\Dd_q u\|_{L^p}^r)^{\frac{1}{r}} &\textrm{ if }r<\infty\\
\max(\|\Sd_0 u\|_{L^p}, \sup_{q\in \mathbb{N}} 2^{qs} \|\Dd_q u\|_{L^p}) &\textrm{ if } r=\infty\end{array}\right.$$ which reduces to the nonhomogeneous Sobolev space $H^s$ for $p=r=2$.

Similarly we also have the norm of the {\it homogenous} Sobolev spaces $\dot H^s$:
\begin{equation*}
	\|u\|_{\dot H^s}=\big(\sum_{q\in\mathbb{Z}}2^{2qs}\|\Dd_q u\|_{L^2}^2\big)^{1/2}
\end{equation*}
and the homogenous Besov spaces $\dot B^s_{p,r}$ for $s\in \R, p,r\in [1,\infty]^2$ 
consisting of all the homogeneous tempered distributions $u$ such that:
\begin{equation*}
	\|u\|_{\dot B^s_{p,r}}\stackrel{def}{=}\left\{\begin{array}{ll}\| (\sum_{q\in\mathbb{Z}} 2^{rqs}\|\Dd_q u\|_{L^p}^r)^{\frac{1}{r}} &\textrm{ if }r<\infty\\
 \sup_{q\in \mathbb{Z}} 2^{qs} \|\Dd_q u\|_{L^p} &\textrm{ if } r=\infty\end{array}\right.
\end{equation*}
which reduces to the homogeneous Sobolev space $\dot H^s$ for $p=r=2$.

Let us note that the homogeneous Besov spaces have somewhat better product rules, and this specificity encoded in Theorem~\ref{theorem_product_homogeneous_sobolev_spaces} will be very useful in our subsequent estimates.

Furthermore we will need the following characterisation of the homogeneous norms, in terms of  operators $\Sd_q u$:
\begin{lemma} {\bf [ Prop. 2.33] ,\cite{MR2768550}} 
	Let $s<0$ and $p,r \in [1,\infty]^2$. A tempered distribution  $u$ belongs to $\dot B^s_{p,r}$ if and only if:
	\begin{equation*}
		(2^{qs} \|\dot S_q u\|_{L^p})_{q\in\mathbb{Z}}\in l^r
	\end{equation*}
	 and for some constant $C$ depending only on the dimension $d$ we have:
	\begin{equation*}
		C^{-|s|+1}\| u\|_{\dot B^s_{p,r}}\le \| (2^{qs}\|\dot S_q u\|_{L^p})_q\|_{l^r}\le C(1+\frac{1}{|s|})\|u\|_{\dot B^s_{p,r}}
	\end{equation*}
\end{lemma}
We will use the following well-known estimates:
\begin{lemma}\label{lemma:bernstein&commutator} (\cite{chemin},\cite{chemin&masmoudi})
      
     {\bf (i)} (Bernstein inequalities) $$2^{-q}\|\nabla \Sd_q u\|_{L^p_x}\le C \|u\|_{L^p_x}, \forall 1\le p\le \infty$$
			$$c\| \Dd_q u\|_{L^p_x}\le 2^{-q}\|\Dd_q \nabla u\|_{L^p_x}\le C\|\Dd_q u\|_{L^p_x}, \forall 1\le p\le \infty$$

    	{\bf (ii)} (Bernstein inequalities) $$\|\Dd_q u\|_{L^b_x}\le 2^{d(\frac{1}{a}-\frac{1}{b})q}\|\Dd_q u\|_{L^a_x},\textrm{ for } b\ge a\ge 1$$
			$$\|\Sd_q u\|_{L^b_x}\le 2^{d(\frac{1}{a}-\frac{1}{b})q}\|\Sd_q u\|_{L^a_x},\textrm{ for } b\ge a\ge 1$$

    {\bf (ii)} (commutator estimate) 
	\begin{equation}\label{commutator}
		\|[\Dd_q, u]v\|_{L^p_x}\leq C 2^{-q}\|\nabla u\|_{L^r_x}\|v\|_{L^s_x}
	\end{equation}  
	with $\frac{1}{p}=\frac{1}{r}+\frac{1}{s}$. The constant $C$ depends only on the function $\varphi$ used in defining $\Dd_q$ but not on $p,r,s$.
\end{lemma}
\begin{proof}
	For the commutator estimate we begin by writing 
	\begin{align*}
		[\Dd_q , u]v(x)	&=\Dd_q(uv)(x)-u(x)\Delta_q v(x)
						 = 2^{qd}	\int 												h(2^q y)		(u(x-y)-u(x))				v(x-y) \dd y						\\
						&=2^{qd}	\int_{\mathbb{R}^d}\int_0^1
											\frac{\partial}{\partial \tau}\Big\{			h(2^q y)		u(x-\tau y)					v(x-y) \dd y		\Big\}\dd\tau	\\
						&=-2^{qd}	\int_0^1\int_{\mathbb{R}^d} 							h(2^q y) 	y\cdot \nabla u(x-\tau y) 	v(x-y) \dd y\dd\tau				\\
						&=-2^{-q}	\int_0^1\int_{\mathbb{R}^d}		\tilde 				h_{2^q}(y)\cdot	 \nabla	u(x-\tau y) 		v(x-y) \dd y\dd\tau,
	\end{align*}
	where $\tilde h(y):=y h(y)\in \mathcal S(R^d)^d$ and $\tilde h_{\lambda}(y):=\lambda^{d}\tilde h(\lambda y)$. 
	Using the Cauchy-Schwartz inequality and a change of variables, we get
	\begin{equation*}
	\begin{aligned}
		|[\Dd_q , u]v(x)|
		&\leq 
			2^{-q}\int_0^1
					\bigg(
						\int_{\mathbb{R}^d} 
							|\tilde h_{2^q}(y)||\nabla u(x-\tau y)|^{\frac{r}{p}}\dd y
					\bigg)^{\frac pr}\dd\tau 
					\bigg(
						\int_{\mathbb{R}^d}
							|\tilde h_{2^q}(y)||v(x-y)|^{\frac sp}\dd y
					\bigg)^{\frac ps}\\
		&=
			2^{-q}\int_0^1 
					\bigg(
						\int_{\mathbb{R}^d} 
							\frac{|\tilde h_{2^q\tau^{-1}}(y)|}{\tau^d}|\nabla u(x- y)|^{\frac{r}{p}}\dd y
					\bigg)^{\frac pr}\dd\tau 
					\bigg(
						\int_{\mathbb{R}^d}
							|\tilde h_{2^q}(y)||v(x-y)|^{\frac sp}\dd y
					\bigg)^{\frac ps}\\								
		&=
			2^{-q}\int_0^1
				\Big(	\frac{|\tilde h_{2^q\tau^{-1}}|}{\tau^d}* |\nabla u|^{\frac{r}{p}}(x) 	\Big)^{\frac pr} \dd \tau\,
				\Big(	|\tilde h_{2^q}|*|v|^{\frac sp} 	(x)										\Big)^{\frac ps}
\end{aligned}
\end{equation*}
Taking the $L^p$ norm in the $x$ variable, using the Cauchy-Schwartz inequality in the $x$ variable and  convolution estimates we obtain
\begin{align*}
		\|[\Dd_q , u] v\|_{L^p_x}
		&\leq  2^{-q}\int_0^1
				\Big\| 
					\Big(	\frac{|\tilde h_{2^q\tau^{-1}}|}{\tau^d}* |\nabla u|^{\frac{r}{p}}(x) 	\Big)^{\frac pr} 
				\Big\|_{L^r_x}\dd \tau\,
				\Big\|
					\Big(	|\tilde h_{2^q}|*|v|^{\frac sp} 	(x)										\Big)^{\frac ps}
				\Big\|_{L^s_x}
		\\
		&\leq 2^{-q}
			\bigg(
				\int_0^1 \|\frac{ |\tilde h_{2^q\tau^{-1}}|}{\tau^d}* |\nabla u|^{\frac rp}\|_{L^p_x}^{\frac pr}\dd\tau
			\bigg) 
			\|
				|\tilde h_{2^q}|* |v|^{\frac sp}
			\|_{L^p_x}^{\frac ps}																									\\
		&
		\leq  2^{-q}\int_0^1 \frac{ \|\tilde h_{2^q\tau^{-1}}\|_{L^1_x}^{\frac{p}{r}}}{\tau^d} \dd \tau \|\nabla u \|_{L^r_x}
				\|\tilde h_{2^q}\|_{L^1_x}^\frac{p}{s}\| v \|_{L^s_x}																\\
		&\leq  2^{-q}\|\tilde h_{2^{-q}} \|_{L^1}^{\frac{p}{r}}\|\tilde h_{2^{-q}} \|_{L^1}^{\frac{p}{s}} \|\nabla u\|_{L^r_x}\|v\|_{L^s_x}.
\end{align*} 
Now, since 
\begin{equation*}
	\| \tilde h_{2^{-q}} \|_{L^1_x} = \int_{\RR^d} 2^{-qd}| \tilde h (2^{-q}x)|\dd x = 
	\int_{\RR^d}| \tilde h(y)|\dd y = \| \tilde h \|_{L^1_x},
\end{equation*}
we finally obtain
\begin{equation*}
	\|[\Dd_q , u] v\|_{L^p_x} \leq  
	2^{-q}\|\tilde h \|_{L^1_x}^{\frac{p}{r}}\|\tilde h \|_{L^1}^{\frac{p}{s}} \|\nabla u\|_{L^r_x}\|v\|_{L^s_x} = 
	\|\tilde h \|_{L^1_x}2^{-q}\|\nabla u\|_{L^r_x}\|v\|_{L^s_x}
\end{equation*}
so the constant in the inequality is $C=\|\tilde h\|_{L^1}$ and it does not depend on $p,r,s$.
\end{proof}
We will also make use of a Bernstein-type inequality evolving the operator $\Sd_q$. 
\begin{lemma}\label{BS2_bernstesin_ineq2}
	there exist two positive constants $\tilde c$ and $\tilde C$ such that 
	\begin{equation*}
		\tilde 
		c\| (\Sd_q -\Sd_{q'}) u\|_{L^p_x}\le 2^{-q}\|(\Sd_q -\Sd_{q'})\nabla u\|_{L^p_x}\le 
		\tilde C\|(\Sd_q -\Sd_{q'}) u\|_{L^p_x}, \forall 1\le p\le \infty,
	\end{equation*}
	for any integers $q$ and $ q'$ with $|q-q'|\leq 5$.
\end{lemma}
\begin{proof} First, we consider new localizer functions as follows:
	\begin{equation*}
		\tilde \varphi_q(\xi) := \frac{1}{10}\sum_{|q-j|\leq 10}\varphi_{j}(\xi) \quad\text{and}\quad 
		\tilde \chi (\xi):=
		\begin{cases}
				\sum_{q\leq -1} \tilde \varphi_{q} (\xi)	&\text{if}\quad \xi\neq 0,\\
				1											&\text{otherwise},
		\end{cases} 
	\end{equation*}
	so that \eqref{lpfond1} and \eqref{lpfond2} are satisfied with $\tilde \varphi$ and $\tilde \chi$ instead of $\varphi$ and $\chi$. Then defining 
	the new homogeneous dyadic block $\dot{\tilde  \Delta}_q$ in the same line of $\Dd_q$, we have
	\begin{equation*}
		\dot{\tilde  \Delta}_q( \Sd_q - \Sd_{q'}) u = 
		 \frac{1}{10}\sum_{|q-j|\leq 10}
		 \Dd_j (\Sd_q - \Sd_{q'})(u)
		 =
		 \frac{1}{10}
		 (\Sd_q - \Sd_{q'})(u).
	\end{equation*}
	Then the inequality turns out from {\bf (i)} of Lemma \ref{lemma:bernstein&commutator}, making use of $ \dot{\tilde  \Delta}_q$ instead of $\Dd_q$. 
\end{proof}
\subsection{The proof of the uniqueness}

\noindent In this section we provide the proof of the uniqueness result for the weak solutions of system \eqref{system}. The main idea is to evaluate the difference 
between two weak solutions in a functional space which is less regular than $L^2_x$ such as $\Hh^{-\frac{1}{2}}$. Such strategy is not new in literature, for 
instance we recall \cite{MR1813331} and \cite{MR2309504}.  We now provide the uniqueness part of the proof of Theorem~\ref{thm: uniqueness}.

\begin{proof}
Let us consider two weak solutions $(u_1,\,Q_1)$ and $(u_2,\,Q_2)$ of system \eqref{system}. We denote  
$\delta u := u_1 - u_2$ and $\delta Q := Q_1 - Q_2$ while $\delta S(Q,\,\nabla u)$ stands for $S(Q_1,\,\nabla u_1) 
-S(Q_2,\,\nabla u_2)$. Similarly, we define  $\delta H(Q)$, $\delta F(Q)$, $\delta \tau$ and $\delta \sigma$. 
Thus $(\delta u,\,\delta Q)$ is a weak solution of
\begin{equation}\label{system_uniqueness}
\begin{cases}
	\,	\partial_t \delta Q - L\Delta \delta Q =  \delta S(Q, \nabla u)	+\Gamma \delta H(Q)	- 
					\delta u\cdot \nabla Q_1 -u_2\cdot\nabla \delta Q 			&\RR_+\times\RR^2,\\
	\,	\partial_t \delta u - \Delta \delta u + \nabla \delta \Pi  = 
			L\,
			\Div\{ 	
					\delta \tau + \delta \sigma 
				\}-
				\delta u \cdot \nabla u_1 -
				u_2\cdot \nabla \delta u										&\RR_+\times\RR^2,\\
	\,	\Div\, \delta u=0														&\RR_+\times\RR^2,\\ 
	\,	(\delta u,\,\delta Q)_{t=0}=(0,\,0)										&\quad\quad\;\;\,\RR^2.
\end{cases}
\end{equation}
First, let us explicitly state $\delta S(Q,\,\nabla u)$, $\delta F(Q)$, $\delta \tau$ and $\delta \sigma$ in terms of
$\delta Q$ and $\delta u$, namely:
\begin{align*}
	\delta S(Q,\nabla u)	=&+	(\xi	\delta	D 	+	\delta 	\Omega		)\delta Q  							+
								(\xi	\delta	D 	+	\delta  \Omega		)(\;	Q_2 + \frac{\Id}{2} \;) 	+
								(\xi			D_2 + 			\Omega_2 	)\delta Q 							+
								\delta	Q 							(\xi	\delta D -	\delta	\Omega 	)		+\\&+	
								(\; 		Q_2 + \frac{\Id}{2} \;)	(\xi	\delta	D 	-	\delta  \Omega )	+
								\delta 	Q							(\xi			D_2 - 			\Omega_2 )	-
		2\xi	\delta 	Q	\,	\trc(	\delta	Q 	\nabla	\delta 	u 	)									- 
		2\xi	\delta 	Q  	\, \trc(	\delta	Q	\nabla  		u_2	) 									+\\&-
		2\xi	\delta 	Q  	\, \trc(	 		Q_2 \nabla	\delta	u 	) 									-
		2\xi			(\,Q_2+\frac{\Id}{2}\,) 	\trc(	\delta	Q	\nabla 	\delta 	u 	) 				-
		2\xi	\delta 	Q  	\, 						\trc(			Q_2	\nabla			u_2 ) 				+\\&-
		2\xi		 	(\,Q_2+\frac{\Id}{2}\,)	 	\trc(	\delta	Q	\nabla			u_2 ) 				-
		2\xi		 	(\,Q_2+\frac{\Id}{2}\,)	 	\trc(			Q_2	\nabla	\delta	u	), 							
\end{align*}
\begin{equation*}
	\delta F(Q)			=		-a\,			\delta Q													
								+b\,		(\,	Q_1	\delta	Q	+	\delta	Q	Q_2	\,)
								-b\,	\trc\{	\delta Q Q_1 + Q_2 \delta Q	\}\frac{\Id}{2}	+			
								-c\,	\big[		\delta	Q	\trc\{			Q_1^2					\}
								+							Q_2	\trc\{	\delta	Q			Q_1			
								+											 	Q_2	\delta	Q			\}
										\big]
\end{equation*}
\begin{equation*}
	\delta H(Q) = \delta	F(Q)	+ L	\Delta \delta	Q.
\end{equation*}
\begin{align*}
	&\delta \tau	=	- \xi			\delta	Q\,										F(	Q_1	)
					- \xi	(\;  			Q_2	+ 	\frac{\Id}{2}	\;)		\delta	F(	Q	)
					-L\xi		 	\delta	Q\,								\Delta	\delta	Q	
					-L\xi			\delta	Q\,								\Delta			Q_2					
					-L\xi	(\;  			Q_2	+ 	\frac{\Id}{2}	\;)		\Delta	\delta	Q					+\\&	
					- \xi			F(	Q_1	)					\delta	Q										
					- \xi	\delta	F(	Q	)	(\; 			Q_2	+ 	\frac{\Id}{2}	\;) 
					-L\xi	\Delta \delta	Q		 			\delta	Q								
					-L\xi	\Delta			Q_2					\delta	Q										
					-L\xi	\Delta	\delta	Q	(\; 			Q_2	+ 	\frac{\Id}{2}	\;) 					+\\&
			+2\xi			\delta	Q								\trc \{		 	Q_1			F(	Q_1	) \}
			+2\xi				 	Q_2								\trc \{	\delta 	Q	 		F(	Q_1	) \}	+\\&
			+2\xi				 	Q_2								\trc \{		 	Q_2	\delta 	F(	Q	) \}			
	+2L\xi			\delta	Q								\trc \{	\delta 	Q	\Delta	\delta	Q	 \}			
	+2L\xi			\delta 	Q								\trc \{	\delta 	Q	\Delta			Q_2	 \}	
	+2L\xi			\delta 	Q								\trc \{		 	Q_2	\Delta	\delta 	Q	 \}			+\\&
	+2L\xi	(\; 			Q_2	+	\frac{\Id}{2}	\;)		\trc \{	\delta 	Q	\Delta	\delta 	Q	 \}			
	+2L\xi			\delta 	Q								\trc \{		 	Q_2	\Delta			Q_2	 \}			
	+2L\xi	(\; 			Q_2	+	\frac{\Id}{2}	\;)		\trc \{	\delta 	Q	\Delta			Q_2	 \}			+\\&
	+2L\xi	(\; 			Q_2	+	\frac{\Id}{2}	\;)		\trc \{		 	Q_2	\Delta	\delta 	Q	 \}			
	-L		\nabla	\delta	Q	\odot \nabla 		Q_1 	  
	-L		\nabla			Q_2	\odot \nabla \delta Q 	 
	-L		\frac{\Id}{2} \trc\{\delta Q 			Q_1 \}	
	-L		\frac{\Id}{2} \trc\{	   Q_2 \delta	Q	\}
\end{align*}

\begin{align*}
	\delta \sigma =				\delta	Q	F(	Q_1	)	+	Q_2	\delta	F(	Q	)
						-		F(	Q_1	)	\delta	Q	-	\delta	F(	Q	)Q_2								  &			
						+L		\delta	Q	\Delta	\delta	Q
						+L				Q_2	\Delta	\delta	Q
						+L		\delta	Q	\Delta			Q_2													\\&
						-L		\Delta	\delta	Q			\delta	Q
						-L		\Delta			Q_2			\delta	Q
						-L		\Delta	\delta	Q					Q_2
\end{align*}

\noindent Taking the inner product in $\Hh^{-1/2}$ of the first equation with  $-L\lambda \Delta \delta Q$ and adding to it  the scalar 
product in $\Hh^{-1/2}$ of the second one by $\frac{1}{\lambda}\delta u$ we get:
\begin{align}
	&\frac{\dd}{\dd t}
	\Big[
		\frac{1}{2\lambda}	\|			\delta u	\|_{\dot{H}^{-\frac{1}{2}}}^2 + 
		L					\|	\nabla	\delta Q	\|_{\dot{H}^{-\frac{1}{2}}}^2
	\Big]+
	\frac{\nu}{\lambda}		\|	\nabla 	\delta u	\|_{\dot{H}^{-\frac{1}{2}}}^2 +
	\Gamma L^2				\|	\Delta	\delta Q	\|_{\dot{H}^{-\frac{1}{2}}}^2
	=																														\nonumber\\&
	-L		\langle	(	\xi	\delta	D 	+	\delta 	\Omega		)\delta Q					,	\Delta	\delta	Q	
			\rangle
\underbracket[0.5pt][1pt]{ 	
	-L\xi	\langle		\delta	D \,	Q_2													,	\Delta	\delta	Q	
			\rangle}_{ \Aa_1}
\underbracket[0.5pt][1pt]{ 	
	-L	\langle		\delta  \Omega\,	Q_2													,	\Delta	\delta	Q	
			\rangle}_{ \Bb_1}
\underbracket[0.5pt][1pt]{ 
	-L\xi	\langle		\frac{\delta	D}{2} 												,	\Delta	\delta	Q
			\rangle}_{ \CC_1}			
\underbracket[0.5pt][1pt]{ 
	-L	\langle		\frac{\delta	\Omega}{2}												,	\Delta	\delta	Q	
			\rangle}_{ \mathcal{D}_1}																						\nonumber\\&
	-L		\langle	(\xi			D_2 + 			\Omega_2 	)\delta Q					,	\Delta	\delta	Q	
			\rangle																	 								
	-L		\langle		\delta	Q 	(\xi	\delta D -	\delta	\Omega 	)					,	\Delta	\delta	Q
			\rangle																								
\underbracket[0.5pt][1pt]{ 	
	-L\xi	\langle		Q_2 \delta	D 														,	\Delta	\delta	Q
			\rangle}_{ \Aa_2}
\underbracket[0.5pt][1pt]{ 	
	+L		\langle		Q_2	\delta  \Omega													,	\Delta	\delta	Q
			\rangle}_{ \Bb_2}																								\nonumber\\&
\underbracket[0.5pt][1pt]{ 
	-L\xi	\langle		\frac{\delta	D}{2} 												,	\Delta	\delta	Q
			\rangle}_{ \CC_2}
\underbracket[0.5pt][1pt]{ 	
	+L		\langle		\frac{\delta  \Omega}{2}											,	\Delta	\delta	Q
			\rangle}_{ \mathcal{D}_2}			
	-L		\langle		\delta 	Q	(\xi			D_2 - 			\Omega_2 )				,	\Delta	\delta	Q
			\rangle																	 							
	+2L\xi	\langle			\delta 	Q	\,	\trc(	\delta	Q 	\nabla	\delta 	u 	)		,	\Delta	\delta	Q
			\rangle						 																					\nonumber\\&
	+2L\xi	\langle			\delta 	Q  	\, \trc(	\delta	Q	\nabla  		u_2	)		,	\Delta	\delta	Q
			\rangle
	+2L\xi	\langle			\delta 	Q  	\, \trc(	 		Q_2 \nabla	\delta	u 	)		,	\Delta	\delta	Q
			\rangle																								
	+2L\xi	\langle					Q_2	\, \trc(	\delta	Q	\nabla 	\delta 	u 	)		,	\Delta	\delta	Q
			\rangle																											\nonumber\\&
\underbracket[0.5pt][1pt]{ 	
	+2L\xi	\langle			\frac{\Id}{2}  \trc(	\delta	Q	\nabla 	\delta 	u 	)		,	\Delta	\delta	Q
			\rangle}_{ =0}		
	+2L\xi	\langle			\delta 	Q  	\, \trc(			Q_2	\nabla			u_2 )		,	\Delta	\delta	Q
			\rangle																								
	+2L\xi	\langle				 	Q_2	\, \trc(	\delta	Q	\nabla			u_2 )		,	\Delta	\delta	Q
			\rangle																											\nonumber\\&
\underbracket[0.5pt][1pt]{ 	
	+2L\xi	\langle			\frac{\Id}{2}  \trc(	\delta	Q	\nabla			u_2 )		,	\Delta	\delta	Q
			\rangle}_{ =0}		
\underbracket[0.5pt][1pt]{ 	
	+2L\xi	\langle				 	Q_2	\, \trc(			Q_2	\nabla	\delta	u	)		,	\Delta	\delta	Q
			\rangle}_{ \EE_1}
\underbracket[0.5pt][1pt]{ 
	+2L\xi	\langle			\frac{\Id}{2}  \trc(			Q_2	\nabla	\delta	u	)		,	\Delta	\delta	Q
			\rangle}_{ =0}																									\nonumber\\&
	+La  \Gamma	\langle	\delta Q															,	\Delta	\delta	Q
				\rangle
	-Lb  \Gamma	\langle	Q_1	\delta	Q	+	\delta	Q	Q_2									,	\Delta	\delta	Q
				\rangle																						
\underbracket[0.5pt][1pt]{ 		
	+Lb  \Gamma	\langle	\trc\{	\delta Q Q_1 + Q_2 \delta Q	\}\frac{\Id}{2}					,	\Delta	\delta	Q
				\rangle}_{=0}																								\nonumber\\&
	+Lc  \Gamma	\langle	\delta	Q	\trc\{			Q_1^2					\}				,	\Delta	\delta	Q
				\rangle
	+Lc  \Gamma	\langle	Q_2	\trc\{	\delta	Q	Q_1 + Q_2	\delta	Q 	    \}				,	\Delta	\delta	Q	
				\rangle																									
	+L	\langle	\delta	u	\cdot \nabla		Q_1											,	\Delta	\delta	Q
		\rangle																												\nonumber\\&
	+L	\langle			u_2	\cdot \nabla \delta Q											,	\Delta	\delta	Q
		\rangle																									
	-a\xi		\langle	\delta Q	Q_1														,	\nabla \delta u
				\rangle																							
	+b\xi		\langle	\delta Q	Q^2_1													,	\nabla \delta u
				\rangle
	-b\xi		\langle	\delta Q	\trc (Q^2_1)\frac{\Id}{2}								,	\nabla \delta u
				\rangle																										\label{energy_uniq} \\&
	-c\xi		\langle	\delta Q	\trc(Q^2_1)Q_1											,	\nabla \delta u
				\rangle																							
	-a\xi	\langle	(\,Q_2 + \frac{\Id}{2}\,)\delta	Q										,	\nabla \delta u
			\rangle																								
	+b\xi	\langle	(\,Q_2 + \frac{\Id}{2}\,)(Q_1	\delta Q + \delta Q	Q_2	)				,	\nabla \delta u
			\rangle																											\nonumber\\&
	-b\xi	\langle	\frac{Q_2}{2}\trc\{	\delta Q Q_1 + Q_2 \delta Q	\}						,	\nabla \delta u
			\rangle																								
\underbracket[0.5pt][1pt]{ 
	-b\xi	\langle	\trc\{	\delta Q Q_1 + Q_2 \delta Q	\}\frac{\Id}{9}						,	\nabla \delta u
			\rangle}_{ =0}
	-c\xi	\langle	(\,Q_2 + \frac{\Id}{2}\,)\delta	Q	\trc\{	Q_1^2	\}					,	\nabla \delta u
			\rangle																											\nonumber\\&			
	-c\xi	\langle	(\,Q_2 + \frac{\Id}{2}\,)Q_2\trc\{\delta Q  Q_1 + Q_2	\delta Q\}		,	\nabla \delta u
			\rangle																								
	+L\xi	\langle			\delta	Q								\Delta \delta	Q		,	\nabla \delta u
			\rangle																						
	+L\xi	\langle			\delta	Q								\Delta 			Q_2		,	\nabla \delta u
			\rangle																											\nonumber\\&
\underbracket[0.5pt][1pt]{ 	
	+L\xi	\langle					Q_2 							\Delta 	\delta	Q		,	\nabla \delta u
			\rangle}_{\Aa_3}
\underbracket[0.5pt][1pt]{ 
	+L\xi	\langle						 					  \frac{\Delta 	\delta	Q}{2}	,	\nabla \delta u
			\rangle}_{\CC_3}
	-a\xi	\langle	Q_1												\delta	Q				,	\nabla \delta u
			\rangle
	+b\xi	\langle	(\,Q_1^2 - \trc\{\,Q^2_1\,\}\frac{\Id}{2}\,)	\delta	Q				,	\nabla \delta u
			\rangle																											\nonumber\\&
	-c\xi	\langle	Q_1^2\trc\{\,Q_1^2\,\}							\delta	Q				,	\nabla \delta u
			\rangle																								
	-a\xi	\langle	\delta Q									  (\,Q_2 + \frac{\Id}{2}\,)	,	\nabla \delta u
			\rangle
	+b\xi	\langle	(\,Q_1	\delta	Q	+	\delta	Q	Q_2	\,)	  (\,Q_2 + \frac{\Id}{2}\,)	,	\nabla \delta u
			\rangle																											\nonumber\\&
	-b\xi	\langle	\trc\{\delta Q Q_1 + Q_2 \delta Q\}\frac{Q_2}{2}						,	\nabla \delta u
			\rangle
\underbracket[0.5pt][1pt]{ 
	-b\xi	\langle	\trc\{\delta Q Q_1 + Q_2 \delta Q\} \frac{\Id}{9}						,	\nabla \delta u
			\rangle}_{ =0}																				
	-c\xi	\langle	\delta	Q	\trc\{Q_1^2\}					  (\,Q_2 + \frac{\Id}{2}\,)	,	\nabla \delta u
			\rangle																											\nonumber\\&
	-c\xi	\langle	Q_2	\trc\{\delta Q\delta Q_1	+Q_2\delta Q\}(\,Q_2 + \frac{\Id}{2}\,)	,	\nabla \delta u
			\rangle
	+L\xi	\langle			\Delta	\delta	Q	\,	\delta	Q								,	\nabla \delta u
			\rangle
	+L\xi	\langle			\Delta			Q_2		\delta	Q								,	\nabla \delta u
			\rangle																											\nonumber\\&
\underbracket[0.5pt][1pt]{ 			
	+L\xi	\langle			\Delta	\delta	Q	Q_2											,	\nabla \delta u
			\rangle}_{ \Aa_4}
\underbracket[0.5pt][1pt]{ 	
	+L\xi	\langle			\frac{\Delta	\delta	Q}{2}							,	\nabla \delta u
			\rangle}_{ \CC_4}																					
	+2a\xi		\langle	\delta Q			\trc\{\,	Q_1^2	\}							,	\nabla \delta u
				\rangle	
	-2b\xi		\langle	\delta Q			\trc\{\,	Q_1^3	\}							,	\nabla \delta u
				\rangle																										\nonumber\\&
\underbracket[0.5pt][1pt]{ 					
	+2b\xi		\langle	\frac{\delta Q}{2}	\trc\{\,	Q_1		\}	\trc\{	Q_1^2	\}		,	\nabla \delta u
				\rangle}_{=0}																							
	+2c\xi		\langle	\delta Q			\trc\{\,	Q_1^2\}^2							,	\nabla \delta u
				\rangle		
	+2a\xi		\langle			  Q_2		\trc\{\delta	Q Q_1	\}						,	\nabla \delta u
				\rangle																										\nonumber\\&
	-2b\xi		\langle			  Q_2		\trc\{\delta	Q Q_1^2	\}						,	\nabla \delta u
				\rangle																							
\underbracket[0.5pt][1pt]{ 	
	+2b\xi		\langle		\frac{Q_2}{2}	\trc\{\delta	Q	 \}\trc\{	Q_1^2	\}		,	\nabla \delta u
				\rangle}_{=0}																					
	+2c\xi		\langle			  Q_2		\trc\{\delta	Q Q_1\}\trc\{	Q_1^2	\}		,	\nabla \delta u
				\rangle																										\nonumber\\&
	+2a\xi		\langle	Q_2\trc\{Q_2	\delta	Q	\}										,	\nabla \delta u
				\rangle																					
	-2b\xi		\langle	Q_2\trc\{Q_2	(\,	Q_1	\delta	Q	+	\delta	Q	Q_2	\,)\}		,	\nabla \delta u
				\rangle																										\nonumber\\&
\underbracket[0.5pt][1pt]{ 
	+2b\xi		\langle	Q_2\trc\{\frac{Q_2}{2}\}\trc\{\delta Q Q_1 + Q_2 \delta Q\}			,	\nabla \delta u
				\rangle}_{=0}																					
	+2c\xi		\langle	Q_2\trc\{Q_2\delta	Q\}\trc\{Q_1^2\}								,	\nabla \delta u
				\rangle																										\nonumber\\&
	+2c\xi		\langle	Q_2\trc\{Q_2^2\}\trc\{	\delta	Q		Q_1	+Q_2\delta	Q\}			,	\nabla \delta u
				\rangle																							
	-2L\xi		\langle	\delta Q				   \trc\{\delta Q	\Delta	\delta Q	\}	,	\nabla \delta u
				\rangle																										\nonumber\\&	
	-2L\xi		\langle	\delta Q				   \trc\{\delta Q	\Delta		   Q_2	\}	,	\nabla \delta u
				\rangle
	-2L\xi		\langle	\delta Q				   \trc\{	    Q_2	\Delta	\delta Q	\}	,	\nabla \delta u
				\rangle																										\nonumber\\&
	-2L\xi		\langle		   Q_2				   \trc\{\delta Q	\Delta	\delta Q	\}	,	\nabla \delta u
				\rangle
\underbracket[0.5pt][1pt]{ 				
	-2L\xi		\langle	\frac{\Id}{2}			   \trc\{\delta Q	\Delta	\delta Q	\}	,	\nabla \delta u
				\rangle}_{ =0}
	-2L\xi		\langle	\delta Q				   \trc\{		Q_2	\Delta		   Q_2	\}	,	\nabla \delta u
				\rangle																										\nonumber\\&
	-2L\xi		\langle		   Q_2				   \trc\{\delta Q	\Delta		   Q_2	\}	,	\nabla \delta u
				\rangle
\underbracket[0.5pt][1pt]{ 
	-2L\xi		\langle	\frac{\Id}{2}			   \trc\{\delta Q	\Delta		   Q_2	\}	,	\nabla \delta u
				\rangle}_{ =0}
\underbracket[0.5pt][1pt]{ 
	-2L\xi		\langle		   Q_2				   \trc\{		Q_2	\Delta	\delta Q	\}	,	\nabla \delta u
				\rangle}_{ \EE_2}																							\nonumber\\&
\underbracket[0.5pt][1pt]{ 
	-2L\xi		\langle	\frac{\Id}{2}			   \trc\{		Q_2	\Delta	\delta Q	\}	,	\nabla \delta u
				\rangle}_{ =0}
	+L			\langle	\nabla	\delta	Q	\odot \nabla 		Q_1 						,	\nabla \delta u
				\rangle		
	+L			\langle	\nabla			Q_2	\odot \nabla \delta Q_1							,	\nabla \delta u
				\rangle																							
\underbracket[0.5pt][1pt]{ 	
	+L			\langle	\frac{\Id}{2} \trc\{\delta Q 			Q_1 \}						,	\nabla \delta u
				\rangle}_{ =0}																								\nonumber\\&
\underbracket[0.5pt][1pt]{ 
	+L			\langle	\frac{\Id}{2} \trc\{	   Q_2 \delta	Q	\}						,	\nabla \delta u
				\rangle}_{ =0}
	+La		\langle	\delta	Q	Q_1															,	\nabla \delta u
			\rangle
	-Lb		\langle	\delta	Q	(\,	Q_1^2	-	\trc\{Q_1^2\}\frac{\Id}{2}	\,)				,	\nabla \delta u
			\rangle
	+Lc		\langle	\delta	Q	Q_1\trc\{Q_1^2\}											,	\nabla \delta u
			\rangle																											\nonumber\\&
	+a		\langle	Q_2	\delta Q															,	\nabla \delta u
			\rangle
	-b		\langle	Q_2	(\,	Q_1	\delta	Q	+	\delta	Q	Q_2	\,)							,	\nabla \delta u
			\rangle																						
	+b		\langle	Q_2 \trc\{	\delta Q Q_1 + Q_2 \delta Q	\}\frac{\Id}{2}					,	\nabla \delta u
			\rangle																											\nonumber\\&
	+c		\langle	Q_2 \delta	Q	\trc\{			Q_1^2						\}			,	\nabla \delta u
			\rangle																								
\underbracket[0.5pt][1pt]{	
	+c		\langle	Q_2^2 			\trc\{	\delta	Q	Q_1 +Q_2	\delta	Q	\}			,	\nabla \delta u
			\rangle}_{\FF_1}																								\nonumber\\&
	-a		\langle		Q_1														\delta	Q	,	\nabla \delta u
			\rangle
	+b		\langle		(\,	Q_1^2	-	\trc\{Q_1^2\}\frac{\Id}{2}	\,)			\delta	Q	,	\nabla \delta u
			\rangle
	-c		\langle		Q_1\trc\{Q_1^2\}										\delta	Q	,	\nabla \delta u
			\rangle																								
	-a		\langle		\delta Q													Q_2		,	\nabla \delta u
			\rangle																											\nonumber\\&
	+b		\langle		(\,	Q_1	\delta	Q	+	\delta	Q	Q_2	\,)					Q_2		,	\nabla \delta u
			\rangle																						
	-b		\langle	 	\trc\{	\delta Q Q_1 + Q_2 \delta Q	\}\frac{\Id}{2}			Q_2		,	\nabla \delta u
			\rangle																								
	-c		\langle		 \delta	Q	\trc\{			Q_1^2						\}	Q_2		,	\nabla \delta u
			\rangle																											\nonumber\\&
\underbracket[0.5pt][1pt]{	
	-c		\langle		 		Q_2	\trc\{	\delta	Q	Q_1 +Q_2	\delta	Q	\}	Q_2		,	\nabla \delta u
			\rangle}_{\FF_2}																								
	-L		\langle	\delta	Q	\Delta	\delta	Q											,	\nabla \delta u
			\rangle																								
\underbracket[0.5pt][1pt]{ 	
	-L		\langle			Q_2	\Delta	\delta	Q											,	\nabla \delta u
			\rangle}_{ \Bb_3}
	-L		\langle	\delta	Q	\Delta			Q_2											,	\nabla \delta u
			\rangle																											\nonumber\\&
	+L		\langle	\Delta	\delta	Q	\delta	Q											,	\nabla \delta u
			\rangle																								
\underbracket[0.5pt][1pt]{ 	
	+L		\langle	\Delta	\delta	Q			Q_2											,	\nabla \delta u
			\rangle}_{ \Bb_4}
	+L		\langle	\Delta			Q_2	\delta	Q											,	\nabla \delta u
			\rangle																								
	-		\langle			u_2	\cdot	\nabla	\delta	u									,			\delta u
			\rangle
	-		\langle	\delta	u	\cdot	\nabla			u_1									,			\delta u			\nonumber
			\rangle.
\end{align}

\noindent
Denoting  $$\Phi(t):= 1/(2\lambda) \| \delta u(t) \|_{ \Hh^{-1/2}}^2 + L  \|\nabla \delta Q(t) \|_{\Hh^{-1/2}}^2$$ we  will aim to show that $\Phi$ satisfies the inequality
\begin{equation}\label{uniqueness_inequality}
	\Phi'(t) \leq \chi(t)\mu( \Phi(t) ), 
\end{equation}
where $\mu$ is an Osgood modulus of continuity (see \cite{MR2768550}, Definition $3.1$), given by

\begin{equation}\label{Osgood_mc}
	\mu(r) := r + r \ln \Big( 1+e+ \frac{1}{r}\Big) + r \ln\Big( 1+e+  \frac{1}{r}\Big) \ln  \ln \Big(1+e + \frac{1}{r}\Big).
\end{equation} with $\chi\in L^1_{loc}$ apriori.
We are going to find a double-logarithmic estimate, hence thanks to the Osgood Lemma (see \cite{MR2768550}, Lemma 3.4) and since $\Phi(0)$ is null, we get that $\Phi\equiv 0$, which yields the uniqueness of the solution for system \eqref{system}. 

\vspace{0.2cm}
\noindent
First, let us observe following simplifications of \eqref{energy_uniq}:
\begin{equation*}
	0=\CC_1 + \CC_2 + \CC_3 + \CC_4 = \mathcal{D}_1 + \mathcal{D}_2 = \FF_1 +\FF_2.
\end{equation*}
The key method we use to obtain the desired estimates  is the para-differential calculus decomposition summarized in the following:
\begin{remark}
	Let $q$ be an integer, and $A$, $B$ be $d\times d$ matrices whose components are homogeneous temperate distributions. We are going to use the following 
	notation:
	\begin{equation*}
	\begin{array}{ll}
		 \J_q^1(A,B)	:=\sum_{|q-q'|\leq 5	}	[\Dd_q, \,		\Sd_{q'-1}A]			\Dd_{q'}	B,	
		&\J_q^3(A,B)	:=							\Sd_{q -1}	A							\Dd_{q}		B,\\
		\\
		 \J_q^2(A,B)	:=\sum_{|q-q'|\leq 5	}	(	\Sd_{q'-1}A	-\Sd_{q-1}A)	\Dd_{q}	\Dd_{q'}	B,		
		&\J_q^4(A,B)	:=\sum_{ q'   \geq q- 5 }	\Dd_q(\Dd_{q'}A\,						\Sd_{q'+2}	B).
	\end{array}
	\end{equation*}
	Than we can decompose the product $AB$ as follows
	\begin{equation}\label{product_decomposition}
	\Dd_q(AB)	=		\J_q^1(A,B)
					+	\J_q^2(A,B)
					+	\J_q^3(A,B)
					+	\J_q^4(A,B)
	\end{equation}
	for any integer $q$, thanks to \eqref{bonydecomp}.
\end{remark}
\noindent Moreover  from now on we will use the notation $\lesssim$ as follows: for any non-negative real numbers $a$ and $b$, we denote $a\lesssim b$ if and only if there exists a positive constant $C$ (independent of $a$ and $b$) such that $a\leq C b $. 
\subsubsection{Estimate of \texorpdfstring{$\Aa_1+\Aa_2+\Aa_3+\Aa_4$}{A1+A2+A3+A4}} Let us begin analyzing the terms $\Aa_1,$ $\Aa_2$, $\Aa_3$ and $\Aa_4$ of \eqref{energy_uniq}. 
First, we observe that
\begin{equation*}
	\Aa_2 	= 	-L\xi	\sum_{q\in\ZZ}	2^{-q}	\langle 
													\Dd_q(Q_2 \delta	D),\Dd_q \Delta \delta Q 
												\rangle_{	L^2_x	}
			=	-L\xi	\sum_{q\in\ZZ}2^{-q}\sum_{i=1}^4		
									\langle	\J_{q}^i(Q_2,\,\delta D),	\Dd_q \Delta \delta Q\rangle_{L^2_x}
\end{equation*}
Now, when $i=1$, we have
\begin{equation}\label{ev1}
\begin{aligned}
	2^{-q}\langle	\J_{q}^1(Q_2,\,\delta D),	\Dd_q \Delta \delta Q\rangle_{L^2_x}& = \sum_{|q-q'|\leq 5}
	2^{-q}
	\langle
		[\Dd_q,\, \Sd_{q'-1}Q_2 ]\Dd_{q'}\delta D, \Dd_q \Delta \delta Q
	\rangle_{L^2_x}\\
	&\lesssim	\sum_{|q-q'|\leq 5}
	2^{-q}
	\|	[\Dd_q,\, \Sd_{q'-1}Q_2 ]\Dd_{q'}\delta D		\|_{L^2_x		}
	\|	\Dd_q \Delta \delta Q							\|_{L^2_x		}\\
	&\lesssim	\sum_{|q-q'|\leq 5}
	2^{-2q}
	\|	\Sd_{q'-1}\nabla Q_2							\|_{L^4_x		}
	\|	\Dd_{q'}\delta D								\|_{L^4_x		}
	\|	\Dd_q \Delta \delta Q							\|_{L^2_x		}\\
	&\lesssim	\sum_{|q-q'|\leq 5}
	2^{-q}
	\|	\Sd_{q'-1}\nabla 	Q_2							\|_{L^2_x		}^{\frac{1}{2}}
	\|	\Sd_{q'-1}\Delta 	Q_2							\|_{L^2_x		}^{\frac{1}{2}}
	\|	\Dd_{q'}\delta 	  	u							\|_{L^4_x		}
	\|	\Dd_q \Delta \delta	Q							\|_{L^2_x		}\\
	&\lesssim	\sum_{|q-q'|\leq 5}
	\|	\nabla 				Q_2							\|_{L^2_x		}^{\frac{1}{2}}
	\|	\Delta 				Q_2							\|_{L^2_x		}^{\frac{1}{2}}
	\|	\Dd_{q'}\delta 	  	u							\|_{L^2_x		}
	2^{-\frac{q}{2}}
	\|	\Dd_q \Delta \delta	Q							\|_{L^2_x		},
\end{aligned}
\end{equation}
for every $q\in \ZZ$. Hence,  we get
\begin{equation}\label{ev1b}
\begin{aligned}
	-L\xi
	\sum_{q\in \ZZ}2^{-q}	\langle	\J_{q}^1(Q_2,\,\delta D),	&\Dd_q \Delta \delta Q\rangle_{L^2_x}	
	\lesssim
	\|	\nabla		Q_2			\|_{			L^2_x			  }^\frac{1}{2}
	\|	\Delta		Q_2			\|_{			L^2_x			  }^\frac{1}{2}
	\|			\delta	u		\|_{			L^2_x			  }
	\|	\Delta	\delta	Q		\|_{	\Hh^{	-\frac{1}{2}	} }\\
	&\lesssim
	\|	\nabla		Q_2			\|_{			L^2_x			  }^\frac{1}{2}
	\|	\Delta		Q_2			\|_{			L^2_x			  }^\frac{1}{2}
	\|			\delta	u		\|_{	\Hh^{	-\frac{1}{2}	}  }^\frac{1}{2}
	\|	\nabla	\delta	u		\|_{	\Hh^{	-\frac{1}{2}	}  }^\frac{1}{2}
	\|	\Delta	\delta	Q		\|_{	\Hh^{	-\frac{1}{2}	} }\\
	&\lesssim
	\|	\nabla			Q_2		\|_{	L^2_x					  }^2
	\|	\Delta			Q_2		\|_{	L^2_x					  }^2		
	\|			\delta	u		\|_{	\Hh^{	-\frac{1}{2}	} }^2+
	C_\nu 
	\|	\nabla	\delta	u		\|_{	\Hh^{	-\frac{1}{2}	} }^2+
	C_{\Gamma, L}
	\|	\Delta	\delta	Q		\|_{	\Hh^{	-\frac{1}{2} 	} }^2,
\end{aligned}
\end{equation}
where we have used the following interpolation inequality:
\begin{equation*}
	\| \delta u \|_{L^2_x} 	\leq  \| \delta u \|_{\Hh^{-\frac{1}{2}}}^\frac{1}{2}\| 			\delta u \|_{\Hh^{ \frac{1}{2}}}^\frac{1}{2}
							=	  \| \delta u \|_{\Hh^{-\frac{1}{2}}}^\frac{1}{2}\| \nabla 	\delta u \|_{\Hh^{-\frac{1}{2}}}^\frac{1}{2}.
\end{equation*}
When $i=2$, the following inequalities are fulfilled:
\begin{equation}\label{ev2}
\begin{aligned}
	2^{-q}\langle	\J_{q}^2(Q_2,\,\delta D),	\Dd_q \Delta \delta Q\rangle_{L^2_x} 
	&= \sum_{|q-q'|\leq 5}
	2^{-q}
	\langle
		(\Sd_{q'-1}Q_2 - \Sd_{q-1}Q_2)\Dd_q \Dd_{q'}\delta D, \Dd_q \Delta \delta Q
	\rangle_{L^2_x}\\
	&\lesssim	\sum_{|q-q'|\leq 5}
	2^{-q}
	\|	(\Sd_{q'-1}Q_2 - \Sd_{q-1}Q_2)					\|_{L^\infty_x}
	\|	\Dd_q \Dd_{q'}\delta D							\|_{L^2_x}
	\|	\Dd_q \Delta \delta Q							\|_{L^2_x}\\
	&\lesssim	\sum_{|q-q'|\leq 5}
	2^{-2q}
	\|	(\Sd_{q'-1}\Delta Q_2 - \Sd_{q-1}\Delta Q_2)	\|_{L^2_x}
	\|	\Dd_q \Dd_{q'}\delta D							\|_{L^2_x}
	\|	\Dd_q \Delta \delta Q							\|_{L^2_x}\\
	&\lesssim	\sum_{|q-q'|\leq 5}
	2^{-2q}
	\|			\Delta Q_2							\|_{L^2_x}
	\|	\Dd_{q'}\delta D								\|_{L^2_x}
	\|	\Dd_q \Delta \delta Q							\|_{L^2_x}\\
	&\lesssim	\sum_{|q-q'|\leq 5}
	2^{-\frac{q'}{2}}
	\|	\Dd_{q'}		\delta u						\|_{L^2_x}
	2^{-\frac{q}{2}}
	\|	\Dd_q \Delta 	\delta Q						\|_{L^2_x}
	\|	\Delta Q_2										\|_{L^2_x},
\end{aligned}
\end{equation}
for any $q\in \ZZ$. Thus, it turns out that
\begin{equation}\label{ev2b}
	-L\xi
	\sum_{q\in \ZZ}2^{-q}\langle	\J_{q}^2(Q_2,\,\delta D),	\Dd_q \Delta \delta Q\rangle_{L^2_x} 
	\lesssim
	\|		\Delta			Q_2			\|_{	L^2_x				}^2
	\|				\delta	u			\|_{	\Hh^{-\frac{1}{2} } }^2+
	C_{\Gamma, L}
	\|		\Delta	\delta	Q			\|_{	\Hh^{-\frac{1}{2} } }^2
\end{equation}
The term corresponding to $i=3$ cannot be estimated as before. We will see that this challenging term will be simplified. Finally,  
when $i=4$, we have
\begin{equation}\label{ev3}
\begin{aligned}
	2^{-q}\langle	\J_{q}^4(Q_2,\,\delta D),	\Dd_q \Delta \delta Q\rangle_{L^2_x} 
	&= L
		2^{-q}
		\sum_{q-q'\leq 5}
		\langle
			\Dd_q\big[\Dd_{q'} Q_2 \Sd_{q'+2}\delta D\big],\, \Dd_q \Delta \delta Q
		\rangle_{L^2_x}\\
	&\lesssim
		2^{-q}
		\sum_{q-q'\leq 5}
		\|	\Dd_{q'} Q_2			\|_{L^\infty_x	}
		\|	\Sd_{q'+2}\delta D		\|_{L^2_x		}
		\|	\Dd_q \Delta \delta Q	\|_{L^2_x		}\\
	&\lesssim
		2^{-q}
		\sum_{q-q'\leq 5}
		2^{-q'}
		\|	\Dd_{q'}\Delta Q_2			\|_{L^2_x	}
		\|	\Sd_{q'+2}\delta D			\|_{L^2_x		}
		2^{q}
		\|	\Dd_q   \nabla \delta Q		\|_{L^2_x		}\\
		&\lesssim
		\sum_{q-q'\leq 5}
		2^{\frac{q-q'}{2}}
		\|	\Dd_{q'}\Delta Q_2			\|_{	L^2_x				}
		2^{-\frac{q'+2}{2}}
		\|	\Sd_{q'+2}\delta D			\|_{	L^2_x				}
		2^{-\frac{q}{2}}
		\|	\Dd_q   \nabla \delta Q		\|_{	L^2_x				}\\
		&\lesssim
		\|	\Delta Q_2					\|_{	L^2_x				}
		2^{-\frac{q}{2}}
		\|	\Dd_q\nabla \delta Q		\|_{	L^2_x }
		\sum_{q-q'\leq 5}
		2^{\frac{q-q'}{2}}
		2^{-\frac{q'+2}{2}}
		\|	\Sd_{q'+2}\delta D			\|_{	L^2_x				},
\end{aligned}
\end{equation}
for any $q\in \ZZ$. Hence, 
\begin{align*}
	-L\xi&\sum_{q\in \ZZ}2^{-q}\langle	\J_{q}^4(Q_2,\,\delta D),	\Dd_q \Delta \delta Q\rangle_{L^2_x}\\
	&\lesssim
	\|	\Delta Q_2									\|_{	L^2_x				}
	\sum_{q\in \ZZ}
	2^{-\frac{q}{2}}\|	\Dd_q\nabla \delta Q				\|_{	L^2_x				}
	\sum_{q'\in\ZZ}
		2^{\frac{q-q'}{2}}1_{(-\infty,5]}(q-q')
		2^{-\frac{q'+2}{2}}
		\|	\Sd_{q'+2}\delta D						\|_{	L^2_x				}\\
	&\lesssim
	\|	\Delta Q_2									\|_{	L^2_x				}
	\| 	\nabla \delta Q								\|_{	\Hh^{-\frac{1}{2} } }
	\Big(
		\sum_{q\in\ZZ}\big|\sum_{q'\in\ZZ}
		2^{q-q'}1_{(-\infty,5]}(q-q')
		2^{-\frac{q'+2}{2}}
		\|	\Sd_{q'+2}\delta D						\|_{	L^2_x				}
	\big|^2
	\Big)^{\frac{1}{2}},
\end{align*}  and by convolution
\begin{align*}
	\Big(
		\sum_{q\in\ZZ}\big|\sum_{q'\in\ZZ}
		&2^{q-q'}1_{(-\infty,5]}(q-q')
		2^{-\frac{q'+2}{2}}
		\|	\Sd_{q'+2}\delta D						\|_{	L^2_x				}
	\big|^2
	\Big)^{\frac{1}{2}}\\
	&\lesssim
	(\sum_{q\leq 5}2^q)\big(
	\sum_{q\in\ZZ}
		2^{-q}
		\|	\Sd_{q}\delta D						\|_{	L^2_x				}^2
	\Big)^{\frac{1}{2}}\big)
	\lesssim \| \nabla \delta u \|_{\Hh^{-\frac{1}{2}}},
\end{align*}
so that
\begin{equation}\label{ev3b}
\begin{aligned}
	-L\xi\sum_{q\in \ZZ}2^{-q}\langle	\J_{q}^4(Q_2,\,\delta D),	\Dd_q \Delta \delta Q\rangle_{L^2_x}
	&\lesssim
	\|	\Delta Q_2					\|_{	L^2_x				}
	\|	\nabla \delta Q				\|_{	\Hh^{-\frac{1}{2} } }
	\|	\nabla \delta u				\|_{	\Hh^{-\frac{1}{2} } }\\
	&\lesssim
	\|	\Delta Q_2					\|_{	L^2_x				}^2
	\|	\nabla \delta Q				\|_{	\Hh^{-\frac{1}{2} } }^2+
	C_\nu 
	\|	\nabla \delta u				\|_{	\Hh^{-\frac{1}{2} } }^2.
\end{aligned}
\end{equation}
Summarizing, it remains to control
\begin{equation*}
	\Aa_1 +\Aa_3 + \Aa_4-L\xi \sum_{q\in\ZZ}2^{-q}\langle	\J_{q}^3(Q_2,\,\delta D),	\Dd_q \Delta \delta Q\rangle_{L^2_x} .
\end{equation*}
Now, observing that 
\begin{align*}
	\Aa_1 	= 	-L\xi	\sum_{q\in\ZZ}	2^{-q}	\langle 
													\Dd_q(\delta	D Q_2 ),\Dd_q \Delta \delta Q 
												\rangle_{	L^2_x	}
			&=	-L\xi	\sum_{q\in\ZZ}	2^{-q}	\langle 
													\Dd_q\tr(\delta	D Q_2 ),\Dd_q \tr\Delta \delta Q 
												\rangle_{	L^2_x	}\\
			&=	-L\xi	\sum_{q\in\ZZ}	2^{-q}	\langle 
													\Dd_q	( Q_2 \delta D ),\Dd_q \Delta \delta Q 
												\rangle_{	L^2_x	}
			= \Aa_2,
\end{align*}
 we estimate $\Aa_1$ with the previous inequalities, so that it remains to control
\begin{align*}
	\Aa_3 + \Aa_4-2L\xi \sum_{q\in\ZZ}2^{-q}\J_{q}^3(Q_2,\,\delta D)
	=
	\Aa_3 + \Aa_4-2L\xi \sum_{q\in\ZZ}2^{-q}
			\int_{\RR^2} 
			\trc
			\big\{
				\Sd_{q-1}Q_2\,\Dd_q \delta D 
			\Dd_q \Delta \delta Q	
			\big\}.	 
\end{align*}

\vspace{0.2cm}
\noindent Now, let us consider 
$\Aa_3=L\xi\langle Q_2 \Delta \delta Q,\nabla\delta u\rangle$. We proceed along the lines used before, namely we use the decomposition given by 
\eqref{product_decomposition}:
\begin{equation*}
	\Aa_3 	= 	L\xi	\sum_{q\in\ZZ}	2^{-q}	\langle 
													\Dd_q( Q_2 \Delta \delta Q),\Dd_q \nabla\delta u
												\rangle_{	L^2_x	}
			=	L\xi	\sum_{q\in\ZZ}	2^{-q}	\sum_{i=1}^4 
												\langle 
													\J_q^i( Q_2,\, \Delta \delta Q),\Dd_q \nabla\delta u
												\rangle_{	L^2_x	}.
\end{equation*}
When $i=1$, proceeding as for \eqref{ev1}, we have
\begin{equation*}
	2^{-q}	\langle 
				\J_q^1( Q_2,\, \Delta \delta Q),\Dd_q \nabla\delta u
			\rangle_{	L^2_x	} 
	\lesssim
	\|				\nabla 	Q_2								\|_{L^2_x	}^\frac{1}{2}
	\|				\Delta 	Q_2								\|_{L^2_x	}^\frac{1}{2}
	\sum_{|q-q'|\leq 5}
	\|	\Dd_{q'}\nabla	\delta Q							\|_{L^2_x	}
	2^{-\frac{q}{2}}
	\|	\Dd_q \nabla \delta u								\|_{L^2_x	},
\end{equation*}
thus, considering the sum over $q\in\ZZ$ as in \eqref{ev1b}, we deduce that
\begin{equation}\label{ev4}
	L\xi
	\sum_{q\in\ZZ}2^{-q}
	\langle 
				\J_q^1( Q_2,\, \Delta \delta Q),\Dd_q \nabla\delta u
	\rangle_{	L^2_x	}
	\lesssim
	\|	\nabla		Q_2				\|_{	L^2_x					  }^2	
	\|	\Delta		Q_2				\|_{	L^2_x					  }^2		
	\|	\nabla	\delta	Q			\|_{	\Hh^{-\frac{1}{2}		} }^2+
	C_\nu 	
	\|	\nabla	\delta	u			\|_{	\Hh^{	-\frac{1}{2}	} }^2+
	C_{\Gamma, L}
	\|	\Delta	\delta	Q			\|_{	\Hh^{	-\frac{1}{2}	} }^2.
\end{equation}
Proceeding as for proving \eqref{ev2}, when $i=2$, we get
\begin{equation*}
	2^{-q} \langle 
				\J_q^2( Q_2,\, \Delta \delta Q),\Dd_q \nabla\delta u
			\rangle_{	L^2_x	}
	\lesssim	\sum_{|q-q'|\leq 5}	
	2^{-\frac{q}{2}}
	\|	\Dd_{q'}		\delta u						\|_{L^2_x}
	2^{-\frac{q}{2}}
	\|	\Dd_q \Delta 	\delta Q						\|_{L^2_x}
	\|	\Delta Q_2										\|_{L^2_x}
\end{equation*}	
for every $q\in \ZZ$. Thus, as in \eqref{ev2b}, it turns out that
\begin{equation}\label{ev5}
	L\xi
	\sum_{q\in \ZZ}2^{-q}
	\langle 
			\J_q^2( Q_2,\, \Delta \delta Q),\Dd_q \nabla\delta u
	\rangle_{	L^2_x	}
	\lesssim
	\|		\Delta			Q_2			\|_{	L^2_x				}^2
	\|				\delta	u			\|_{	\Hh^{-\frac{1}{2} } }^2+
	C_{\Gamma, L}
	\|		\Delta	\delta	Q			\|_{	\Hh^{-\frac{1}{2} } }^2
\end{equation}
Finally, with the same strategy as for  \eqref{ev3},  we observe that
\begin{equation*}
	2^{-q}
	\langle 
		\J_q^4( Q_2,\, \Delta \delta Q),\Dd_q \nabla\delta u
	\rangle_{	L^2_x	}
	\lesssim
		\|	\Delta Q_2					\|_{	L^2_x				}
		2^{-q}
		\|		\Dd_q		\delta	u		\|_{	L^2_x 				}^2
		\sum_{q-q'\leq 5}
		2^{\frac{q-q'}{2}}
		2^{-\frac{q'+2}{2}}
		\|	\Sd_{q'+2}\Delta \delta Q	\|_{	L^2_x				},	
\end{equation*}
hence, as for \eqref{ev3b}, we obtain
\begin{equation}\label{ev6}
\begin{aligned}
	L\xi
	\sum_{q\in \ZZ}2^{-q}
	\langle 
		\J_q^4( Q_2,\, \Delta \delta Q),\Dd_q \nabla\delta u
	\rangle_{	L^2_x	}
	&\lesssim
	\|	\Delta Q_2					\|_{	L^2_x				}
	\|	 \delta u					\|_{	\Hh^{-\frac{1}{2} } }
	\|	\Delta \delta Q				\|_{	\Hh^{-\frac{1}{2} } }\\
	&\lesssim
	\|	\Delta Q_2					\|_{	L^2_x				}^2
	\|	 \delta u					\|_{	\Hh^{-\frac{1}{2} } }^2+
	C_{\Gamma, L}
	\|	\Delta \delta Q				\|_{	\Hh^{-\frac{1}{2} } }^2.
\end{aligned}
\end{equation}
Summarizing all the previous considerations, we note that it remains to control
\begin{align*}
	& \Aa_4+L\xi \sum_{q\in\ZZ}2^{-q}
	\Big[
			\langle 
				\J_q^3( Q_2,\, \Delta \delta Q),\Dd_q \nabla\delta u
			\rangle_{L^2_x}
			-2\int_{\RR^2} 
			\trc
			\big\{
				\Sd_{q-1}Q_2\,\Dd_q \delta D 
			\Dd_q \Delta \delta Q	
			\big\}
	\Big] =\\&	
	=
	\Aa_4 +			
	L\xi \sum_{q\in\ZZ}2^{-q}
			\int_{\RR^2} 
			\Big[
			\trc\{	\Sd_{q-1}Q_2\,\Dd_q \Delta \delta Q \Dd_q\nabla \delta u	\}-
			2\trc
			\big\{
				\Sd_{q-1}Q_2\,\Dd_q \delta D 
			\Dd_q \Delta \delta Q	
			\big\}
			\Big].								 
\end{align*}
We handle the last term $\Aa_4$ arguing as for $\Aa_3$, since $\Aa_4$ is given by
\begin{align*}
	L\xi\langle \Dd_q(\Delta \delta Q Q_2), 		\Dd_q\nabla \delta u\rangle_{\Hh^{-\frac{1}{2}}}
	=L\xi\langle \Dd_q(Q_2\Delta \delta Q) , \tr	\Dd_q\nabla \delta u\rangle_{\Hh^{-\frac{1}{2}}}
	=L\xi\sum_{q\in\ZZ}2^{-q}\sum_{i=1}^4 \langle \J_q^i(Q_2,\,\Delta \delta Q),\, \Dd_q\tr\nabla \delta u\rangle_{L^2_x}.
\end{align*}
The terms related to $i=1,2,4$ are estimated similarily as $\Aa_3$. Hence it remains to evaluate
\begin{align*}
	L\xi \sum_{q\in\ZZ}2^{-q}
			\Big\{
			\langle \J_q^3(Q_2,\,\Delta \delta Q),\, \Dd_q\tr\nabla \delta u\rangle_{L^2_x} +
			\int_{\RR^2} 
			\big[
			\trc\{	\Sd_{q-1}Q_2\,\Dd_q \Delta \delta Q \Dd_q\nabla \delta u	\}-
			2\trc
			\big\{
				\Sd_{q-1}Q_2\,\Dd_q \delta D 
			\Dd_q \Delta \delta Q	
			\big\}
			\big]
			\Big\} = \\=
	2L\xi \sum_{q\in\ZZ}2^{-q}			
			\int_{\RR^2} 
			\Big[
			\trc\{	\Sd_{q-1}Q_2\,\Dd_q \Delta \delta Q \Dd_q \delta D	\}-
			\trc
			\big\{
				\Sd_{q-1}Q_2\,\Dd_q \delta D 
			\Dd_q \Delta \delta Q	
			\big\}
			\Big]=0,
\end{align*}
which is a null series since the trace acts on symmetric matrices, so that we can permute their order.

\vspace{0.2cm}
\subsubsection{Estimate of \texorpdfstring{$\Bb_1+\Bb_2+\Bb_3+\Bb_4$}{TEXT}}

\vspace{0.1cm}	
\noindent Now we want to estimate $\Bb_1+\Bb_2+\Bb_3+\Bb_4$, namely
\begin{equation*}
-L\langle		\delta \Omega Q_2 - Q_2 \delta \Omega      ,\Delta\delta Q	\rangle_{\dot{H}^{-\frac{1}{2}}} -
	L		\langle		Q_2 \Delta \delta Q	-		 
						\Delta \delta Q  Q_2 		 ,	\nabla	\delta	u	\rangle_{\dot{H}^{-\frac{1}{2}}}.
\end{equation*}
First let us consider
\begin{equation*}
	\Bb_2 = 
	L\langle Q_2 \delta \Omega  ,\Delta\delta Q	\rangle_{\dot{H}^{-\frac{1}{2}}} = 
	L\sum_{q\in \ZZ} 2^{-q} \langle \Dd_q ( \delta \Omega Q_2), \Dd_q \Delta \delta Q\rangle_{L^2_x}
	=	L	\sum_{q\in\ZZ}	2^{-q}	\sum_{i=1}^4\langle \J_{q}^i(Q_2,\,\delta \Omega),\,\Dd_q \Delta \delta Q\rangle_{L^2_x}.
\end{equation*}
Proceeding exactly as for proving \eqref{ev1b}, \eqref{ev2b} and \eqref{ev3b}, with $\delta \Omega$ instead of $\delta D$, the following estimates are obtained:
\begin{equation*}
	L\sum_{q\in \ZZ}2^{-q}\langle \J_{q}^1(Q_2,\,\delta \Omega),\,\Dd_q \Delta \delta Q\rangle_{L^2_x}	
	\lesssim
	\|	\nabla			Q_2		\|_{	L^2_x					  }^2
	\|	\Delta			Q_2		\|_{	L^2_x					  }^2		
	\|			\delta	u		\|_{	\Hh^{-\frac{1}{2}		} }^2+
	C_\nu 
	\|	\nabla	\delta	u		\|_{	\Hh^{-\frac{1}{2}		} }^2+
	C_{\Gamma, L}
	\|	\Delta	\delta	Q		\|_{	\Hh^{-\frac{1}{2}		} }^2,
\end{equation*}
\begin{equation*}
	L\sum_{q\in \ZZ}2^{-q}\langle \J_{q}^2(Q_2,\,\delta \Omega),\,\Dd_q \Delta \delta Q\rangle_{L^2_x}
	\lesssim
	\|		\Delta			Q_2			\|_{	L^2_x				}^2
	\|				\delta	u			\|_{	\Hh^{-\frac{1}{2} } }^2+
	C_{\Gamma, L}
	\|		\Delta	\delta	Q			\|_{	\Hh^{-\frac{1}{2} } }^2,
\end{equation*}
and
\begin{equation*}
	L\sum_{q\in \ZZ}2^{-q}\langle \J_{q}^4(Q_2,\,\delta \Omega),\,\Dd_q \Delta \delta Q\rangle_{L^2_x}
	\lesssim
	\|	\Delta Q_2					\|_{	L^2_x				}^2
	\|	\nabla \delta Q				\|_{	\Hh^{-\frac{1}{2} } }^2+
	C_\nu 
	\|	\nabla \delta u				\|_{	\Hh^{-\frac{1}{2} } }^2.
\end{equation*}
Now observing that
\begin{equation*}
	\Bb_1
	= - L\langle  \delta\Omega\, Q_2, \Delta \delta Q\rangle_{\Hh^{-\frac{1}{2}}}
	= - L\langle  \tr(\delta\Omega\, Q_2), \tr\Delta \delta Q\rangle_{\Hh^{-\frac{1}{2}}}
	= 	L\langle  Q_2 \delta\Omega, \Delta \delta Q\rangle_{\Hh^{-\frac{1}{2}}} = \Bb_2,
\end{equation*}
it remains to control
\begin{equation*}
	\Bb_3 + \Bb_4 + 2L\sum_{q_\in \ZZ}2^{-q} \langle \J_q^3( Q_2, \delta \Omega),\, \Dd_q \Delta\delta Q\rangle_{L^2_x}
	=
	\Bb_3 + \Bb_4 + 2L\sum_{q_\in \ZZ}2^{-q}
	\int_{\RR^2}\trc\{ \Sd_{q-1} Q_2 \Dd_q \delta \Omega  \Dd_q \Delta\delta Q \}.
\end{equation*}

\vspace{0.2cm}
\noindent Now, we turn to  $\Bb_3$:
\begin{align*}
	-\Bb_3 	= L\langle	Q_2 \Delta \delta Q, \nabla\delta u		\rangle_{\dot{H}^{-\frac{1}{2}}}
			= L\sum_{q\in\ZZ}2^{-q}
				\langle \Dd_q(Q_2 \Delta \delta Q), \Dd_q	\nabla	\delta	u
				\rangle_{L^2_x}
			= L\sum_{q\in\ZZ}2^{-q}\sum_{i=1}^4 
				\langle\J_q^i(Q_2,\,\Delta\delta Q),\,\Dd_q\nabla \delta u \rangle_{L^2_x}.
\end{align*}
We remark that $\Bb_3 = -\Aa_3/\xi$, hence the terms related to $i=1,2,4$ are estimated as in \eqref{ev4}, \eqref{ev5} and \eqref{ev6}.
Thus it remains to control
\begin{align*}
	\Bb_4  + L\sum_{q\in \ZZ}2^{-q}
	&\big[ 
		\langle\J_q^3(Q_2,\,\Delta\delta Q),\,\Dd_q\nabla \delta u \rangle_{L^2_x} +
		2\int_{\RR^2}\trc\{ \Sd_{q-1} Q_2 \Dd_q \delta \Omega  \Dd_q \Delta\delta Q \}
	\big]= \\
	&=
	\Bb_4  + L\sum_{q\in \ZZ}2^{-q}
	\int_{\RR^2}
	 \big[
 \trc\{ \Sd_{q-1} Q_2 \Dd_q\Delta\delta Q \Dd_q \nabla\delta u\} + 
	  2\trc\{ \Sd_{q-1} Q_2 \Dd_q \delta \Omega  \Dd_q \Delta\delta Q \}
	 \big].
\end{align*}
Observing that $\Bb_4 = -\Aa_4/\xi$ we argue as for $\Bb_3$, hence it remains to evaluate
\begin{align*}
	L \sum_{q_\in \ZZ}2^{-q}
	&\Big\{ 
		\langle\J_q^3(Q_2,\,\Delta\delta Q) \Dd_q\tr\nabla \delta u \rangle_{L^2_x} +
		\int_{\RR^2}
	 	\big[
	  		\trc\{ \Sd_{q-1} Q_2 \Dd_q\Delta\delta Q \Dd_q\nabla \delta u\} + 
	  		2\trc\{ \Sd_{q-1} Q_2 \Dd_q \delta \Omega  \Dd_q \Delta\delta Q \}
		\big]
	\Big\} =\\&=
	2L \sum_{q_\in \ZZ}2^{-q}
	\int_{\RR^2} 
	\big[
	  	\trc\{ \Sd_{q-1} Q_2 \Dd_q\Delta\delta Q \Dd_q \delta \Omega\} + 
	  	\trc\{ \Sd_{q-1} Q_2 \Dd_q \delta \Omega  \Dd_q \Delta\delta Q \}
	\big]=0,
\end{align*}
where for the cancellation we used that $\Sd_{q-1} Q_2$ and $\Dd_q \Delta \delta Q$ are symmetric while $\Dd_q \delta \Omega$ is skew-symmetric.
\subsubsection{One-logarithmic Estimates} In this subsection, we evaluate the terms of \eqref{energy_uniq} which are related to
the single-logarithmic term of the equality \eqref{uniqueness_inequality}.

\vspace{0.2cm}
\noindent
\underline{Estimate of $\langle \delta Q \trc\{ Q_2 \nabla u_2 \}, \Delta \delta Q\rangle$}. 
Let us fix a positive real number $N>0$ and  split the considered term into two parts, the high and the low frequencies:
\begin{equation*}
	\langle \delta Q \trc\{ Q_2 \nabla u_2 \}, \Delta \delta Q\rangle_{\Hh^{-\frac{1}{2} }} 
	=
	\langle \delta Q \trc\{(\Sd_N Q_2) \nabla u_2 \}, \Delta \delta Q\rangle_{\Hh^{-\frac{1}{2} }} + 
	\langle \delta Q \trc\{(\sum_{q\geq N}\Dd_q Q_2) \nabla u_2 \}, \Delta \delta Q\rangle_{\Hh^{-\frac{1}{2} }}.
\end{equation*}
At first we deal with the low frequencies, observing that
\begin{align*}
	\langle \delta Q \trc\{(\Sd_N Q_2) &\nabla u_2 \}, \Delta \delta Q\rangle_{\Hh^{-\frac{1}{2} }}
	\lesssim
	\|			\delta 	Q \trc\{(\Sd_N Q_2) \nabla u_2 \} 	\|_{	\Hh^{-\frac{1}{2}}	}
	\|	\Delta 	\delta 	Q									\|_{	\Hh^{-\frac{1}{2} }	}\\
	&\lesssim
	\|			\delta Q									\|_{	\Hh^{\frac{1}{2}}	}
	\|			(\Sd_N 	Q_2)				\nabla u_2		\|_{	L^2_x				}
	\|	\Delta 	\delta Q									\|_{	\Hh^{-\frac{1}{2} }	}
	\lesssim
	\|	\nabla \delta Q										\|_{	\Hh^{-\frac{1}{2}}	}
	\|			\Sd_N Q_2									\|_{	L^\infty_x			}
	\| 										\nabla u_2		\|_{	L^2_x				}
	\|	\Delta \delta Q										\|_{	\Hh^{-\frac{1}{2} }	},
\end{align*}
hence, by Theorem \ref{Thm_sqrt_N}, we get
\begin{equation*}
\begin{aligned}
	\langle 	\delta Q \trc\{(\Sd_N Q_2) \nabla u_2 \}, \Delta \delta Q	\rangle_{\Hh^{-\frac{1}{2} }}
	&\lesssim
	\|		\nabla \delta	Q		\|_{	\Hh^{-\frac{1}{2}}	}
	(	\|	 Q_2	 \|_{L^2_x} + \sqrt{N}\| 	\nabla Q_2		\|_{L^2_x}	)
	\| 		\nabla 			u_2		\|_{		L^2_x			}
	\|		\Delta 	\delta 	Q		\|_{	\Hh^{-\frac{1}{2} }	}\\
	&\lesssim
	(1+N)
	\|		\nabla \delta 	Q		\|_{	\Hh^{-\frac{1}{2}}	}^2
	\|		(Q_2,\,\nabla Q_2)		\|_{		L^2_x			}^2
	\| 		\nabla 			u_2		\|_{		L^2_x			}^2
	+
	C_\Gamma
	\|		\Delta \delta 	Q		\|_{	\Hh^{-\frac{1}{2} }	}^2.
\end{aligned}
\end{equation*}
For the high frequencies, we proceed as follows:
\begin{align*}
	\langle \delta Q &\trc\{(\sum_{q\geq N}\Dd_q Q_2) \nabla u_2 \}, \Delta \delta Q\rangle_{\Hh^{-\frac{1}{2} }}
	\lesssim
	\|			\delta Q \trc\{(\sum_{q\geq N}\Dd_q Q_2) \nabla u_2 \} 	\|_{	\Hh^{-\frac{1}{2}}	}
	\|	\Delta 	\delta Q												\|_{	\Hh^{-\frac{1}{2} }	}\\
	&\lesssim
	\|			\delta Q												\|_{	\Hh^{\frac{3}{4}}	}
	\|			(\sum_{q\geq N}\Dd_q Q_2) 				\nabla 	u_2		\|_{	\Hh^{-\frac{1}{4}}	}
	\|	\Delta 	\delta Q												\|_{	\Hh^{-\frac{1}{2} }	}\\
	&\lesssim
	\|				(Q_1,\,Q_2)											\|_{	L^2_x				}^{\frac{1}{4}}
	\|	\nabla 		(Q_1,\,Q_2)											\|_{	L^2_x				}^{\frac{3}{4}}
	\|			\sum_{q\geq N}\Dd_q Q_2									\|_{	\Hh^{\frac{3}{4}}	}
	\| 													\nabla 	u_2		\|_{	L^2_x				}
	\|	\Delta 	\delta Q												\|_{	\Hh^{-\frac{1}{2} }	}\\
	&\lesssim
	\|			  	(Q_1,\,Q_2)											\|_{	L^2_x				}^{\frac{1}{4}}
	\|	\nabla 		(Q_1,\,Q_2)											\|_{	L^2_x				}^{\frac{3}{4}}
	(	
	\sum_{q\geq N}2^{\frac{3}{4}q}
	\|				\Dd_q Q_2											\|_{	L^2_x				}
	)
	\| 													\nabla u_2		\|_{	L^2_x				}
	\|	\Delta \delta Q													\|_{	\Hh^{-\frac{1}{2} }	}\\
	&\lesssim
	\|				(Q_1,\,Q_2)											\|_{	L^2_x				}^{\frac{1}{4}}
	\|	\nabla 		(Q_1,\,Q_2)											\|_{	L^2_x				}^{\frac{3}{4}}	
	(
	\sum_{q\geq N}2^{-\frac{q}{4}}
	\|				\Dd_q \nabla Q_2									\|_{	L^2_x				}
	)
	\| 													\nabla u_2		\|_{	L^2_x				}
	\|	\Delta \delta Q													\|_{	\Hh^{-\frac{1}{2} }	}\\
	&\lesssim
	\|				(Q_1,\,Q_2)											\|_{	L^2_x				}^{\frac{1}{4}}
	\|		\nabla 	(Q_1,\,Q_2)											\|_{	L^2_x				}^{\frac{3}{4}}
	(		\sum_{q\geq N}2^{-\frac{q}{4}}	)
	\|		\nabla 	Q_2													\|_{	L^2_x				}
	\| 													\nabla u_2		\|_{	L^2_x				}
	\|		\Delta \delta Q												\|_{	\Hh^{-\frac{1}{2} }	}\\
	&\lesssim
	2^{-\frac{N}{4}}
	\|				(Q_1,\,Q_2)											\|_{	L^2_x				}^{\frac{1}{4}}
	\|		\nabla 	(Q_1,\,Q_2)											\|_{	L^2_x				}^{\frac{3}{4}}
	\|		\nabla 	Q_2													\|_{	L^2_x				}
	\| 													\nabla u_2		\|_{	L^2_x				}
	\|		\Delta \delta Q												\|_{	\Hh^{-\frac{1}{2} }	}.	
\end{align*}

\bigskip
\noindent
Now, fixing $t>0$ arbitrary, and taking $N=N(t) := \lceil \ln (1+e+1/\Phi(t) )\rceil >0$, where $\lceil \cdot \rceil$ is the ceiling function, we get
\begin{align*}
	\langle \delta Q(t) &\trc\{ Q_2(t) \nabla u_2(t) \}, \Delta \delta Q(t)\rangle_{\Hh^{-\frac{1}{2} }}
	\lesssim
	\|		(Q_2,\,\nabla Q_2)	(t)	\|_{		L^2_x			}^2
	\| 		\nabla 			u_2	(t)	\|_{		L^2_x			}^2	
	\Phi(t)\ln\Big( 1+e + \frac{1}{\Phi(t)}\Big)	+\\ &+
	\|				(Q_1,\,Q_2)	(t)										\|_{	L^2_x				}^{\frac{1}{2}}
	\|		\nabla 	(Q_1,\,Q_2)	(t)										\|_{	L^2_x				}^{\frac{3}{2}}
	\|		\nabla 	Q_2			(t)										\|_{	L^2_x				}^2
	\| 													\nabla u_2	(t)	\|_{	L^2_x				}^2
	\Phi(t)
	+
	C_\Gamma
	\|		\Delta \delta 	Q	(t)	\|_{	\Hh^{-\frac{1}{2} }	}^2.
\end{align*}
Thus we have obtained a one-logarithmic term of \eqref{uniqueness_inequality}.
Similarly, we can handle the estimate of the following elements:
\begin{align*}
	+2L\xi	\langle			\delta 	Q	\,	\trc(	\delta	Q 	\nabla	\delta 	u 	)		,	\Delta	\delta	Q
			\rangle_{	\Hh^{	-\frac{1}{2}	}	}					 										
	+2L\xi	\langle			\delta 	Q  	\, \trc(	\delta	Q	\nabla  		u_2	)		,	\Delta	\delta	Q
			\rangle_{	\Hh^{	-\frac{1}{2}	}	}
	+2L\xi	\langle			\delta 	Q  	\, \trc(	 		Q_2 \nabla	\delta	u 	)		,	\Delta	\delta	Q
			\rangle_{	\Hh^{	-\frac{1}{2}	}	}																	+\\
	+2L\xi	\langle					Q_2	\, \trc(	\delta	Q	\nabla 	\delta 	u 	)		,	\Delta	\delta	Q
			\rangle_{	\Hh^{	-\frac{1}{2}	}	}
	+2L\xi	\langle				 	Q_2	\, \trc(	\delta	Q	\nabla			u_2 )		,	\Delta	\delta	Q
			\rangle_{	\Hh^{	-\frac{1}{2}	}	}																
	-2L\xi		\langle	\delta Q				   \trc\{\delta Q	\Delta	\delta Q	\}	,	\nabla \delta u
				\rangle_{	\Hh^{	-\frac{1}{2}	}	}																-\\	
	-2L\xi		\langle	\delta Q				   \trc\{\delta Q	\Delta		   Q_2	\}	,	\nabla \delta u
				\rangle_{	\Hh^{	-\frac{1}{2}	}	}
	-2L\xi		\langle	\delta Q				   \trc\{	    Q_2	\Delta	\delta Q	\}	,	\nabla \delta u
				\rangle_{	\Hh^{	-\frac{1}{2}	}	}														
	-2L\xi		\langle		   Q_2				   \trc\{\delta Q	\Delta	\delta Q	\}	,	\nabla \delta u
				\rangle_{	\Hh^{	-\frac{1}{2}	}	}																-\\
	-2L\xi		\langle	\delta Q				   \trc\{		Q_2	\Delta		   Q_2	\}	,	\nabla \delta u
				\rangle_{	\Hh^{	-\frac{1}{2}	}	}																					
	-2L\xi		\langle		   Q_2				   \trc\{\delta Q	\Delta		   Q_2	\}	,	\nabla \delta u
				\rangle_{	\Hh^{	-\frac{1}{2}	}	}.
\end{align*}


\subsubsection{Double-Logarithmic Estimates} 
In this subsection, we perform the most challenging estimate. Now, we want to control $\EE_1+\EE_2$, namely
\begin{equation}\label{E1+E2}
\begin{aligned}
	\EE_1 + \EE_2 
	&= 2L\xi 
	\big(\,
		\langle	Q_2	\trc\{	Q_2	\nabla	\delta	u	\},\,	\Delta	\delta Q	\rangle_{\Hh^{-\frac{1}{2}}} -
		\langle	Q_2	\trc\{	Q_2	\Delta	\delta	Q	\},\,	\nabla	\delta u	\rangle_{\Hh^{-\frac{1}{2}}}
	\big)\\
	&=	2L\xi
	\sum_{q\in\ZZ}2^{-q}
	\int_{\RR^2}
	\trc
	\big\{\,	
			\Dd_q(Q_2\trc\{Q_2\nabla \delta	u\})\,	\Dd_q \Delta \delta Q	-
			\Dd_q(Q_2\trc\{Q_2\Delta \delta Q\})\,	\Dd_q \nabla \delta u	\,
	\big\}\\
	&=	2L\xi
	\sum_{i=1}^4
	\sum_{q\in\ZZ}2^{-q}
	\int_{\RR^2}
	\trc
	\big\{\,	
			\J_q^i(Q_2,\,\trc\{Q_2\nabla \delta	u\}\Id)\,	\Dd_q \Delta \delta Q	-
			\J_q^i(Q_2,\,\trc\{Q_2\Delta \delta Q\}\Id)\,	\Dd_q \nabla \delta u	\,
	\big\}.
\end{aligned}
\end{equation}
We we will see that there are elements inside this decomposition that generate the double-logarithmic term in \eqref{uniqueness_inequality}.
We proceed by considering the indexes $i=1,2,3,4$, step by step.

\vspace{0.1cm}
\noindent
\underline{Estimate of $\J_q^1$.} 
We start with the term of \eqref{E1+E2} related to $i=1$, passing trough the following decomposition:
\begin{equation}\label{J1q}
\begin{aligned}
	\sum_{j=1}^4\sum_{|q-q'|\leq 5}
	\int_{\RR_2}
	\trc
	\big\{
	\big(
		[\Dd_q,\,\Sd_{q'-1} Q_2]\trc\{\J^j_{q'}(Q_2,\,\nabla \delta u)\}&\Id
	\big) \,\Dd_q \Delta \delta Q 	- \\&- 
	\big(	 	
		[\Dd_q,\,\Sd_{q'-1} Q_2]\trc\{\J^j_{q'}(Q_2,\,\Delta \delta Q)\}\Id
	\big)\,\Dd_q \nabla \delta u	\,
	\big\}.
\end{aligned}
\end{equation}
When $j=1$, we have
\begin{align*}
	\I^1_1(q,q',q'')&:=\int_{\RR_2}
		\Big\{\Big(\,
		[\Dd_q,\,\Sd_{q'-1} Q_2]\trc\{[\Dd_{q'},\Sd_{q''-1}Q_2]\Dd_{q''}\nabla \delta u)\}\Id \Big)\,\Dd_q \Delta \delta Q	+\\&	 
		\quad\quad\quad\quad\quad\quad\quad\quad\quad\quad\quad-
		\Big([\Dd_q,\,\Sd_{q'-1} Q_2]\trc\{[\Dd_{q'},\Sd_{q''-1}Q_2]\Dd_{q''}\Delta \delta Q)\}\Id\Big)\Dd_q \nabla \delta u	
	\Big\}\\
	&\lesssim
	2^{-q}
	\|	\Sd_{q' -1}	\nabla	Q_2									\|_{L^\infty_x}
	2^{-q'}
	\|	\Sd_{q''-1} \nabla	Q_2									\|_{L^\infty_x}
	\|	\Dd_{q''  }	(	\nabla \delta u,\, \Delta \delta Q	)	\|_{L^2_x}
	\|	\Dd_{q }	(	\nabla \delta u,\, \Delta \delta Q	)	\|_{L^2_x}\\
	&\lesssim
	2^{-q-q'}
	2^{\frac{q'}{2}}
	\|	\Sd_{q' -1}	\nabla	Q_2									\|_{L^4_x}
	2^{\frac{q''}{2}}
	\|	\Sd_{q''-1}	\nabla	Q_2									\|_{L^4_x}
	2^{q''}
	\|	\Dd_{q'' }	(		 	\delta u,\, \nabla \delta Q	)	\|_{L^2_x}
	2^q
	\|	\Dd_{q   }	(		 	\delta u,\, \nabla \delta Q	)	\|_{L^2_x},\\
\end{align*}
which yields
\begin{equation}\label{I11(q,q',q'')est1}
	\I^1_1(q,q',q'')
	\lesssim
	2^{\frac{3q''}{2}-\frac{q'}{2}}
	\|				\nabla	Q_2									\|_{L^2_x}
	\|				\Delta	Q_2									\|_{L^2_x}
	\|	\Dd_{q'' }	(			\delta u,\, \nabla \delta Q	)	\|_{L^2_x}
	\|	\Dd_{q   }	(		 	\delta u,\, \nabla \delta Q	)	\|_{L^2_x}.
\end{equation}
Hence, taking the sum, we deduce that
\begin{equation}\label{I11(q,q',q'')est1b}
\begin{aligned}
	2L\xi\sum_{q\in\ZZ  }&\sum_{|q-q'|\leq 5}\sum_{|q'-q''|\leq 5}2^{-q}\I^1_1(q,q',q'')\lesssim\\
	&\lesssim
	\sum_{q\in\ZZ  }\sum_{|q-q'|\leq 5}\sum_{|q'-q''|\leq 5}2^{-q}2^{\frac{3q''}{2}-\frac{q'}{2}}
	\|				\nabla	Q_2									\|_{L^2_x}
	\|				\Delta	Q_2									\|_{L^2_x}
	\|	\Dd_{q''}	(			\delta u,\, \nabla \delta Q	)	\|_{L^2_x}
	\|	\Dd_{q  }	(		 	\delta u,\, \nabla \delta Q	)	\|_{L^2_x}\\
	&\lesssim
	\|				\nabla	Q_2									\|_{L^2_x}
	\|				\Delta	Q_2									\|_{L^2_x}
	\sum_{q\in\ZZ  }
	\sum_{|q-q''|\leq 10}
	\|	\Dd_{q''}	(			\delta u,\, \nabla \delta Q	)	\|_{L^2_x}
	\|	\Dd_{q  }	(		 	\delta u,\, \nabla \delta Q	)	\|_{L^2_x}\\
	&\lesssim
	\|				\nabla	Q_2									\|_{L^2_x}
	\|				\Delta	Q_2									\|_{L^2_x}
	\|				(			\delta u,\, \nabla \delta Q	)	\|_{L^2_x}^2\\
	&\lesssim
	\|				\nabla	Q_2									\|_{L^2_x}
	\|				\Delta	Q_2									\|_{L^2_x}
	\|				(			\delta u,\, \nabla \delta Q	)	\|_{\Hh^{-\frac{1}{2}}}
	\|				(	\nabla	\delta u,\, \Delta \delta Q	)	\|_{\Hh^{-\frac{1}{2}}}\\
	&\lesssim
	\|				\nabla	Q_2									\|_{L^2_x}^2
	\|				\Delta	Q_2									\|_{L^2_x}^2
	\|				(			\delta u,\, \nabla \delta Q	)	\|_{\Hh^{-\frac{1}{2}}}^2 + 
	C_\nu		
	\|					\nabla	\delta u						\|_{\Hh^{-\frac{1}{2}}}^2 +
	C_{\Gamma, L}
	\|					\Delta	\delta Q						\|_{\Hh^{-\frac{1}{2}}}^2.
\end{aligned}
\end{equation}
Now, when $j=2$ in \eqref{J1q}, we remark that 
\begin{align*}
	\I^1_2(q,q',q'')&:=\int_{\RR_2}
	\trc
	\Big\{\Big(
	[\Dd_q,\,\Sd_{{q'}-1} Q_2]\trc\{(\Sd_{q''-1}Q_2-\Sd_{q'-1}Q_2)\,\Dd_{q'}\Dd_{q''}\nabla\delta u)\}\Id\Big)\,\Dd_q \Delta\delta Q 
		+\\&	
		\quad\quad\quad\quad\quad\quad\quad\quad
	\Big([\Dd_q,\,\Sd_{{q'}-1} Q_2]\trc\{(\Sd_{q''-1}Q_2-\Sd_{q'-1}Q_2)\,\Dd_{q'}\Dd_{q''}\Delta\delta Q)\}\Id\Big)\,\Dd_q \nabla\delta u	
	\Big\}\\
	&\lesssim
	2^{-q}
	\|			\Sd_{q' -1}	\nabla	Q_2									\|_{L^\infty_x}
	\|			\Sd_{q''-1}			Q_2	- \Sd_{q'-1} 			Q_2		\|_{L^\infty_x}
	\|	\Dd_q	\Dd_{q''  }	(	\nabla \delta u,\, \Delta \delta Q	)	\|_{L^2_x}
	\|			\Dd_{q    }	(	\nabla \delta u,\, \Delta \delta Q	)	\|_{L^2_x}\\
	&\lesssim
	2^{-q}
	\|			\Sd_{q' -1}	\nabla	Q_2									\|_{L^\infty_x}
	2^{-q'}
	\|			\Sd_{q''-1}	\nabla	Q_2	- \Sd_{q'-1} 	\nabla	Q_2		\|_{L^\infty_x}
	\|			\Dd_{q''  }	(	\nabla \delta u,\, \Delta \delta Q	)	\|_{L^2_x}
	\|			\Dd_{q    }	(	\nabla \delta u,\, \Delta \delta Q	)	\|_{L^2_x}\\
	&\lesssim
	2^{-q-q'}
	2^{\frac{q'}{2}}
	\|			\Sd_{q' -1}	\nabla	Q_2									\|_{L^4_x}
	2^{\frac{q''}{2}}
	\|			\Sd_{q''-1} \nabla	Q_2									\|_{L^4_x}
	2^{q''}
	\|			\Dd_{q''  }	(		 	\delta u,\, \nabla \delta Q	)	\|_{L^2_x}
	2^{q}
	\|			\Dd_{q    }	(		 	\delta u,\, \nabla \delta Q	)	\|_{L^2_x}\\
	&\lesssim
	2^{\frac{3q''}{2}-\frac{q'}{2}}
	\|				\nabla	Q_2									\|_{L^2_x}
	\|				\Delta	Q_2									\|_{L^2_x}
	\|	\Dd_{q''}	(			\delta u,\, \nabla \delta Q	)	\|_{L^2_x}
	\|	\Dd_{q }	(		 	\delta u,\, \nabla \delta Q	)	\|_{L^2_x},
\end{align*}
which is equivalent to \eqref{I11(q,q',q'')est1}. Hence, proceeding as in \eqref{I11(q,q',q'')est1b}, we get
\begin{equation*}
	2L\xi\sum_{q\in\ZZ  }\sum_{\substack{|q-q'|\leq 5\\|q'-q''|\leq 5} }2^{-q}\I^1_2(q,q',q'')
	\lesssim
	\|				\nabla	Q_2									\|_{L^2_x}^2
	\|				\Delta	Q_2									\|_{L^2_x}^2
	\|				(			\delta u,\, \nabla \delta Q	)	\|_{\Hh^{-\frac{1}{2}}}^2 + 
	C_\nu		
	\|					\nabla	\delta u						\|_{\Hh^{-\frac{1}{2}}}^2 +
	C_{\Gamma, L}
	\|					\Delta	\delta Q						\|_{\Hh^{-\frac{1}{2}}}^2.
\end{equation*}
Concerning the term of \eqref{J1q} related to $j=4$, we have
\begin{equation}\label{ineq3}
\begin{aligned}
	\I^1_4(q,q',q'')&:=\int_{\RR_2}
	\trc
	\Big\{\Big(
	[\Dd_q,\,\Sd_{q'-1} Q_2]\trc\{\Dd_{q'}( \,\Dd_{q''}Q_2 \Sd_{q''+2}\nabla\delta u\,)\}\Id\Big)\Dd_q \Delta\delta Q 
		-\\&	
		\quad\quad\quad\quad\quad\quad\quad\quad\quad\quad\quad-
	\Big([\Dd_q,\,\Sd_{q'-1} Q_2]\trc\{\Dd_{q'}(\,\Dd_{q''}Q_2 \Sd_{q''+2}\Delta\delta  Q\,)\}\Id\Big)\Dd_q \nabla\delta u	
	\Big\}\\
	&\lesssim
	2^{-q}
	\|			\Sd_{q' -1}	\nabla	Q_2									\|_{L^\infty_x}
	\|	\Dd_q	(\, \Dd_{q''  }Q_2	\Sd_{q''+2}	\nabla \delta u,\;
					\Dd_{q''  }Q_2	\Sd_{q''+2} \Delta \delta Q\,)		\|_{L^2_x}
	\|			\Dd_{q    }	(	\nabla \delta u,\, \Delta \delta Q	)	\|_{L^2_x}\\
	&\lesssim
	2^{-q}
	\|			\Sd_{q' -1}	\nabla	Q_2									\|_{L^\infty_x}
	2^{q}
	\|	\Dd_q	(\, \Dd_{q''  }Q_2	\Sd_{q''+2}	\nabla \delta u,\;
					\Dd_{q''  }Q_2	\Sd_{q''+2} \Delta \delta Q\,)		\|_{L^1_x}
	\|			\Dd_{q    }	(	\nabla \delta u,\, \Delta \delta Q	)	\|_{L^2_x}\\	
	&\lesssim
	\|			\Sd_{q' -1}	\nabla	Q_2									\|_{L^\infty_x}
	\|			\Dd_{q''  }			Q_2									\|_{L^2_x}
	\|			\Sd_{q''+2}	(	\nabla \delta u,\, \Delta \delta Q	)	\|_{L^2_x}
	\|			\Dd_{q    }	(	\nabla \delta u,\, \Delta \delta Q	)	\|_{L^2_x}\\
	&\lesssim
	\|			\Sd_{q' -1}	\nabla	Q_2									\|_{L^\infty_x}
	2^{-2q''}
	\|			\Dd_{q''  }	\Delta	Q_2									\|_{L^2_x}
	\|			\Sd_{q''+2}	(	\nabla \delta u,\, \Delta \delta Q	)	\|_{L^2_x}
	\|			\Dd_{q    }	(	\nabla \delta u,\, \Delta \delta Q	)	\|_{L^2_x}\\
	&\lesssim
	2^{-2q''}
	2^{q'}
	\|			\Sd_{q' -1}	\nabla	Q_2									\|_{L^2_x}
	\|			\Dd_{q''  }	\Delta	Q_2									\|_{L^2_x}
	2^{q''}
	\|			\Sd_{q''+2}	(		 	\delta u,\, \nabla \delta Q	)	\|_{L^2_x}
	\|			\Dd_{q    }	(	\nabla 	\delta u,\, \Delta \delta Q	)	\|_{L^2_x}\\
	&\lesssim
	2^{q'-q''}
	\|				\nabla	Q_2									\|_{L^2_x}
	\|				\Delta	Q_2									\|_{L^2_x}
	\|			\Sd_{q''+2}	(		   \delta u,\, \nabla \delta Q	)	\|_{L^2_x}
	\|			\Dd_{q    }	(	\nabla \delta u,\, \Delta \delta Q	)	\|_{L^2_x}\\
\end{aligned}
\end{equation}
Hence
\begin{align}
	2&L\xi\sum_{q\in\ZZ}\;\sum_{|q-q'|\leq 5}\; \sum_{q''\geq q'-5}				2^{-q}\I^1_4(q,q',q'') \lesssim				\nonumber\\
	&\lesssim
	\|				\nabla	Q_2									\|_{L^2_x}
	\|				\Delta	Q_2									\|_{L^2_x}
	\sum_{q\in\ZZ}2^{-\frac{q}{2}} 
	\|			\Dd_{q    }	(	\nabla \delta u,\, \Delta \delta Q	)	\|_{L^2_x}\;
	\sum_{[q-q'|\leq 5}
	2^{\frac{q'-q}{2}}
	\sum_{q''\geq q' -5 }2^{\frac{q'-q''}{2}}
	2^{-\frac{q''}{2}}
	\|			\Sd_{q''+2}	(		   \delta u,\, \nabla \delta Q	)	\|_{L^2_x}											\nonumber\\
	&\lesssim
	\|				\nabla	Q_2									\|_{L^2_x}
	\|				\Delta	Q_2									\|_{L^2_x}
	\sum_{q\in\ZZ}2^{-\frac{q}{2}} 
	\|			\Dd_{q    }	(	\nabla \delta u,\, \Delta \delta Q	)	\|_{L^2_x}\;
	\sum_{q''\geq q -10 }2^{\frac{q-q''}{2}}
	2^{-\frac{q''}{2}}
	\|			\Sd_{q''+2}	(		   \delta u,\, \nabla \delta Q	)	\|_{L^2_x}											\label{ineq3b}\\
	&\lesssim
	\|				\nabla	Q_2									\|_{L^2_x}
	\|				\Delta	Q_2									\|_{L^2_x}
	\|			(	\nabla \delta u,\, \Delta \delta Q	)		\|_{\Hh^{-\frac{1}{2}}}
	\Big[
	\sum_{q\in\ZZ}\;
		\Big|
			\sum_{q-q''\leq 10 }2^{\frac{q-q''}{2}}
			2^{-\frac{q''+2}{2}}
			\|	\Sd_{q''+2}	(		   \delta u,\, \nabla \delta Q	)	\|_{L^2_x}
		\Big|^2
	\Big]^{\frac{1}{2}}																										\nonumber\\
	&\lesssim
	\|				\nabla	Q_2									\|_{L^2_x}
	\|				\Delta	Q_2									\|_{L^2_x}
	\|			(	\nabla \delta u,\, \Delta \delta Q	)		\|_{\Hh^{-\frac{1}{2}}}
	\Big(
		\sum_{q\leq 10 }
		2^{\frac{q}{2}}
	\Big)
	\Big(
		\sum_{q\in\ZZ}
		2^{-q}
		\|	\Sd_{q}	(		   \delta u,\, \nabla \delta Q	)	\|_{L^2_x}^2
	\Big)^{\frac{1}{2}}																										\nonumber\\
	&\lesssim
	\|				\nabla	Q_2									\|_{L^2_x}
	\|				\Delta	Q_2									\|_{L^2_x}
	\|			(	\nabla \delta u,\, \Delta \delta Q	)		\|_{\Hh^{-\frac{1}{2}}}
	\|			(		   \delta u,\, \nabla \delta Q	)		\|_{\Hh^{-\frac{1}{2}}}										\nonumber\\
	&\lesssim
	\|				\nabla	Q_2									\|_{L^2_x}^2
	\|				\Delta	Q_2									\|_{L^2_x}^2
	\|				(			\delta u,\, \nabla \delta Q	)	\|_{\Hh^{-\frac{1}{2}}}^2 + 
	C_\nu		
	\|					\nabla	\delta u						\|_{\Hh^{-\frac{1}{2}}}^2 +
	C_{\Gamma, L}
	\|					\Delta	\delta Q						\|_{\Hh^{-\frac{1}{2}}}^2.									\nonumber
\end{align}
\noindent Concerning \eqref{J1q}, it remains to control the term related to $j=3$. 
We fix $0<\ee \leq 5/6$ and we consider the low frequencies $q\leq N$, for some suitable positive $N\geq 1$ (so that $1+\sqrt{N}< 2\sqrt{N}$):
\begin{align*}
	\I^1_3(q,q'):&=\int_{\RR_2}
	\trc
	\Big\{\Big(
	[\Dd_q,\,\Sd_{q'-1} Q_2]\trc\{\Sd_{q'-1}Q_2\,\Dd_{q'}\nabla\delta u\}\Id\Big)\,\Dd_q \Delta\delta Q 
		- \\&\quad\quad\quad\quad\quad\quad\quad\quad\quad-
	\Big([\Dd_q,\,\Sd_{q'-1} Q_2]\trc\{\Sd_{q'-1}Q_2\,\Dd_{q'}\Delta\delta Q\}\Id\Big)\,\Dd_q \nabla\delta u	
	\Big\}\\
	&\lesssim
	2^{-q}
	\|\Sd_{q'-1} \nabla Q_2																\|_{L^{\frac{2}{\ee}}_x		}
	\|\Sd_{q'-1}Q_2\,\Dd_{q'}\nabla\delta u, \,\Sd_{q'-1}Q_2\,\Dd_{q'}\Delta\delta Q	\|_{L^{\frac{2}{1-\ee}}_x	}
	\|\Dd_q(\,\nabla \delta u,\,\Delta \delta Q)										\|_{L^2_x					}\\
	&\lesssim
	\|	\Sd_{q'-1} \nabla Q_2										\|_{	L^{\frac{2}{\ee}}_x		}
	\|	\Sd_{q'-1}		  Q_2										\|_{	L^\infty_x				}	
	2^{-q}
	\|	\Dd_{q'}(\,	\nabla	\delta u,\,	\Delta	\delta Q	\,)		\|_{	L^{\frac{2}{1-\ee}}_x	}
	\|	\Dd_q	(\,	\nabla	\delta u,\,	\Delta	\delta Q	\,)		\|_{	L^2_x					}.
\end{align*}

Thanks to Theorem \ref{Thm_sqrt_N}, we get
\begin{equation*}
	\|	\Sd_{q'-1}		  Q_2										\|_{	L^\infty_x				}
	\lesssim
	(1+ \sqrt{q'-1})\|(Q_2,\,\nabla Q_2)\|_{L^2_x}
	\lesssim
	(1+ \sqrt{N})\|(Q_2,\,\nabla Q_2)\|_{L^2_x}
	\lesssim
	\sqrt{N}\|(Q_2,\,\nabla Q_2)\|_{L^2_x},
\end{equation*}
hence $\I^1_3(q,q')$ is bounded by
\begin{equation}\label{ineq2db}
\begin{aligned}
	\I^1_3(q,q')
	\lesssim
	\sqrt{N}
	\|	(Q_2,\,	\nabla Q_2)										\|_{	L^{\frac{2}{\ee}}_x		}
	\|	(Q_2,\,	\nabla Q_2)										\|_{	L^2_x					}
	\|	\Dd_{q'}(\,			\delta u,\,	\nabla	\delta Q	\,)		\|_{	L^{\frac{2}{1-\ee}}_x	}
	\|	\Dd_q	(\,	\nabla 	\delta u,\,	\Delta	\delta Q	\,)		\|_{	L^2_x					}.
\end{aligned}
\end{equation}
Now, we will need the following inequality, which will finally lead to the delicate double-logarithmic estimate:
\begin{equation*}
	\|	(Q_2,\,	\nabla Q_2)								\|_{	L^{\frac{2}{\ee}}_x		}
	\leq
	\frac{C}{\sqrt{\ee}}
	\|	(Q_2,\,	\nabla Q_2)								\|_{	L^2_x				}^{\ee}
	\|	(\nabla Q_2,\,	\Delta Q_2)						\|_{	L^2_x				}^{1-\ee},
\end{equation*}
This is a consequence of Lemma \ref{lemmaappx}, imposing $p=1/\ee$, where $C$ is a positive constant independent of $\ee$ and $Q_2$. We will see that the double-logarithmic term comes out of a suitable choice of $\ee$  in terms of $N$.
Again, using Lemma \ref{lemmaappx} we have
\begin{align*}
	\|	\Dd_{q'}(\,			\delta u,\,	\nabla	\delta Q	\,)		\|_{	L^{\frac{2}{1-\ee}}_x	}
	&\leq
	\frac{C}{1-\ee}
	\|	\Dd_{q'}(\,			\delta u,\,	\nabla	\delta Q	\,)		\|_{	L^2_x				}^{1-\ee}
	\|	\Dd_{q'}(\,	\nabla 	\delta u,\,	\Delta	\delta Q	\,)		\|_{	L^2_x				}^{\ee}\\
	&\leq
	6 C
	\|	\Dd_{q'}(\,			\delta u,\,	\nabla	\delta Q	\,)		\|_{	L^2_x				}^{1-\ee}
	\|	\Dd_{q'}(\,	\nabla 	\delta u,\,	\Delta	\delta Q	\,)		\|_{	L^2_x				}^{\ee},
\end{align*}
since $\ee\leq 5/6$. Hence \eqref{ineq2db} becomes

\begin{equation}\label{ineq2d}
\begin{aligned}
	\I^1_3(q,q')
	&\lesssim
	\sqrt{\frac{N}{\ee}}
	\|	(Q_2,\,	\nabla Q_2)								\|_{	L^2_x				}^{1+\ee}
	\|	(\nabla Q_2,\,	\Delta Q_2)						\|_{	L^2_x				}^{1-\ee}
	{\scriptstyle \times}\\ &\quad\quad\quad\quad\quad\quad\quad\quad\quad\quad\quad\quad	{\scriptstyle \times}
	\|	\Dd_{q'}(\,			\delta u,\,	\nabla	\delta Q	\,)		\|_{	L^2_x				}^{1-\ee}
	\|	\Dd_{q'}(\,	\nabla 	\delta u,\,	\Delta	\delta Q	\,)		\|_{	L^2_x				}^{\ee}
	\|	\Dd_q	(\,	\nabla 	\delta u,\,	\Delta	\delta Q	\,)		\|_{	L^2_x				},
\end{aligned}
\end{equation}
thus, since $ab\leq a^{2/(1-\ee)} + b^{2/(1+\ee)}$, we deduce

\begin{align*}
	\I^1_3(q,q')
	&\lesssim
	\Big(\frac{N}{\ee}\Big)^{\frac{1}{1-\ee}}
	\|	(Q_2			,\,\nabla Q_2)					\|_{	L^2_x				}^{\frac{2(1+\ee)}{1-\ee}}
	\|	(\nabla Q_2	,\,\Delta Q_2)					\|_{	L^2_x				}^2
	\|	\Dd_{q'}(\,			\delta u,\,	\nabla	\delta Q	\,)		\|_{	L^2_x				}^2 + \\
	&\quad\quad\quad\quad\quad\quad\quad\quad\quad\quad\quad\quad+
	\min\{C_\nu,\,C_\Gamma \}
	\|	\Dd_{q'}(\,	\nabla 	\delta u,\,	\Delta	\delta Q	\,)		\|_{	L^2_x				}^{\frac{2\ee}{1+\ee}}
	\|	\Dd_q	(\,	\nabla 	\delta u,\,	\Delta	\delta Q	\,)		\|_{	L^2_x				}^{\frac{2}{1+\ee}}\\
	&\lesssim
	\Big(\frac{N}{\ee}\Big)^{\frac{1}{1-\ee}}
	\|	 	(Q_2			,\,\nabla Q_2)										\|_{	L^2_x				}^{\frac{2(1+\ee)}{1-\ee}}
	\|	 	(\nabla Q_2	,\,\Delta Q_2)										\|_{	L^2_x				}^2
	\|	\Dd_{q'}(\,			\delta u,\,	\nabla	\delta Q	\,)		\|_{	L^2_x				}^2 + \\
	&\quad\quad\quad\quad\quad\quad\quad\quad\quad\quad\quad\quad+
	\min\{C_\nu,\,C_\Gamma \}\Big(
	\|	\Dd_{q'}(\,	\nabla 	\delta u,\,	\Delta	\delta Q	\,)		\|_{	L^2_x				}^2+
	\|	\Dd_q	(\,	\nabla 	\delta u,\,	\Delta	\delta Q	\,)		\|_{	L^2_x				}^2\Big).
\end{align*}
Imposing $\ee = (1+\ln N)^{-1}$ and observing that $\frac{1}{1-\ee}=1+1/\ln N$
\begin{equation*}
	N^{\frac{1}{1-\ee}} = N\, N^{\frac{1}{\ln N}} = eN, \quad
	\ee^{-\frac{1}{1-\ee}} = \ee^{-1}\ee^{-\frac{\ee}{1-\ee}} = 
	(1+\ln N)e^{\frac{\ee}{1-\ee}\ln \frac{1}{\ee}}\lesssim (1+\ln N),
\end{equation*}
we obtain:
\begin{align*}
	\I^1_3(q,q')
	&\lesssim
	N(1+\ln N)
	\max\big\{
	\|	(Q_2			,\,\nabla Q_2)								\|_{	L^2_x				}^{6}, 1\big\}
	\|	(\nabla Q_2	,\,\Delta Q_2)								\|_{	L^2_x				}^2
	\|	\Dd_{q'}(\,			\delta u,\,	\nabla	\delta Q	\,)		\|_{	L^2_x				}^2 + \\
	&\quad\quad\quad\quad\quad\quad\quad\quad\quad\quad\quad\quad+
	\min\Big\{C_\nu,\,C_{\Gamma,L}\Big\}\Big(
	\|	\Dd_{q'}(\,	\nabla 	\delta u,\,	\Delta	\delta Q	\,)		\|_{	L^2_x				}^2+
	\|	\Dd_q	(\,	\nabla 	\delta u,\,	\Delta	\delta Q	\,)		\|_{	L^2_x				}^2\Big),
\end{align*}	
which yields
\begin{align*}
	\sum_{q\leq N}\sum_{|q-q'|\leq 5}2^{-q}\I^1_3(q,q')
	\lesssim
	N(1+\ln N)
	\max\big\{
	\|			 	(Q_2			,\,\nabla Q_2)					\|_{	L^2_x				}^{6}, 1\big\}&
	\|			 	(\nabla Q_2	,\,\Delta Q_2)					\|_{	L^2_x				}^2
	\|			(\,			\delta u,\,	\nabla	\delta Q	\,)		\|_{\Hh^{-\frac{1}{2}}		}^2 + \\&+
	C_\nu		
	\|					\nabla	\delta u						\|_{\Hh^{-\frac{1}{2}}}^2 +
	C_{\Gamma, L}
	\|					\Delta	\delta Q						\|_{\Hh^{-\frac{1}{2}}}^2.
\end{align*}
For the high frequencies, namely for $q>N\geq 1$, we proceed as follows:
\begin{align*}
	\I^1_3(q,q')
	&\lesssim 
	2^{-q}
	\|\Sd_{q'-1} \nabla Q_2																\|_{L^{\infty}_x			}
	\|\Sd_{q'-1}Q_2\,\Dd_{q'}\nabla\delta u, \,\Sd_{q'-1}Q_2\,\Dd_{q'}\Delta\delta Q	\|_{L^{2}_x					}
	\|\Dd_q(\,\nabla \delta u,\,\Delta \delta Q)										\|_{L^2_x					}\\
	&\lesssim
	2^{-q}(1+\sqrt{q'})
	\|	(\nabla Q_2,\,  \Delta Q_2)													\|_{L^{2}_x					}
	\|	\Sd_{q'-1} 		   Q_2														\|_{L^\infty_x				}
	\|	\Dd_{q'}(\,	\nabla 	\delta u,\,	\Delta	\delta Q	\,)							\|_{	L^2_x				}
	\|	\Dd_q	(\,	\nabla 	\delta u,\,	\Delta	\delta Q	\,)							\|_{	L^2_x				}\\
	&\lesssim
	2^{-q}(1+\sqrt{q'})^2
	\|	(\nabla Q_2,\,  \Delta Q_2)													\|_{L^2_x			}
	\|	(Q_2,\,		\nabla  Q_2	)													\|_{L^2_x			}
	\|	\Dd_{q'}(\,	\nabla 	\delta u,\,	\Delta	\delta Q	\,)							\|_{	L^2_x				}
	\|	\Dd_q	(\,	\nabla 	\delta u,\,	\Delta	\delta Q	\,)							\|_{	L^2_x				}\\
	&\lesssim
	q'
	\|	(\nabla Q_2,\,  \Delta Q_2)													\|_{L^2_x			}
	\|	(Q_2,\,		\nabla  Q_2	)													\|_{L^2_x			}
	\|	\Dd_{q'}(\,	\nabla 	\delta u,\,	\Delta	\delta Q	\,)							\|_{	L^2_x				}
	\|	\Dd_q	(\,		 	\delta u,\,	\nabla	\delta Q	\,)							\|_{	L^2_x				},
\end{align*}
which implies
\begin{align*}
	\sum_{q>N}\sum_{|q-q'|\leq 5}&2^{-q}\I^1_3(q,q')\\
	&\lesssim
	\sum_{q>N}\sum_{|q-q'|\leq 5}2^{-2q}
	q'
	\|		(\nabla Q_2,\,  \Delta Q_2)												\|_{L^2_x				}
	\|		(Q_2,\,		\nabla  Q_2	)												\|_{L^2_x				}
	\|	\Dd_{q'}(\,	\nabla 	\delta u,\,	\Delta	\delta Q	\,)							\|_{	L^2_x			}
	\|	\Dd_q	(\,		 	\delta u,\,	\nabla	\delta Q	\,)							\|_{	L^2_x			}\\
	&\lesssim
	\|			(\nabla Q_2,\,  \Delta Q_2)											\|_{L^2_x				}
	\|			(Q_2,\,		\nabla  Q_2	)											\|_{L^2_x				}
	\|			(\,		 	\delta u,\,	\nabla	\delta Q	\,)							\|_{L^2_x	}
	\|			(\,	\nabla 	\delta u,\,	\Delta	\delta Q	\,)							\|_{\Hh^{-\frac{1}{2}}	}
	\sum_{q>N}\sum_{|q-q'|\leq 5}2^{-\frac{3}{2}q+\frac{1}{2}q'}q'\\
	&\lesssim
	2^{-\frac{N}{2}}
	\|			(\nabla Q_2,\,  \Delta Q_2)											\|_{L^2_x				}
	\|			(Q_2,\,		\nabla  Q_2	)											\|_{L^2_x				}
	\|			(\,		 	\delta u,\,	\nabla	\delta Q	\,)							\|_{L^2_x				}
	\|			(\,	\nabla 	\delta u,\,	\Delta	\delta Q	\,)							\|_{\Hh^{-\frac{1}{2}}	}.
\end{align*}
Summarizing, we get
\begin{equation}\label{ineq2}
\begin{aligned}
	\sum_{q\in\ZZ}&\sum_{|q-q'|\leq 5}2^{-q}\I^1_3(q,q')\lesssim
	N(1+\ln N)
	\max\big\{
	\|			 (Q_2,\,	\nabla Q_2)									\|_{	L^2_x				}^{6}, 1\big\}
	\|			 (\nabla Q_2,\,	\Delta Q_2	)						\|_{	L^2_x				}^2
	\|			(\,			\delta u,\,	\nabla	\delta Q	\,)		\|_{\Hh^{-\frac{1}{2}}		}^2 + \\&+
	2^{-N}
	\|			(\nabla Q_2,\,  \Delta Q_2)												\|_{L^2_x				}^2
	\|			(Q_2,\,		\nabla  Q_2	)												\|_{L^2_x				}^2
	\|			(		 	 u_1,\,u_2,\,	\nabla	Q_1,\, \nabla Q_2	)				\|_{L^2_x				}^2
	+
	C_\nu		
	\|					\nabla	\delta u						\|_{\Hh^{-\frac{1}{2}}}^2 +
	C_{\Gamma, L}
	\|					\Delta	\delta Q						\|_{\Hh^{-\frac{1}{2}}}^2.
\end{aligned}
\end{equation}

Choosing $N=N(t):=\lceil \ln (1+e+1/\Phi(t))\rceil $ (thus $\ee<1/(1+\ln \ln\{1+e\})<5/6$) where with $\lceil \cdot \rceil$ we denote the ceiling function, relation \eqref{ineq2} implies
\begin{equation}\label{ineq2b}
\begin{aligned}
	&\sum_{q\in\ZZ}\sum_{|q-q'|\leq 5}2^{-q}\I^3_1(q,q')\\
	&\lesssim
	\max\big\{
	\|			 	(Q_2,\,\nabla Q_2)								\|_{	L^2_x				}^{6}, 1\big\}
	\|			 	(\nabla Q_2,\,\Delta Q_2	)						\|_{	L^2_x				}^2
	\|			(\,			\delta u,\,	\nabla	\delta Q	\,)		\|_{\Hh^{-\frac{1}{2}}		}^2
	\ln\Big( e +  \frac{1}{\Phi(t)}\Big) \Big(1 + \ln \ln\Big(e + \frac{1}{\Phi(t)}\Big)\Big)
	 + \\&+
	\|			(\nabla Q_2,\,  \Delta Q_2)												\|_{L^2_x				}^2
	\|			(Q_2,\,		\nabla  Q_2	)												\|_{L^2_x				}^2
	\|			(	 u_1,\,u_2,\,\nabla	Q_1,\, \nabla Q_2	)							\|_{L^2_x				}^2
	\Phi(t)+
	C_\nu		
	\|					\nabla	\delta u						\|_{\Hh^{-\frac{1}{2}}}^2 +
	C_{\Gamma, L}
	\|					\Delta	\delta Q						\|_{\Hh^{-\frac{1}{2}}}^2.
\end{aligned}
\end{equation}
\noindent
\underline{Estimate of $\J_q^2$} Now, we handle the term of \eqref{E1+E2} related to $i=2$, namely
\begin{equation}\label{J2q}
\begin{aligned}
	\sum_{j=1}^4\sum_{|q-q'|\leq 5}
	\int_{\RR_2}
	\trc
	\big\{\,
		(\,\Sd_{q'-1} Q_2-\Sd_{q-1} Q_2\,)&\trc\{\Dd_{q}\J^j_{q'}(Q_2,\,\nabla \delta u)\}\Dd_q \Delta \delta Q 	+\\&
		(\,\Sd_{q'-1} Q_2-\Sd_{q-1} Q_2\,) \trc\{\Dd_{q}\J^j_{q'}(Q_2,\,\Delta \delta Q)\}\Dd_q \nabla \delta u	\,
	\big\}.
\end{aligned}
\end{equation}
When $j=1$, we have 
\begin{align*}
	\I_1^2&(q,q',q''):=
	\int_{\RR_2}
	\trc
	\Big\{
		(\,\Sd_{q'-1} Q_2-\Sd_{q-1} Q_2\,)
		\trc\{\Dd_{q}\big(\,[\Dd_{q'},\Sd_{q''-1}Q_2]\Dd_{q''}\nabla \delta u\,\big)\}\Dd_q \Delta \delta Q	+\\&	 
		\quad\quad\quad\quad\quad\quad\quad\quad\quad\quad-
		(\,\Sd_{q'-1} Q_2-\Sd_{q-1} Q_2\,)
		\trc\{\Dd_{q}\big(\,[\Dd_{q'},\Sd_{q''-1}Q_2]\Dd_{q''}\Delta \delta Q\,\big)\}\Dd_q \nabla \delta u	
	\Big\}\\
	&\lesssim
	\|	\Sd_{q'-1} Q_2-\Sd_{q-1} Q_2																  \|_{L^\infty_x}
	\|	\Dd_{q}\big(\,[\Dd_{q'},\Sd_{q''-1}Q_2]\Dd_{q''}( \nabla \delta u,\, \Delta \delta Q )\,\big) \|_{L^2_x}
	\|	\Dd_{q }	(	\nabla \delta u,\, \Delta \delta Q	)										  \|_{L^2_x}\\
	&\lesssim
	2^{-\frac{q}{2}}
	\|	\Sd_{q'-1}\nabla Q_2-\Sd_{q-1}\nabla Q_2													  \|_{L^4_x}
	2^{\frac{q}{2}}
	\|	\Dd_{q}\big(\,[\Dd_{q'},\Sd_{q''-1}Q_2]\Dd_{q''}( \nabla \delta u,\, \Delta \delta Q )\,\big) \|_{L^\frac{4}{3}_x}
	\|	\Dd_{q }	(	\nabla \delta u,\, \Delta \delta Q	)										  \|_{L^2_x}\\	
	&\lesssim
	\|	\nabla	Q_2												\|_{L^4_x}
	2^{-q'}
	\|	\Sd_{q''-1}	\nabla	Q_2									\|_{L^4_x}
	\|	\Dd_{q'' }	(	\nabla 	\delta u,\, \Delta \delta Q	)	\|_{L^2_x}
	\|	\Dd_{q   }	(	\nabla	\delta u,\, \Delta \delta Q	)	\|_{L^2_x}\\
	&\lesssim
	2^{-q'+q''+q}
	\|				\nabla	Q_2									\|_{L^2_x}
	\|				\Delta	Q_2									\|_{L^2_x}
	\|	\Dd_{q'' }	(			\delta u,\, \nabla \delta Q	)	\|_{L^2_x}
	\|	\Dd_{q   }	(			\delta u,\, \nabla \delta Q	)	\|_{L^2_x}.
\end{align*}
Since $|q-q'|\leq 5$ and $|q'-q''|\leq 5$ then $-q'+q''+q \simeq 3q''/2 -q'/2$, so that the last inequality is equivalent to \eqref{I11(q,q',q'')est1}. 
Hence, proceeding as in \eqref{I11(q,q',q'')est1b}, we get
\begin{equation*}
	2L\xi\sum_{q\in\ZZ  }\sum_{\substack{|q-q'|\leq 5\\|q'-q''|\leq 5} }2^{-q}\I_1^2(q,q',q'')
	\lesssim
	\|				\nabla	Q_2									\|_{L^2_x}^2
	\|				\Delta	Q_2									\|_{L^2_x}^2
	\|				(			\delta u,\, \nabla \delta Q	)	\|_{\Hh^{-\frac{1}{2}}}^2 + 
	C_\nu		
	\|					\nabla	\delta u						\|_{\Hh^{-\frac{1}{2}}}^2 +
	C_{\Gamma, L}
	\|					\Delta	\delta Q						\|_{\Hh^{-\frac{1}{2}}}^2.
\end{equation*}	
When $j=2$, we observe that
\begin{align*}
	\I_2^2&(q,q',q''):=
	\int_{\RR_2}
	\trc
	\Big\{
		(\,\Sd_{q'-1} Q_2-\Sd_{q-1} Q_2\,)
		\trc\{\,(\Sd_{q''-1}Q_2-\Sd_{q'-1}Q_2) \Dd_{q'}\Dd_{q''}\nabla \delta u\,\}\Dd_q \Delta \delta Q	+\\&	 
		\quad\quad\quad\quad\quad\quad\quad\quad\quad\quad-
		(\,\Sd_{q'-1} Q_2-\Sd_{q-1} Q_2\,)
		\trc\{(\Sd_{q''-1}Q_2-\Sd_{q'-1}Q_2)\Dd_{q'}\Dd_{q''}\Delta \delta Q\,\}\Dd_q \nabla \delta u	
	\Big\}\\
	&\lesssim
	\|	\Sd_{q'-1} Q_2-\Sd_{q-1} Q_2															\|_{L^\infty_x}
	\|	\Sd_{q''-1}Q_2-\Sd_{q'-1}Q_2															\|_{L^\infty_x}	
	\|	\Dd_{q'}\Dd_{q''}( \nabla \delta u,\, \Delta \delta Q ) \|_{L^2_x}						\|_{L^2_x}
	\|	\Dd_{q }	(	\nabla \delta u,\, \Delta \delta Q	)									\|_{L^2_x}\\
	&\lesssim
	2^{-\frac{q+q'}{2}}
	\|	\Sd_{q' -1}\nabla Q_2-\Sd_{q -1}\nabla Q_2												\|_{L^4_x}
	\|	\Sd_{q''-1}\nabla Q_2-\Sd_{q'-1}\nabla Q_2												\|_{L^4_x}
	\|	\Dd_{q''}    ( 	\nabla \delta u,\, \Delta \delta Q  )									\|_{L^2_x}
	\|	\Dd_{q }	(	\nabla \delta u,\, \Delta \delta Q	)									\|_{L^2_x}\\	
	&\lesssim
	2^{\frac{q'}{2}+\frac{q}{2}}
	\|	\nabla	Q_2												\|_{L^4_x}^2
	\|	\Dd_{q'' }	(		 	\delta u,\, \nabla \delta Q	)	\|_{L^2_x}
	\|	\Dd_{q   }	(			\delta u,\, \Delta \delta Q	)	\|_{L^2_x}\\
	&\lesssim
	2^{\frac{q'}{2}+\frac{q}{2}}
	\|				\nabla	Q_2									\|_{L^2_x}
	\|				\Delta	Q_2									\|_{L^2_x}
	\|	\Dd_{q'  }	(			\delta u,\, \nabla \delta Q	)	\|_{L^2_x}
	\|	\Dd_{q   }	(			\delta u,\, \nabla \delta Q	)	\|_{L^2_x}.\\
\end{align*}
Since $|q-q'|\leq 5$ and $|q'-q''|\leq 5$ then $q'/2+q/2 \simeq 3q''/2 -q'/2$, so that the last inequality is equivalent to \eqref{I11(q,q',q'')est1}. 
Hence, proceeding as in \eqref{I11(q,q',q'')est1b}, we get
\begin{equation*}
	2L\xi\sum_{q\in\ZZ  }\sum_{\substack{|q-q'|\leq 5\\|q'-q''|\leq 5} }2^{-q}\I_2^2(q,q',q'')
	\lesssim
	\|				\nabla	Q_2									\|_{L^2_x}^2
	\|				\Delta	Q_2									\|_{L^2_x}^2
	\|				(			\delta u,\, \nabla \delta Q	)	\|_{\Hh^{-\frac{1}{2}}}^2 + 
	C_\nu		
	\|					\nabla	\delta u						\|_{\Hh^{-\frac{1}{2}}}^2 +
	C_{\Gamma, L}
	\|					\Delta	\delta Q						\|_{\Hh^{-\frac{1}{2}}}^2.
\end{equation*}	
When $j=4$:
\begin{align*}
	\I_4^2&(q,q',q''):=\int_{\RR_2}
		\Big\{
		(\,\Sd_{q'-1} Q_2-\Sd_{q-1} Q_2\,)
		\trc\{\,\Dd_{q'}(\,\Dd_{q''}Q_2\Sd_{q''+2}\nabla \delta u\,)\,\}\Dd_q \Delta \delta Q	+\\&	 
		\quad\quad\quad\quad\quad\quad\quad\quad\quad\quad-
		(\,\Sd_{q'-1} Q_2-\Sd_{q-1} Q_2\,)
		\trc\{\,\Dd_{q'}(\,\Dd_{q''}Q_2\Sd_{q''+2}\Delta \delta Q\,)\,\}\Dd_q \nabla \delta u	
	\Big\}\\
	&\lesssim
	\|	\Sd_{q'-1} Q_2-\Sd_{q-1} Q_2															\|_{L^\infty_x}
	\|	\Dd_{q'}(\,\Dd_{q''}Q_2\Sd_{q''+2}(\nabla \delta u,\,\Delta \delta Q\,) \,)				\|_{L^2_x}
	\|	\Dd_{q }	(	\nabla \delta u,\, \Delta \delta Q	)									\|_{L^2_x}\\
	&\lesssim
	2^{q'}
	\|	\Sd_{q'-1} \nabla Q_2-\Sd_{q-1} \nabla Q_2												\|_{L^2_x}
	\|	\Dd_{q'}(\,\Dd_{q''}Q_2\Sd_{q''+2}(\nabla \delta u,\,\Delta \delta Q\,) \,)				\|_{L^1_x}
	\|	\Dd_{q }	(	\nabla \delta u,\, \Delta \delta Q	)									\|_{L^2_x}\\
	&\lesssim
	2^{q'-q''}
	\|	\nabla Q_2																				\|_{L^2_x}
	\|	\Dd_{q''  } \Delta  Q_2																	\|_{L^2_x}
	\|	\Sd_{q''+2} ( 		   \delta u,\, \nabla \delta Q  )									\|_{L^2_x}
	\|	\Dd_{q    }	(	\nabla \delta u,\, \Delta \delta Q	)									\|_{L^2_x}\\
	&\lesssim
	2^{q'-q''}
	\|	\nabla Q_2																				\|_{L^2_x}
	\|	\Delta Q_2																				\|_{L^2_x}
	\|	\Sd_{q''+2} ( 		   \delta u,\, \nabla \delta Q  )									\|_{L^2_x}
	\|	\Dd_{q    }	(	\nabla \delta u,\, \Delta \delta Q	)									\|_{L^2_x},	
\end{align*}
which is equivalent to the last inequality of \eqref{ineq3}. Thus, arguing as in \eqref{ineq3b}, we deduce
\begin{equation*}
	2L\xi\sum_{q\in\ZZ  }\sum_{\substack{|q-q'|\leq 5\\q''\geq q'-5}}2^{-q}\I_4^2(q,q',q'')
	\lesssim
	\|				\nabla	Q_2									\|_{L^2_x}^2
	\|				\Delta	Q_2									\|_{L^2_x}^2
	\|				(			\delta u,\, \nabla \delta Q	)	\|_{\Hh^{-\frac{1}{2}}}^2 + 
	C_\nu		
	\|					\nabla	\delta u						\|_{\Hh^{-\frac{1}{2}}}^2 +
	C_{\Gamma, L}
	\|					\Delta	\delta Q						\|_{\Hh^{-\frac{1}{2}}}^2.
\end{equation*}
\noindent When $j=3$ we fix a real number $N>1$ and we consider the low frequencies $q'\leq N$ as follows
\begin{equation}\label{ineq1}
\begin{aligned}
	\I^2_3(q,q'):&=
	\int_{\RR_2}
	\trc
	\Big\{
	(\,\Sd_{q'-1} Q_2-\Sd_{q-1} Q_2\,)\trc\{\,\Dd_q(\,\Sd_{q'-1}Q_2\,\Dd_{q'}\nabla\delta u\,)\,\}\Dd_q \Delta\delta Q +\\&	 	\quad\quad\quad\quad\quad\quad\quad\quad\quad\quad-
	(\,\Sd_{q'-1} Q_2-\Sd_{q-1} Q_2\,)\trc\{\,\Dd_q(\,\Sd_{q'-1}Q_2\,\Dd_{q'}\Delta\delta Q\,)\}\Dd_q \nabla\delta u	
	\Big\}\\
	&\lesssim
	\|\Sd_{q'-1} Q_2-\Sd_{q-1} Q_2													\|_{L^{\infty}_x			}
	\|\Dd_q(\,\Sd_{q'-1}Q_2\,\Dd_{q'}(\nabla\delta u,\,\Delta \delta Q)\,)\,\}		\|_{L^{2}_x					}
	\|\Dd_q(\,\nabla \delta u,\,\Delta \delta Q)									\|_{L^2_x					}\\
	&\lesssim
	2^{-q}
	\|	\Sd_{q'-1}\Delta Q_2-\Sd_{q-1}\Delta Q_2					\|_{	L^2_x					}
	\|	\Sd_{q'-1}		  Q_2										\|_{	L^\infty_x				}	
	\|	\Dd_{q'}(\,	\nabla	\delta u,\,	\Delta	\delta Q	\,)		\|_{	L^2_x					}
	\|	\Dd_q	(\,	\nabla	\delta u,\,	\Delta	\delta Q	\,)		\|_{	L^2_x					}\\
	&\lesssim
	\|	\Sd_{q'-1}\Delta Q_2-\Sd_{q-1}\Delta Q_2					\|_{	L^2_x					}
	\|	\Sd_{q'-1}		  Q_2										\|_{	L^\infty_x				}	
	\|	\Dd_{q'}(\,	\nabla	\delta u,\,	\Delta	\delta Q	\,)		\|_{	L^2_x					}
	\|	\Dd_q	(\,			\delta u,\,	\nabla	\delta Q	\,)		\|_{	L^2_x					}.
\end{aligned}
\end{equation}
If $q'\leq 1$ then $\|\Sd_{q'-1} Q_2\|_{L^\infty_x}\lesssim 2^{\frac{q'}{2}}\|\Sd_{q'-1} Q_2 \|_{L^2_x}\leq \| Q_2 \|_{L^2_x}$, while if 
$1<q'\leq N$ we have
\begin{equation*}
	\|\Sd_{q'-1} Q_2\|_{L^\infty_x}
	\lesssim 
	(\| Q_2\|_{L^2_x}+\sqrt{q'-1}\| \nabla Q_2\|_{L^2_x})
	\lesssim
	(\|Q_2\|_{L^2_x}+\sqrt{N}\|\nabla Q_2\|_{L^2_x}),
\end{equation*} 
thanks to Theorem \ref{Thm_sqrt_N}. Therefore, we deduce that
\begin{align*}
	\I^2_3(q,q')
	&\lesssim
	\|	\Delta Q_2													\|_{	L^2_x					}
	(
	\|							Q_2									\|_{	L^2_x					}+
	\sqrt{N}
	\|					\nabla  Q_2									\|_{	L^2_x					}
	)	
	\|	\Dd_{q'}(\,	\nabla	\delta u,\,	\Delta	\delta Q	\,)		\|_{	L^2_x					}
	\|	\Dd_q	(\,			\delta u,\,	\nabla	\delta Q	\,)		\|_{	L^2_x					}\\
	&\lesssim
	(1+N)
	\|	\Delta 	Q_2													\|_{	L^2_x					}^2
	\|	(Q_2,\,\nabla  Q_2)											\|_{	L^2_x					}^2	
	\|	\Dd_q	(\,			\delta u,\,	\nabla	\delta Q	\,)		\|_{	L^2_x					}^2+
	C_\nu		
	\|	\Dd_{q'}		\nabla	\delta u							\|_{	L^2_x					}^2 +
	C_{\Gamma, L}	
	\|	\Dd_{q'}		\Delta	\delta Q							\|_{	L^2_x					}^2.
\end{align*}
Hence
\begin{equation}\label{I23q<N}
\begin{aligned}
	\sum_{q'\leq N}\sum_{| q'-q|\leq 5}2^{-q}\I^2_3(q,q')
	\lesssim
	(1+N)
	\|	\Delta 	Q_2													\|_{	L^2_x					}^2
	\|	(Q_2,\,\nabla  Q_2)											\|_{	L^2_x					}^2	&
	\|			(\,			\delta u,\,	\nabla	\delta Q	\,)		\|_{	\Hh^{-\frac{1}{2}}		}^2+ \\ &+
	C_\nu		
	\|					\nabla	\delta u							\|_{	\Hh^{-\frac{1}{2}}		}^2 +
	C_{\Gamma, L}	
	\|					\Delta	\delta Q							\|_{	\Hh^{-\frac{1}{2}}		}^2.
\end{aligned}
\end{equation}
For the high frequencies $q'>N$ we get,
\begin{equation}\label{ineq1b}
\begin{aligned}
	\I_3^2(q,q')
	&\lesssim
	\|\Sd_{q'-1} Q_2-\Sd_{q-1} Q_2													\|_{L^{\infty}_x			}
	\|\Dd_q(\,\Sd_{q'-1}Q_2\,\Dd_{q'}(\nabla\delta u,\,\Delta \delta Q)\,)\,\}		\|_{L^{2}_x					}
	\|\Dd_q(\,\nabla \delta u,\,\Delta \delta Q)									\|_{L^2_x					}\\
	&\lesssim
	2^{-q}
	\|	\Sd_{q'-1}\Delta Q_2-\Sd_{q-1}\Delta Q_2					\|_{	L^2_x					}
	\|	\Sd_{q'-1}		  Q_2										\|_{	L^\infty_x				}	
	\|	\Dd_{q'}(\,	\nabla	\delta u,\,	\Delta	\delta Q	\,)		\|_{	L^2_x					}
	\|	\Dd_q	(\,	\nabla	\delta u,\,	\Delta	\delta Q	\,)		\|_{	L^2_x					}\\
	&\lesssim
	2^{\frac{q'-q}{2}}
	\|	\Delta Q_2														\|_{	L^2_x					}
	( 1 + \sqrt{q'-1} )
	\|		(Q_2,\,	\nabla  Q_2)								\|_{	L^2_x					}
	\|	\Dd_{q'}(\,		\delta u,\,	\nabla	\delta Q\,)					\|_{	L^2_x					}
	\|	(\,	\nabla	\delta u,\,	\Delta	\delta Q	\,)					\|_{	\Hh^{-\frac{1}{2}}		}\\
	&\lesssim
	( 1 + \sqrt{q'-1} )
	\|	\Delta 	Q_2														\|_{	L^2_x					}
	\|	(Q_2,\, \nabla  Q_2	)											\|_{	L^2_x					}
	\|	(\,		 	\delta u,\,	\		\delta Q	\,)					\|_{	L^2_x					}
	\|	(\,	\nabla	\delta u,\,	\Delta	\delta Q	\,)					\|_{	\Hh^{-\frac{1}{2}}		},
\end{aligned}
\end{equation}
therefore
\begin{equation}\label{I23q>N}
\begin{aligned}
	\sum_{q'>N}\sum_{|q-q'|\leq 5}2^{-q}\I_3^2&(q,q')
	\lesssim
	2^{-N}
	\|	\Delta 	Q_2														\|_{	L^2_x					}
	\|	\nabla  Q_2														\|_{	L^2_x					}
	\|	(\,		 	\delta u,\,	\		\delta Q	\,)					\|_{	L^2_x					}
	\|	(\,	\nabla	\delta u,\,	\Delta	\delta Q	\,)					\|_{	\Hh^{-\frac{1}{2}}		}\\
	&\lesssim
	2^{-2N}
	\|	\Delta 	Q_2														\|_{	L^2_x					}^2
	\|	\nabla  Q_2														\|_{	L^2_x					}^2
	\|	(\,		 	\delta u,\,	\		\delta Q	\,)					\|_{	L^2_x					}^2	+
	C_\nu		
	\|					\nabla	\delta u							\|_{	\Hh^{-\frac{1}{2}}		}^2 +
	C_{\Gamma, L}	
	\|					\Delta	\delta Q							\|_{	\Hh^{-\frac{1}{2}}		}^2.
\end{aligned}
\end{equation}
Summarizing \eqref{I23q<N} and \eqref{I23q>N}, we get
\begin{equation}\label{ineq4c}
\begin{aligned}
	\sum_{q'\in\ZZ}\sum_{| q'-q|\leq 5}2^{-q} \I^2_3(q,q')
	&\lesssim
	(1+N)
	\|	\Delta 	Q_2													\|_{	L^2_x					}^2
	\|	(Q_2,\,\nabla  Q_2)											\|_{	L^2_x					}^2	
	\|			(\,			\delta u,\,	\nabla	\delta Q	\,)			\|_{	\Hh^{-\frac{1}{2}}		}^2 +\\ &+
	2^{-2N}
	\|	\Delta 	Q_2													\|_{	L^2_x					}^2
	\|	\nabla  Q_2													\|_{	L^2_x					}^2
	\|	(\,		 	\delta u,\,	\		\delta Q	\,)					\|_{	L^2_x					}^2	+
	C_\nu		
	\|					\nabla	\delta u								\|_{	\Hh^{-\frac{1}{2}}		}^2 +
	C_{\Gamma, L}	
	\|					\Delta	\delta Q								\|_{	\Hh^{-\frac{1}{2}}		}^2.
\end{aligned}
\end{equation}
Now we define $N:=\lceil\ln\{e + 1/\Phi(t)\}/2\rceil$, obtaining
\begin{equation}\label{ineq4b}
\begin{aligned}
	\sum_{q\in\ZZ}&\sum_{| q'-q|\leq 5}2^{-q}\I^2_3(q,q')
	\lesssim
	\|	\Delta 	Q_2(t)												\|_{	L^2_x					}^2
	\|	(Q_2,\,\nabla  Q_2)	(t)										\|_{	L^2_x					}^2	
	\|			(\,			\delta u,\,	\nabla	\delta Q(t)	\,)		\|_{	\Hh^{-\frac{1}{2}}		}^2+
	C_\nu		
	\|					\nabla	\delta u(t)							\|_{	\Hh^{-\frac{1}{2}}		}^2 +\\ &+
	C_{\Gamma, L}	
	\|					\Delta	\delta Q(t)							\|_{	\Hh^{-\frac{1}{2}}		}^2+
	\|	\Delta 	Q_2(t)												\|_{	L^2_x					}^2
	\|	\nabla  Q_2(t)												\|_{	L^2_x					}^2
	\|	(\,		 	\delta u,\,	\		\delta Q	\,)(t)				\|_{	L^2_x					}^2
	\Big( 1+ \ln\Big(e +\frac{1}{\Phi(t)}\Big) \Big).
\end{aligned}
\end{equation}

\noindent
\underline{Estimate of $\J_q^3$} Now, let us deal with the term of \eqref{E1+E2} related to $i=3$, namely
\begin{equation}\label{Jq3}
\begin{aligned}
	\int_{\RR_2}
	\trc
	&\big\{
		\Sd_{q-1} Q_2\trc\{\Dd_{q}(Q_2\nabla \delta u)\}\Dd_q \Delta \delta Q 	-	 
		\Sd_{q-1} Q_2]\trc\{\Dd_{q}(Q_2\Delta \delta Q)\}\Dd_q \nabla \delta u	\,
	\big\}
	=\\&=
	\sum_{j=1}^4
	\int_{\RR_2}
	\trc
	\big\{\,
		\Sd_{q-1} Q_2\trc\{\J^j_{q'}(Q_2,\,\nabla \delta u)\}\Dd_q \Delta \delta Q 	-	 
		\Sd_{q-1} Q_2\trc\{\J^j_{q'}(Q_2,\,\Delta \delta Q)\}\Dd_q \nabla \delta u	\,
	\big\}.
\end{aligned}
\end{equation}
Let us consider $j=1$ and define
\begin{align*}
	\I_1^3(q,q')&:=\int_{\RR_2}
	\trc
	\Big\{
		\Sd_{q-1} Q_2\trc\{[\Dd_{q},\Sd_{q'-1}Q_2]\Dd_{q'}\nabla \delta u)\}\Dd_q \Delta \delta Q-	
		\Sd_{q-1} Q_2\trc\{[\Dd_{q},\Sd_{q'-1}Q_2]\Dd_{q'}\Delta \delta Q)\}\Dd_q \nabla \delta u	
	\Big\}.
\end{align*}
 We proceed as for proving \eqref{ineq2d}: we fix a positive real $\ee\in (0,5/6]$ and  we consider the low frequencies $q\leq N$, for a suitable positive $N\geq 1$.
\begin{align*}
	\I^3_1(q,q')&=\int_{\RR_2}
	\trc
	\Big\{
	\Sd_{q-1}Q_2\trc\{[\Dd_q,\,\Sd_{q'-1} Q_2]\,\Dd_{q'}\nabla\delta u\}\Dd_q \Delta\delta Q 
		-
	\Sd_{q-1}Q_2\trc\{[\Dd_q,\,\Sd_{q'-1} Q_2]\,\Dd_{q'}\Delta\delta Q\}\Dd_q \nabla\delta u	
	\Big\}\\
	&\lesssim
	2^{-q'}
	\|	\Sd_{q-1}		  Q_2										\|_{	L^\infty_x				}	
	\|	\Sd_{q'-1}  \nabla Q_2										\|_{	L^{\frac{2}{\ee}}_x		}
	\|	\Dd_{q'}(\,	\nabla	\delta u,\,	\Delta	\delta Q	\,)		\|_{	L^{\frac{2}{1-\ee}}_x	}
	\|	\Dd_q	(\,	\nabla	\delta u,\,	\Delta	\delta Q	\,)		\|_{	L^2_x					}\\
	&\lesssim
	(1+\sqrt{N})
	2^{q-q'}
	\|				(Q_2,\,\nabla Q_2)							\|_{	L^2_x					}
	\|	\Sd_{q'-1} 	\nabla Q_2									\|_{	L^{\frac{2}{\ee}}_x		}
	\|	\Dd_{q'}(\,			\delta u,\,	\nabla	\delta Q	\,)		\|_{	L^{\frac{2}{1-\ee}}_x	}
	\|	\Dd_q	(\,	\nabla 	\delta u,\,	\Delta	\delta Q	\,)		\|_{	L^2_x					}\\
	&\lesssim
	\sqrt{\frac{N}{\ee}}
	\|				(Q_2,\,\nabla Q_2)								\|_{	L^2_x				}
	\|	\Sd_{q'-1} 	\nabla Q_2										\|_{	L^2_x				}^\ee
	\|	\Sd_{q'-1} 	\Delta Q_2										\|_{	L^2_x				}^{1-\ee}
	{\scriptstyle \times}\\ &\quad\quad\quad\quad\quad\quad\quad\quad\quad\quad\quad\quad	{\scriptstyle \times}
	\|	\Dd_{q'}(\,			\delta u,\,	\nabla	\delta Q	\,)		\|_{	L^2_x				}^{1-\ee}
	\|	\Dd_{q'}(\,	\nabla 	\delta u,\,	\Delta	\delta Q	\,)		\|_{	L^2_x				}^{\ee}
	\|	\Dd_q	(\,	\nabla 	\delta u,\,	\Delta	\delta Q	\,)		\|_{	L^2_x				},
\end{align*}
which is equivalent to the last inequality of \eqref{ineq2d}. Hence, arguing as for proving \eqref{ineq2b}, we get

\begin{align*}
	&\sum_{q\in\ZZ}\sum_{|q-q'|\leq 5}2^{-q}\I^3_1(q,q')\\
	&\lesssim
	\max\big\{
	\|			 	(Q_2,\,\nabla Q_2)								\|_{	L^2_x				}^{6}, 1\big\}
	\|			 	(\nabla Q_2,\, \Delta Q_2)						\|_{	L^2_x				}^2
	\|			(\,			\delta u,\,	\nabla	\delta Q	\,)		\|_{\Hh^{-\frac{1}{2}}		}^2
	\ln\Big( 1+e + \frac{1}{\Phi(t)}\Big)\big(1 + \ln \ln\Big(1+e + \frac{1}{\Phi(t)}\Big)\Big)
	 + \\&+
	\|			(\nabla Q_2,\,\Delta Q_2)												\|_{L^2_x				}^2
	\|			(Q_2,\,	\nabla  Q_2)														\|_{L^2_x				}^2
	\|			(	 u_1,\,u_2,\,\nabla	Q_1,\, \nabla Q_2	)							\|_{L^2_x				}^2
	\Phi(t)+
	C_\nu		
	\|					\nabla	\delta u						\|_{\Hh^{-\frac{1}{2}}}^2 +
	C_{\Gamma, L}
	\|					\Delta	\delta Q						\|_{\Hh^{-\frac{1}{2}}}^2,
\end{align*}

Further on, when $j=2$ in \eqref{Jq3}, let us consider the low frequencies $q\leq N$:
\begin{align*}
	\I^3_2(q,q')&:=\int_{\RR_2}
	\trc
	\Big\{
		\Sd_{q-1} Q_2\trc\{(\Sd_{q'-1}Q_2-\Sd_{q-1}Q_2)\,\Dd_{q}\Dd_{q'}\nabla\delta u)\}\Dd_q \Delta \delta Q+\\&
		\quad\quad\quad\quad\quad\quad\quad\quad\quad\quad\quad-	
		\Sd_{q-1} Q_2\trc\{(\Sd_{q'-1}Q_2-\Sd_{q-1}Q_2)\,\Dd_{q'}\Dd_{q'}\Delta\delta Q)\}\Dd_q \nabla \delta u	
	\Big\}\\
	&\lesssim
	\|	\Sd_{q-1}		  Q_2												\|_{	L^\infty_x				}
	\|	\Sd_{q'-1} Q_2-\Sd_{q-1} Q_2										\|_{	L^\infty_x				}		
	\|	\Dd_{q}\Dd_{q'}(\,	\nabla	\delta u,\,	\Delta	\delta Q	\,)		\|_{	L^2_x					}
	\|	\Dd_q	(\,	\nabla	\delta u,\,	\Delta	\delta Q	\,)				\|_{	L^2_x					}\\
	&\lesssim
	\|	\Sd_{q-1}		  Q_2												\|_{	L^\infty_x				}
	\|	\Sd_{q'-1}\Delta Q_2-\Sd_{q-1}\Delta Q_2							\|_{	L^2_x					}		
	\|	\Dd_{q'}		(\,	\nabla	\delta u,\,	\Delta	\delta Q	\,)		\|_{	L^2_x					}
	\|	\Dd_q			(\,			\delta u,\,	\nabla	\delta Q	\,)		\|_{	L^2_x					},
\end{align*}
which is as the last inequalities of \eqref{ineq1} (recalling that $q\sim q'$). Moreover for the high frequencies $q>N$
\begin{align*}
	\I^3_2(q,q')
	&\lesssim
	\|	\Sd_{q-1}		  Q_2												\|_{	L^\infty_x				}
	\|	\Sd_{q'-1} Q_2-\Sd_{q-1} Q_2										\|_{	L^\infty_x				}		
	\|	\Dd_{q}\Dd_{q'}(\,	\nabla	\delta u,\,	\Delta	\delta Q	\,)		\|_{	L^2_x					}
	\|	\Dd_q	(\,	\nabla	\delta u,\,	\Delta	\delta Q	\,)				\|_{	L^2_x					}\\
	&\lesssim
	(1+\sqrt{q-1})
	\|	(Q_2,\,\nabla  Q_2)												\|_{	L^2_x					}
	2^{-q}
	\|	\Sd_{q'-1} \Delta Q_2-\Sd_{q-1}\Delta Q_2							\|_{	L^2_x					}{\scriptstyle \times} \\ &
	\quad\quad\quad\quad\quad\quad\quad\quad\quad\quad\quad\quad\quad\quad\quad\quad\quad\quad {\scriptstyle \times}		
	\|	\Dd_q	(\,	\nabla	\delta u,\,	\Delta	\delta Q	\,)				\|_{	L^2_x					}
	\|	\Dd_q	(\,	\nabla	\delta u,\,	\Delta	\delta Q	\,)				\|_{	L^2_x					}\\
	&\lesssim
	(1+\sqrt{q-1})
	\|	\Delta 	Q_2														\|_{	L^2_x					}
	\|	(Q_2,\,\nabla  Q_2)												\|_{	L^2_x					}
	\|	(\,		 	\delta u,\,	\nabla	\delta Q	\,)					\|_{	L^2_x					}
	\|	(\,	\nabla	\delta u,\,	\Delta	\delta Q	\,)					\|_{	\Hh^{-\frac{1}{2}}		},
\end{align*}

which is the equivalent to the last inequality \eqref{ineq1b}. Hence, arguing as for proving \eqref{ineq4b}, we get
\begin{align*}
	\sum_{q\in\ZZ}&\sum_{| q'-q|\leq 5}2^{-q}\I^3_2(q,q')
	\lesssim
	\|	\Delta 	Q_2(t)												\|_{	L^2_x					}^2
	\|	(Q_2,\,\nabla  Q_2)	(t)										\|_{	L^2_x					}^2	
	\|			(\,			\delta u,\,	\nabla	\delta Q(t)	\,)		\|_{	\Hh^{-\frac{1}{2}}		}^2+
	C_\nu		
	\|					\nabla	\delta u(t)							\|_{	\Hh^{-\frac{1}{2}}		}^2 +\\ &+
	C_{\Gamma, L}	
	\|					\Delta	\delta Q(t)							\|_{	\Hh^{-\frac{1}{2}}		}^2+
	\|	\Delta 	Q_2(t)												\|_{	L^2_x					}^2
	\|	\nabla  Q_2(t)												\|_{	L^2_x					}^2
	\|	(\,		 	\delta u,\,	\		\delta Q	\,)(t)			\|_{	L^2_x					}^2\Phi(t)
	\big( 1+ \ln\Big(1+e + \frac{1}{\Phi(t)}\Big) \big).
\end{align*}

Now, when $j=3$ in \eqref{Jq3}, we observe that
\begin{equation*}
	\I^3_3(q):=\int_{\RR_2}
		\Big\{
		\trc\{\,\Sd_{q-1}Q_2\Dd_q\nabla \delta u\,\}\trc\{\,\Sd_{q-1} Q_2\Dd_q \Delta \delta Q\,\}-	
		\trc\{\,\Sd_{q-1}Q_2\Dd_q\Delta \delta Q\,\}\trc\{\,\Sd_{q-1} Q_2\Dd_q \nabla \delta u\,\}	
	\Big\}=0,
\end{equation*}
for any $q\in \ZZ$.
\par Thus it remains to control the $j=4$ term, namely
\begin{align*}
	\I^3_4(q,q')&:=\int_{\RR_2}
		\Big\{
		\Sd_{q-1} Q_2\trc\{(\Dd_q(\,\Dd_{q'}Q_2\Sd_{q'+2}\nabla\delta u\,)\,\}\Dd_q \Delta \delta Q-
		\Sd_{q-1} Q_2\trc\{(\Dd_q(\,\Dd_{q'}Q_2\Sd_{q'+2}\Delta\delta Q\,)\,\}\Dd_q \nabla \delta u	
	\Big\}\\
	&\lesssim
	\|	\Sd_{q-1} Q_2															\|_{L^\infty_x}
	\|	\Dd_q(\,\Dd_{q'}Q_2\Sd_{q'+2}(\nabla\delta u,\,\Delta \delta Q)\,)		\|_{L^2_x}
	\|	\Dd_q (\nabla \delta u	,\,\Delta \delta Q)								\|_{L^2_x}.
\end{align*}
At first let us consider the low frequencies $q\leq N$, with $N>1$:
\begin{align*}
	\I^3_4(q,q')
	&\lesssim
	(1+\sqrt{N})
	\|	(Q_2,\, \nabla Q_2)														\|_{L^2_x}
	\|	\Dd_{q'}Q_2																\|_{L^\infty_x}
	\|	\Sd_{q'+2}(\nabla\delta u,\, \Delta \delta Q)							\|_{L^2_x}
	\|	\Dd_q (\nabla \delta u	,\,\Delta \delta Q)								\|_{L^2_x}\\
	&\lesssim
	(1+\sqrt{N})
	\|			(Q_2,\,\nabla Q_2)												\|_{L^2_x}
	2^{-q'}
	\|	\Dd_{q'}	\Delta Q_2													\|_{L^2_x}
	\|	\Sd_{q'+2}(\nabla\delta u,\, \Delta \delta Q)							\|_{L^2_x}
	\|	\Dd_q (\nabla \delta u	,\,\Delta \delta Q)								\|_{L^2_x}\\
	&\lesssim
	(1+\sqrt{N})
	\|				\nabla Q_2													\|_{L^2_x}
	\|	\Dd_{q'}	\Delta Q_2													\|_{L^2_x}
	\|	\Sd_{q'+2}(\delta u,\, \nabla \delta Q)									\|_{L^2_x}
	\|	\Dd_q (\nabla \delta u	,\,\Delta \delta Q)								\|_{L^2_x}\\
	&\lesssim
	(1+\sqrt{N})
	2^{\frac{\,q'}{2}}
	\|				\nabla Q_2													\|_{L^2_x}
	\|	\Dd_{q'}	\Delta Q_2													\|_{L^2_x}
	\|	(\delta u,\, \nabla \delta Q)											\|_{\Hh^{-\frac{1}{2}}}
	\|	\Dd_q (\nabla \delta u	,\,\Delta \delta Q)								\|_{L^2_x},
\end{align*}
which yields
\begin{align*}
	\sum_{q\leq N}&\sum_{q'\geq q-5}2^{-q}\I^3_4(q,q')\\
	&\lesssim 
	(1+\sqrt{N})
	\|			(Q_2,\,\nabla Q_2)												\|_{L^2_x}
	\|	(\delta u,\, \nabla \delta Q)											\|_{\Hh^{-\frac{1}{2}}}
	\sum_{q\leq N}\sum_{q'\geq q-5}2^{\frac{\,q'}{2}-q}
	\|	\Dd_{q'}	\Delta Q_2													\|_{L^2_x}
	\|	\Dd_q (\nabla \delta u	,\,\Delta \delta Q)								\|_{L^2_x}\\
	&\lesssim 
	(1+\sqrt{N})
	\|			(Q_2,\,\nabla Q_2)												\|_{L^2_x}
	\|	(		\delta u,\, \nabla \delta Q)									\|_{\Hh^{-\frac{1}{2}}}
	\sum_{q\in  \ZZ}
	2^{-\frac{q}{2}}
	\|	\Dd_q(\nabla \delta u,\,	\Delta \delta Q)								\|_{L^2_x}
	\sum_{q'\geq q-5}2^{\frac{\,q'-q}{2}}
	\|	\Dd_{q'}	\Delta Q_2													\|_{L^2_x}\\
	&\lesssim
	(1+\sqrt{N})
	\|			(Q_2,\,\nabla Q_2)												\|_{L^2_x}
	\|	(		\delta u,\, \nabla \delta Q)									\|_{\Hh^{-\frac{1}{2}}}{\scriptstyle{\times}}\\
	&\hspace{5cm}{\scriptstyle{\times}}
	\Big(
	\sum_{q'\in\ZZ}
	\Big|
		\sum_{q\in\ZZ}
		2^{\frac{q-q'}{2}}1_{(-\infty, 5]}(q-q')
		\|		\Dd_q		\Delta Q_2										\|_{L^2_x}
	\Big|^2
	\Big)^\frac{1}{2}
	\|	(\nabla \delta u,\,	\Delta \delta Q)									\|_{\Hh^{-\frac{1}{2}}},	\\
\end{align*}
thus by convolution
\begin{align*}
	\sum_{q\leq N}&\sum_{q'\geq q-5}2^{-q}\I^3_4(q,q')
	(1+\sqrt{N})
	\|			(Q_2,\,\nabla Q_2)												\|_{L^2_x}
	\|				\Delta Q_2													\|_{L^2_x}
	\|	(		\delta u,\, \nabla \delta Q)									\|_{\Hh^{-\frac{1}{2}}}
	\|	(\nabla	\delta u,\, \Delta \delta Q)									\|_{\Hh^{-\frac{1}{2}}}
	\\
	&\lesssim
	(1+N)
	\|			(Q_2,\,\nabla Q_2)												\|_{L^2_x}^2
	\|				\Delta Q_2													\|_{L^2_x}^2
	\|	(		\delta u,\, \nabla \delta Q)									\|_{\Hh^{-\frac{1}{2}}}^2+
	C_\nu		
	\|					\nabla	\delta u							\|_{	\Hh^{-\frac{1}{2}}		}^2 +
	C_{\Gamma, L}	
	\|					\Delta	\delta Q							\|_{	\Hh^{-\frac{1}{2}}		}^2.
\end{align*}
For the high frequencies, $q>N$, 
\begin{align*}
	&\sum_{q\geq N}\sum_{q'\geq q-5}2^{-q}\I^3_2(q,q')\lesssim\\
	&\lesssim 
	\sum_{q\geq N}\sum_{q'\geq q-5}
	2^{-q}
	(1+\sqrt{q-1})
	\|		\Sd_{q-1}(Q_2,\,	\nabla Q_2)										\|_{L^2_x}
	\|			\Dd_{q'}		 Q_2											\|_{L^\infty_x}
	\|	\Sd_{q'+2}(\nabla \delta u,\, \Delta \delta Q)							\|_{L^2_x}
	\|	\Dd_q (\nabla \delta u	,\,\Delta \delta Q)								\|_{L^2_x}\\
&\lesssim 
	\|		(Q_2,\,	\nabla Q_2)													\|_{L^2_x}
	\|	 (\nabla \delta u	,\,\Delta \delta Q)									\|_{L^2_x}
	\sum_{q\geq N}
	2^{-\frac{q}{2}}
	(1+\sqrt{q})
	\sum_{q'\geq q-5}
	2^{\frac{q'-q}{2}}	
	\|			\Dd_{q'} \nabla		 Q_2										\|_{L^2_x}
	2^{-\frac{q'}{2}}
	\|	\Sd_{q'+2}(\nabla \delta u,\, \Delta \delta Q)							\|_{L^2_x}\\
	&\lesssim 
	\|		(Q_2,\,	\nabla Q_2)													\|_{L^2_x}
	\|	 (\nabla \delta u	,\,\Delta \delta Q)									\|_{L^2_x}
	\|		(\nabla \delta u,\, \Delta \delta Q)								\|_{\Hh^{-\frac{1}{2}}}
	2^{-\frac{N}{2}}
	\big(
	\sum_{q\in\ZZ}
	\big|
	\sum_{q'\geq q-5}
	2^{\frac{q'-q}{2}}	
	\|			\Dd_{q'} \nabla		 Q_2										\|_{L^2_x}
	\big|^2
	\big)^{\frac{1}{2}},
\end{align*}
so that, by convolution
\begin{equation*}
	\sum_{q\geq N}\sum_{q'\geq q-5}2^{-q}\I^3_2(q,q')\lesssim
	2^{-N}
	\|		(Q_2,\,	\nabla Q_2)													\|_{L^2_x}
	\|					 \nabla		 Q_2										\|_{L^2_x}
	\|	 (\nabla \delta u	,\,\Delta \delta Q)									\|_{L^2_x}
	\|		(\nabla \delta u,\, \Delta \delta Q)								\|_{\Hh^{-\frac{1}{2}}}.
\end{equation*}
Summarizing, we get
\begin{align*}
	\sum_{q\in\ZZ}\sum_{q'\geq q-5}&2^{-q}\I^3_2(q,q')\lesssim 
	(1+N)
	\|			(Q_2,\,\nabla Q_2)												\|_{L^2_x}^2
	\|				\Delta Q_2													\|_{L^2_x}^2
	\|	(		\delta u,\, \nabla \delta Q)									\|_{\Hh^{-\frac{1}{2}}}^2+
	C_\nu		
	\|					\nabla	\delta u							\|_{	\Hh^{-\frac{1}{2}}		}^2 +\\&+
	C_{\Gamma, L}	
	\|					\Delta	\delta Q							\|_{	\Hh^{-\frac{1}{2}}		}^2 +
	2^{-2N}
	\|		(Q_2,\,	\nabla Q_2)													\|_{L^2_x}^2
	\|					 \nabla		 Q_2										\|_{L^2_x}^2
	\|	 (\nabla \delta u	,\,\Delta \delta Q)									\|_{L^2_x}^2,
\end{align*}
which is similar to \eqref{ineq4c}, hence we can conclude as in \eqref{ineq4b}.

\noindent
\underline{Estimate of $\J_q^4$} 
Now, we handle the last term of \eqref{E1+E2}, which is related to $i=4$, namely
\begin{equation}\label{Jq4}
\begin{aligned}
	&\sum_{q'\geq q- 5}
	\int_{\RR_2}
	\trc
	\big\{
		\Dd_q\big[\,\Dd_{q'} Q_2\trc\{\Sd_{q'+2}(Q_2\nabla \delta u)\}\,\big]\Dd_q \Delta \delta Q  -
		 \Dd_q\big[\,\Dd_{q'} Q_2\trc\{\Sd_{q'+2}(Q_2\Delta \delta Q)\}\,\big]\Dd_q \nabla \delta u 
	\big\}
	=\\ &=
	\sum_{q'\leq q-5}\;\sum_{q''\leq q'+1}
	\int_{\RR_2}
	\trc
	\big\{
		\Dd_q\big[\,\Dd_{q'} Q_2\trc\{\Dd_{q'}(Q_2\nabla \delta u)\}\,\big]\Dd_q \Delta \delta Q -
		\Dd_q\big[\,\Dd_{q'} Q_2\trc\{\Dd_{q'}(Q_2\Delta \delta Q)\}\,\big]\Dd_q \nabla \delta u
	\big\}\\
	&=	
	\sum_{j=1}^4\sum_{q'\leq q-5}\;\sum_{q''\leq q'+1}
	\int_{\RR_2}
	\trc
	\big\{
		\Dd_q\big[\,\Dd_{q'} Q_2\trc\{\J_{q''}^j(Q_2,\,\nabla \delta u)\}\,\big]\Dd_q \Delta \delta Q - \\&
	\quad\quad\quad\quad\quad\quad\quad\quad\quad\quad\quad\quad\quad\quad\quad\quad\quad\quad\quad\quad\quad\quad\quad\quad-	
		\Dd_q\big[\,\Dd_{q'} Q_2\trc\{\J_{q''}^j(Q_2,\,\Delta \delta Q)\}\,\big]\Dd_q \nabla \delta u 
	\big\}.
\end{aligned}
\end{equation}

\noindent First, we consider the term related to $j=1$, that is
\begin{equation}\label{ineq5}
\begin{aligned}
	\I_{1}^4(q, q', q'', q''') 
	:&= 
	\int_{\RR_2}
	\trc
	\Big\{\;
	\Dd_q\big[\,\Dd_{q'} Q_2\trc\{[\Dd_{q''},\,\Sd_{q'''-1}Q_2]\,\Dd_{q'''}\nabla \delta u\}\,\big]\Dd_q \Delta \delta Q\, +\\&
	\quad\quad\quad\quad\quad\quad\quad\quad\quad\quad\quad\quad-
	\Dd_q\big[\,\Dd_{q'} Q_2\trc\{[\Dd_{q''},\,\Sd_{q'''-1}Q_2]\,\Dd_{q''}\Delta \delta Q\}\,\big]\Dd_q \nabla \delta u\;
	\Big\}\\
	&\lesssim
	\|		\Dd_q	\big[\,
						\Dd_{q'} Q_2
						\trc\{	[\Dd_{q''},\,S_{q'''-1}Q_2]\,\Dd_{q'''}(\nabla \delta u,\,\Delta \delta Q)	\,\}\,
					\big]
	\|_{L^2_x}
	\|	\Dd_q	(\nabla	\delta	u,\,	\Delta	\delta	Q	)		\|_{L^2_x}\\
	&\lesssim
	2^q
	\|		\Dd_q	\big[\,
						\Dd_{q'} Q_2
						\trc\{	[\Dd_{q''},\,S_{q'''-1}Q_2]\,\Dd_{q'''}(\nabla \delta u,\,\Delta \delta Q)	\,\}\,
					\big]
	\|_{L^1_x}
	\|	\Dd_q	(\nabla	\delta	u,\,	\Delta	\delta	Q	)		\|_{L^2_x		}\\
	&\lesssim
	2^{q-q''}
	\|		\Dd_{q'} 				Q_2								\|_{L^\infty_x	}
	\|		\Sd_{q'''-1}	\nabla	Q_2								\|_{L^2_x		}
	\|		\Dd_{q'''}(\nabla \delta u,\,\Delta \delta Q)			\|_{L^2_x		}
	\|	\Dd_q	(\nabla	\delta	u,\,	\Delta	\delta	Q	)		\|_{L^2_x		}\\
	&\lesssim
	2^{q-q'-q''}
	\|		\Dd_{q'} 		\Delta	Q_2								\|_{L^2_x		}
	\|		\Sd_{q'''-1}	\nabla	Q_2								\|_{L^2_x		}
	\|		\Dd_{q'''}(\nabla \delta u,\,\Delta \delta Q)			\|_{L^2_x		}
	\|	\Dd_q	(\nabla	\delta	u,\,	\Delta	\delta	Q	)		\|_{L^2_x		}\\
	&\lesssim
	2^{q-q'-q''+q'''}
	\|						\Delta	Q_2								\|_{L^2_x		}
	\|						\nabla	Q_2								\|_{L^2_x		}
	\|		\Dd_{q'''}(		 \delta u,\,\nabla \delta Q)			\|_{L^2_x		}
	\|	\Dd_q	(\nabla	\delta	u,\,	\Delta	\delta	Q	)		\|_{L^2_x		}.
\end{aligned}
\end{equation}
Hence, taking the sum in $q$, $q'$, $q''$ and $q'''$ (and observing that $|q''-q'''|\leq 5$), we get
\begin{equation}\label{ineq5b}
\begin{aligned}
	&\sum_{q\in\ZZ}\,\sum_{q'\geq q-5}\,\sum_{q''\leq q'+1}\,\sum_{|q'''-q''|\leq 5}	2^{-q}\I_{1}^4(q, q', q'', q''') 
	\lesssim\\
	&\lesssim
	\|						\nabla	Q_2								\|_{L^2_x				}
	\|				(\nabla	 \delta	u,\,\Delta	\delta	Q	)		\|_{\Hh^{-\frac{1}{2}}	}
	\sum_{q,\,q',\,q''}
	2^{\frac{q}{2}-q'}
	\|		\Dd_{q'} 		\Delta	Q_2								\|_{L^2_x		}
	\|		\Dd_{q''}(		 \delta u,\,\nabla \delta Q)			\|_{L^2_x		}\\
	&\lesssim
	\|						\nabla	Q_2								\|_{L^2_x				}
	\|				(\nabla	 \delta	u,\,\Delta	\delta	Q	)		\|_{\Hh^{-\frac{1}{2}}	}
	\sum_{q\in\ZZ}
	2^{\frac{q}{2}}
	\sum_{q'\geq q-5}
	2^{-q'}
	\sum_{q''\leq q'+1}
	2^{\frac{q''}{2}}
	\|		\Dd_{q'} 		\Delta	Q_2								\|_{L^2_x		}
	2^{-\frac{q''}{2}}
	\|		\Dd_{q''}(		 \delta u,\,\nabla \delta Q)			\|_{L^2_x		}\\
	&\lesssim
	\|						\nabla	Q_2								\|_{L^2_x				}
	\|				(\nabla	 \delta	u,\,\Delta	\delta	Q	)		\|_{\Hh^{-\frac{1}{2}}	}
	\sum_{q''\in\ZZ}
	2^{-\frac{q''}{2}}
	\|		\Dd_{q''}(		 \delta u,\,\nabla \delta Q)			\|_{L^2_x		}
	\sum_{q'\geq q''+1}
	2^{\frac{q''}{2}-q'}
	\|		\Dd_{q'} 		\Delta	Q_2								\|_{L^2_x		}
	\sum_{q\leq q'+5}
	2^{\frac{q}{2}}\\
	&\lesssim
	\|						\nabla	Q_2								\|_{L^2_x				}
	\|				(\nabla	 \delta	u,\,\Delta	\delta	Q	)		\|_{\Hh^{-\frac{1}{2}}	}
	\sum_{q''\in\ZZ}
	2^{-\frac{q''}{2}}
	\|		\Dd_{q''}(		 \delta u,\,\nabla \delta Q)			\|_{L^2_x		}
	\sum_{q'\geq q''+1}
	2^{\frac{q''}{2}-\frac{q'}{2}}
	\|		\Dd_{q'} 		\Delta	Q_2								\|_{L^2_x		}\\
	&\lesssim
	\|						\nabla	Q_2								\|_{L^2_x				}
	\|				(\nabla	 \delta	u,\,\Delta	\delta	Q	)		\|_{\Hh^{-\frac{1}{2}}	}
	\|				(		 \delta	u,\,\nabla	\delta	Q	)		\|_{\Hh^{-\frac{1}{2}}	}
	\Big(
		\sum_{q''\in\ZZ}
		\Big|
			\sum_{q'\geq q''+1}
			2^{\frac{q''}{2}-\frac{q'}{2}}
			\|		\Dd_{q'} 		\Delta	Q_2								\|_{L^2_x		}
		\Big|^2\;
	\Big)^{\frac{1}{2}}\\
	&\lesssim
	\|						\nabla	Q_2								\|_{L^2_x				}
	\|						\Delta	Q_2								\|_{L^2_x				}	
	\|				(\nabla	 \delta	u,\,\Delta	\delta	Q	)		\|_{\Hh^{-\frac{1}{2}}	}
	\|				(		 \delta	u,\,\nabla	\delta	Q	)		\|_{\Hh^{-\frac{1}{2}}	}\\
	&\lesssim
	\|						\nabla	Q_2								\|_{L^2_x				}^2
	\|						\Delta	Q_2								\|_{L^2_x				}^2
	\|				(		 \delta	u,\,\nabla	\delta	Q	)		\|_{\Hh^{-\frac{1}{2}}	}^2+
	C_\nu
	\|				\nabla	\delta	u								\|_{\Hh^{-\frac{1}{2}}	}^2+
	C_{\Gamma, L}
	\|				\Delta	\delta	Q								\|_{\Hh^{-\frac{1}{2}}	}^2.
\end{aligned}
\end{equation}

\noindent When $j=2$ in \eqref{Jq4}, we observe that
\begin{align*}
	\I_{2}^4(q, q', &q'', q''') 
	:= 
	\int_{\RR_2}
	\trc
	\Big\{\;
	\Dd_q	\big[\,\Dd_{q'} Q_2
				\trc\{(\Sd_{q'''-1}Q_2-\Sd_{q''-1}Q_2)\,\Dd_{q''}\Dd_{q'''}\nabla \delta u\}\,
			\big]\Dd_q \Delta \delta Q\, -\\&
	\quad\quad\quad\quad\quad\quad\quad\quad\quad\quad-
	\Dd_q	\big[\,\Dd_{q'} Q_2
				\trc\{(\Sd_{q'''-1}Q_2-\Sd_{q''-1}Q_2)\,\Dd_{q''}\Dd_{q'''}\Delta \delta Q\}\,
			\big]\Dd_q \nabla \delta u\;
	\Big\}\\
	&\lesssim
	\|	\Dd_q	\big[\,
					\Dd_{q'} Q_2
					\trc\{(\Sd_{q'''-1}Q_2-\Sd_{q''-1}Q_2)\,\Dd_{q''}\Dd_{q'''}(\nabla \delta u,\,\Delta \delta Q)	\,\}\,
				\big]
	\|_{L^2_x}
	\|	\Dd_q	(\nabla	\delta	u,\,	\Delta	\delta	Q	)		\|_{L^2_x}\\
	&\lesssim
	2^q
	\|	\Dd_q	\big[\,
					\Dd_{q'} Q_2
					\trc\{(\Sd_{q'''-1}Q_2-\Sd_{q''-1}Q_2)\,\Dd_{q''}\Dd_{q'''}(\nabla \delta u,\,\Delta \delta Q)	\,\}\,
				\big]
	\|_{L^1_x}
	\|	\Dd_q	(\nabla	\delta	u,\,	\Delta	\delta	Q	)		\|_{L^2_x		}\\
	&\lesssim
	2^{q}
	\|		\Dd_{q'} 				Q_2								\|_{L^\infty_x	}
	\|		(\Sd_{q'''-1}Q_2-\Sd_{q''-1}Q_2)						\|_{L^2_x		}
	\|	\Dd_{q''}\Dd_{q'''}(\nabla \delta u,\,\Delta \delta Q)		\|_{L^2_x		}
	\|	\Dd_q	(\nabla	\delta	u,\,	\Delta	\delta	Q	)		\|_{L^2_x		}\\
	&\lesssim
	2^{q-q'-q''}
	\|		\Dd_{q'} 		\Delta	Q_2								\|_{L^2_x		}
	\|		(\Sd_{q'''-1}\nabla Q_2-\Sd_{q''-1}\nabla Q_2)			\|_{L^2_x		}
	\|		\Dd_{q''}(\nabla \delta u,\,\Delta \delta Q)			\|_{L^2_x		}
	\|	\Dd_q	(\nabla	\delta	u,\,	\Delta	\delta	Q	)		\|_{L^2_x		}\\
	&\lesssim
	2^{q-q'}
	\|		\Dd_{q'} 		\Delta	Q_2								\|_{L^2_x		}
	\|						\nabla	Q_2								\|_{L^2_x		}
	\|		\Dd_{q''}	( 		 \delta u,\,\nabla \delta Q )		\|_{L^2_x		}
	\|		\Dd_q	  (\nabla \delta	u,\,\Delta \delta Q	)		\|_{L^2_x		},
\end{align*}
which is equivalent to the last inequality of \eqref{ineq5} (since $|q''-q'''|\leq 5$). Hence, arguing as for proving \eqref{ineq5b}, the 
following estimate holds:
\begin{align*}
	\sum_{q\in\ZZ}\,\sum_{q'\geq q-5}\,\sum_{q''\leq q'+1}\,&\sum_{|q'''-q''|\leq 5}	2^{-q}\I_{2}^4(q, q', q'', q''') 
	\lesssim\\
	&\lesssim
	\|						\nabla	Q_2								\|_{L^2_x				}^2
	\|						\Delta	Q_2								\|_{L^2_x				}^2
	\|				(		 \delta	u,\,\nabla	\delta	Q	)		\|_{\Hh^{-\frac{1}{2}}	}^2+
	C_\nu
	\|				\nabla	\delta	u								\|_{\Hh^{-\frac{1}{2}}	}^2+
	C_{\Gamma, L}
	\|				\Delta	\delta	Q								\|_{\Hh^{-\frac{1}{2}}	}^2.
\end{align*}

\noindent Now, let us analyze the term in \eqref{Jq4} related to $j=3$. Assuming $q''\leq N$ for a suitable positive $N$, we get  
\begin{align*}
	\I_{3}^4&(q, q', q'') 
	:= \int_{\RR_2}
	\trc
	\big\{
	\Dd_q\big[\,\Dd_{q'} Q_2\trc\{\Sd_{q''-1}Q_2\Dd_{q''}\nabla \delta u\}\big]\Dd_q \Delta \delta Q -
	\Dd_q\big[\,\Dd_{q'} Q_2\trc\{\Sd_{q''-1}Q_2\Dd_{q''}\Delta \delta Q\}\big]\Dd_q \nabla \delta u 
	\big\}\\
	&\lesssim	
	\|\Dd_q\big[\,\Dd_{q'} Q_2\trc\{\Sd_{q''-1}Q_2\,\Dd_{q''}(\nabla \delta u,\, \Delta \delta Q	)\}\,\big]			\|_{L^2_x}
	\|								\Dd_q (\nabla \delta u,\, \Delta \delta Q	)									\|_{L^2_x}\\
	&\lesssim	
	2^q
	\|\Dd_q\big[\,\Dd_{q'} Q_2\trc\{\Sd_{q''-1}Q_2\,\Dd_{q''}(\nabla \delta u,\, \Delta \delta Q)\}\,\big]			\|_{L^1_x}
	\|								\Dd_q (\nabla \delta u,\, \Delta \delta Q	)									\|_{L^2_x}\\
	&\lesssim	
	2^{q}
	\|		\Dd_{q'} Q_2																							\|_{L^2_x}
	\|		\Sd_{q''-1}Q_2																							\|_{L^\infty_x}
	\|		\Dd_{q''}(\nabla \delta u,\, \Delta \delta Q )															\|_{L^2_x}
	\|		\Dd_q	(\nabla \delta u,\, \Delta \delta Q	)															\|_{L^2_x}\\
	&\lesssim
	2^{q}
	\|		\Dd_{q'} Q_2																							\|_{L^2_x}
	(1+\sqrt{N})
	\|		(Q_2,\,	\nabla Q_2)																						\|_{L^2_x}
	\|		\Dd_{q''}(\nabla \delta u,\, \Delta \delta Q	)														\|_{L^2_x}
	\|		\Dd_q	(\nabla \delta u,\, \Delta \delta Q	)															\|_{L^2_x}\\
	&\lesssim
	(1+\sqrt{N})
	2^{\frac{3q}{2}+\frac{3q''}{2}-2q'}
	\|		\Dd_{q'} \Delta  Q_2																					\|_{L^2_x		}
	\|		(Q_2,\,	\nabla Q_2)																						\|_{L^2_x	}
	2^{-\frac{\;q''}{2}}
	\|				\Dd_{q''}(\nabla \delta u,\, \Delta \delta Q	)												\|_{L^2_x}
	\|						 (\nabla \delta u,\, \Delta \delta Q	)												\|_{\Hh^{-\frac{1}{2}}}
\end{align*}
Hence
\begin{align*}
	&\sum_{q''\leq N}\;\sum_{q'\geq q''-1} \;\sum_{q\leq q'+5}2^{-q} \I_{3}^4(q, q', q'') \lesssim
	(1+\sqrt{N})
	\|		(Q_2,\,	\nabla Q_2)																\|_{L^2_x	}
	\|				(\nabla \delta u, \Delta \delta Q	)									\|_{\Hh^{-\frac{1}{2}}}	{\scriptstyle\times}\\&
	\quad\quad\quad\quad\quad\quad\quad\quad\quad\quad\quad\quad\quad\quad{\scriptstyle\times}
	\sum_{q''\leq N}
	2^{-\frac{q''}{2}}
	\|		\Dd_{q''}(		\delta u,	\nabla \delta Q )									\|_{L^2_x}	
	\sum_{q'\geq q''-1}
	2^{\frac{3q''}{2}-2q'}
	\|		\Dd_{q'} \Delta Q_2																\|_{L^2_x		}	
	\sum_{q\leq q'+5}
	2^{\frac{q}{2}}\\
	&\lesssim
	(1+\sqrt{N})
	\|		(Q_2,\,	\nabla Q_2)																\|_{L^2_x	}
	\|				(\nabla \delta u,\, \Delta \delta Q	)									\|_{\Hh^{-\frac{1}{2}}}
	\sum_{q''\leq N}\;
	2^{-\frac{\;q''}{2}}
	\|		\Dd_{q''}(		\delta u,\,	\nabla \delta Q )									\|_{L^2_x}
	\sum_{q'\geq q''-1}
	2^{\frac{3q''}{2}-\frac{3q'}{2}}
	\|		\Dd_{q'} \Delta Q_2																\|_{L^2_x		}\\
	&\lesssim
	(1+\sqrt{N})
	\|		(Q_2,\,	\nabla Q_2)																\|_{L^2_x	}
	\|			(\nabla \delta u,\, \Delta \delta Q	)										\|_{\Hh^{-\frac{1}{2}}}
	\|			(		\delta u,\,	\nabla \delta Q )										\|_{\Hh^{-\frac{1}{2}}}
	\Big(
	\sum_{q''\in\ZZ }\;
	\Big|\sum_{q'\geq q''-1}
	2^{\frac{3}{2}q''-\frac{3}{2}q'}
	\|		\Dd_{q'} \Delta Q_2																\|_{L^2_x		}
	\Big|^2
	\Big)^{\frac{1}{2}}\\
	&\lesssim
	(1+\sqrt{N})
	\|		(Q_2,\,	\nabla Q_2)																\|_{L^2_x	}
	\|		\Delta Q_2																		\|_{L^2_x		}
	\|			(\nabla \delta u,\, \Delta \delta Q	)										\|_{\Hh^{-\frac{1}{2}}}
	\|			(		\delta u,\,	\nabla \delta Q )										\|_{\Hh^{-\frac{1}{2}}}.
\end{align*}
Considering the high frequencies $q''> N$
\begin{align*}
	\I_{3}^4(q, q', q'') 
	&\lesssim	
	\|\Dd_q\big[\,\Dd_{q'} Q_2\trc\{\Sd_{q''}Q_2\,\Dd_{q''}(\nabla \delta u,\, \Delta \delta Q	)\}\,\big]			\|_{L^2_x}
	\|		\Dd_q (\nabla \delta u,\, \Delta \delta Q	) 															\|_{L^2_x}\\
	&\lesssim	2^q
	\|\Dd_q\big[\,\Dd_{q'} Q_2\trc\{\Sd_{q''}Q_2\,\Dd_{q''}(\nabla \delta u,\, \Delta \delta Q	)\}\,\big]			\|_{L^1_x}
	\|	\Dd_q 	(		\delta u,\,	\nabla \delta Q )																\|_{L^2_x}\\
	&\lesssim
	2^{q}
	\|		\Dd_{q'} Q_2																							\|_{L^2_x}
	\|		\Sd_{q''}Q_2																							\|_{L^\infty_x}
	\|		\Dd_{q''}(\nabla \delta u,\, \Delta \delta Q	)														\|_{L^2_x}
	\|		\Dd_q	(\nabla \delta u,\, \Delta \delta Q	)															\|_{L^2_x}\\
	&\lesssim
	2^{\frac{3q}{2}-2q'}
	\|		\Dd_{q'}	\Delta Q_2																					\|_{L^2_x}
	(1+\sqrt{q''})
	\|		(Q_2,\,\nabla Q_2)																						\|_{L^2_x}
	2^{q''}
	\|		\Dd_{q''}(		 \delta u,\, \nabla \delta Q	)														\|_{L^2_x}
	2^{-\frac{q}{2}}
	\|		\Dd_q	 (\nabla \delta u,\, \Delta \delta Q	)														\|_{L^2_x}\\
	&\lesssim
	(1+\sqrt{q''})
	2^{\frac{3q}{2}+q''-2q'}
	\|				 \Delta  Q_2																					\|_{L^2_x		}
	\|				(Q_2,\,\nabla Q_2)																						\|_{L^2_x	}
	\|				(		\delta u,\,	\nabla \delta Q )															\|_{L^2_x}
	\|				(\nabla \delta u,\, \Delta \delta Q	)															\|_{\Hh^{-\frac{1}{2}}},
\end{align*}
which implies
\begin{align*}
	&\sum_{q''> N}\;\sum_{q'\geq q''-1} \;\sum_{q\leq q'+5}2^{-q} \I_{3}^4(q, q', q'')\lesssim\\
	&\lesssim
	\|				\Delta Q_2																\|_{L^2_x		}	
	\|				(Q_2,\,\nabla Q_2)														\|_{L^2_x	}
	\|				(\nabla \delta u,\, \Delta \delta Q	)									\|_{\Hh^{-\frac{1}{2}}}	
	\|				(		\delta u,\,	\nabla \delta Q )									\|_{L^2_x}
	\sum_{q''> N}
	(1+\sqrt{q''})
	2^{q''}
	\sum_{q'\geq q''-1}
	2^{-2q'}
	\sum_{q\leq q'+5}
	2^{\frac{q}{2}}\\
	&\lesssim
	\|				\Delta Q_2																\|_{L^2_x		}	
	\|				(Q_2,\,\nabla Q_2)														\|_{L^2_x	}
	\|				\Delta \delta Q															\|_{\Hh^{-\frac{1}{2}}}	
	\|				(		\delta u,\,	\nabla \delta Q )									\|_{L^2_x}
	\sum_{q''> N}
	(1+\sqrt{q''})
	2^{q''}
	\sum_{q'\geq q''-1}
	2^{-2q'+\frac{q'}{2}}\\
	&\lesssim
	\|				\Delta Q_2																\|_{L^2_x		}	
	\|				(Q_2,\,\nabla Q_2)														\|_{L^2_x	}
	\|				(\nabla \delta u,\, \Delta \delta Q	)									\|_{\Hh^{-\frac{1}{2}}}	
	\|				(		\delta u,\,	\nabla \delta Q )									\|_{L^2_x}
	\sum_{q''> N}
	(1+\sqrt{q''})
	2^{-\frac{\;q''}{2}}\\
	&\lesssim
	\|				\Delta Q_2																\|_{L^2_x		}	
	\|				(Q_2,\,\nabla Q_2)														\|_{L^2_x	}
	\|				(\nabla \delta u,\, \Delta \delta Q	)									\|_{\Hh^{-\frac{1}{2}}}	
	\|				(		\delta u,\,	\nabla \delta Q )									\|_{L^2_x}
	2^{-\frac{N}{2}},
\end{align*}
Summarizing the last inequalities we obtain an estimate similar to \eqref{ineq4c}, so that we can conclude arguing as in \eqref{ineq4b}. 
Finally, it remains to examine when $j=4$, as last term. Let us define
\begin{align*}
	\I_{4}^4(q, q', &q'', q''') 
	:= 
	\int_{\RR_2}
	\trc
	\Big\{\;
	\Dd_q	\big[\,\Dd_{q'} Q_2
				\trc\{\Dd_{q''}(\,\Dd_{q'''}Q_2\,\Sd_{q'''+2}\nabla \delta u\,)\}\,
			\big]\Dd_q \Delta \delta Q\, +\\&
	\quad\quad\quad\quad\quad\quad\quad\quad\quad\quad-
	\Dd_q	\big[\,\Dd_{q'} Q_2
				\trc\{\Dd_{q''}(\,\Dd_{q'''}Q_2\,\Sd_{q'''+2}\Delta \delta Q\,)\}\,
			\big]\Dd_q \nabla \delta u\;
	\Big\}\\
	&\lesssim
	\|	\Dd_q	\big[\,
					\Dd_{q'} Q_2
					\trc\{\Dd_{q''}[\,\Dd_{q'''}Q_2\,\Sd_{q'''+2}(\nabla \delta u,\,\Delta \delta Q)\, ]	\,\}\,
				\big]
	\|_{L^2_x}
	\|	\Dd_q	(\nabla	\delta	u,\,	\Delta	\delta	Q	)		\|_{L^2_x}\\
	&\lesssim
	2^q
	\|	\Dd_q	\big[\,
					\Dd_{q'} Q_2
					\trc\{\Dd_{q''}[\,\Dd_{q'''}Q_2\,\Sd_{q'''+2}(\nabla \delta u,\,\Delta \delta Q)\, ]	\,\}\,
				\big]
	\|_{L^1_x}
	\|	\Dd_q	(\nabla	\delta	u,\,	\Delta	\delta	Q	)		\|_{L^2_x		}\\
	&\lesssim
	2^{q}
	\|		\Dd_{q'} 				Q_2														\|_{L^2_x		}
	\|		\Dd_{q''}[\,\Dd_{q'''}Q_2\,\Sd_{q'''+2}(\nabla \delta u,\,\Delta \delta Q)	\,]	\|_{L^2_x		}
	\|		\Dd_q		(\nabla	\delta	u,\,	\Delta	\delta	Q	)						\|_{L^2_x		}\\
	&\lesssim
	2^{q+q''}
	\|		\Dd_{q'} 				Q_2														\|_{L^2_x		}
	\|		\Dd_{q''}[\,\Dd_{q'''}Q_2\,\Sd_{q'''+2}(\nabla \delta u,\,\Delta \delta Q)	\,]	\|_{L^1_x		}
	\|		\Dd_q		(\nabla	\delta	u,\,	\Delta	\delta	Q	)						\|_{L^2_x		}\\
	&\lesssim
	2^{q+q''}
	\|		\Dd_{q'} 		\Delta	Q_2								\|_{L^2_x		}
	\|		\Dd_{q'''}				Q_2								\|_{L^2_x		}
	\|		\Sd_{q'''+2}(\nabla \delta u,\,\Delta \delta Q)			\|_{L^2_x		}
	\|	\Dd_q	(\nabla	\delta	u,\,	\Delta	\delta	Q	)		\|_{L^2_x		}\\
	&\lesssim
	2^{q-q'+q''-q'''}
	\|		\Dd_{q'} 		\nabla	Q_2								\|_{L^2_x		}
	\|		\Dd_{q'''}		\Delta	Q_2								\|_{L^2_x		}
	\|		\Sd_{q'''+2}(	 \delta u,\,\nabla \delta Q)			\|_{L^2_x		}
	\|	\Dd_q	(\nabla	\delta	u,\,	\Delta	\delta	Q	)		\|_{L^2_x		}.
\end{align*}
Hence, taking the sum in $q$, $q'$, $q''$ and $q'''$, we get
\begin{align*}
	&\sum_{q\in\ZZ}\,\sum_{q'\geq q-5}\,\sum_{q''\leq q'-1}\,\sum_{q'''\geq q''+5}	2^{-q}\I_{4}^4(q, q', q'', q''') 
	\lesssim\\
	&\lesssim
	\|			(\nabla	\delta	u,\,	\Delta	\delta	Q	)		\|_{\Hh^{-\frac{1}{2}}}
	\|						\nabla	Q_2								\|_{L^2_x		}
	\sum_{q\in\ZZ}\,\sum_{q'\geq q-5}\,\sum_{q''\leq q'-1}\,\sum_{q'''\geq q''+5}
	2^{\frac{q}{2}-q'+q''-q'''}
	\|		\Dd_{q'''}		\Delta	Q_2								\|_{L^2_x		}
	\|			\Sd_{q'''+2}(	 \delta u,\,\nabla \delta Q)		\|_{L^2_x}\\
	&\lesssim
	\|			(\nabla	\delta	u,\,	\Delta	\delta	Q	)		\|_{\Hh^{-\frac{1}{2}}}
	\|						\nabla	Q_2								\|_{L^2_x		}
	\sum_{q'''\in\ZZ}
	\sum_{q''\leq q'''{-}5}
	2^{q''-q'''}
	\|		\Dd_{q'''}		\Delta	Q_2								\|_{L^2_x		}
	\|			\Sd_{q'''+2}(	 \delta u,\,\nabla \delta Q)		\|_{L^2_x}
	\sum_{q'\geq q''+1}2^{-q'}
	\sum_{q \leq q'+5}
	2^{\frac{q}{2}}\\
	&\lesssim
	\|			(\nabla	\delta	u,\,	\Delta	\delta	Q	)		\|_{\Hh^{-\frac{1}{2}}}
	\|						\nabla	Q_2								\|_{L^2_x		}
	\sum_{q'''\in\ZZ}
	\sum_{q''\leq q'''{-}5}
	2^{q''-q'''}
	\|		\Dd_{q'''}		\Delta	Q_2								\|_{L^2_x		}
	\|			\Sd_{q'''+2}(	 \delta u,\,\nabla \delta Q)		\|_{L^2_x}
	\sum_{q'\geq q''+1}
	2^{-\frac{q'}{2}}\\
	&\lesssim
	\|			(\nabla	\delta	u,\,	\Delta	\delta	Q	)		\|_{\Hh^{-\frac{1}{2}}}
	\|						\nabla	Q_2								\|_{L^2_x		}
	\sum_{q'''\in\ZZ}
	\sum_{q''\leq q'''-5}
	2^{\frac{q''-q'''}{2}}
	\|		\Dd_{q'''}		\Delta	Q_2								\|_{L^2_x		}
	2^{-\frac{q'''}{2}}
	\|			\Sd_{q'''+2}(	 \delta u,\,\nabla \delta Q)		\|_{L^2_x}\\
	&\lesssim	
	\|			(\nabla	\delta	u,\,	\Delta	\delta	Q	)		\|_{\Hh^{-\frac{1}{2}}}
	\|						\nabla	Q_2								\|_{L^2_x		}
	\sum_{q'''\in\ZZ}
	\|		\Dd_{q'''}		\Delta	Q_2								\|_{L^2_x		}
	2^{-\frac{q'''}{2}}
	\|			\Sd_{q'''+2}(	 \delta u,\,\nabla \delta Q)		\|_{L^2_x}\\
	&\lesssim
	\|			(\nabla	\delta	u,\,	\Delta	\delta	Q	)		\|_{\Hh^{-\frac{1}{2}}}
	\|						\nabla	Q_2								\|_{L^2_x		}
	\|						\Delta	Q_2								\|_{L^2_x		}
	\|			(	 \delta u,\,\nabla \delta Q)					\|_{\Hh^{-\frac{1}{2}}}\\
	&\lesssim
	\|						\nabla	Q_2								\|_{L^2_x				}^2
	\|						\Delta	Q_2								\|_{L^2_x				}^2
	\|				(		 \delta	u,\,\nabla	\delta	Q	)		\|_{\Hh^{-\frac{1}{2}}	}^2+
	C_\nu
	\|				\nabla	\delta	u								\|_{\Hh^{-\frac{1}{2}}	}^2+
	C_{\Gamma, L}
	\|				\Delta	\delta	Q								\|_{\Hh^{-\frac{1}{2}}	}^2
\end{align*}
and this concludes the estimates of the term $\EE_1+\EE_2$.


\subsubsection{Remaining Terms}

For the sake of completeness, now we analyze all the remaining terms. However we point out that they are going to be estimates using simply just Theorem 
\ref{theorem_product_homogeneous_sobolev_spaces}, hence they are not a challenging drawback. For instance, let us observe that
\begin{align*}
	L	\langle			 (\xi	\delta D	+	\delta \Omega	)	\delta Q	,	\Delta \delta Q	\rangle_{\Hh^{	-\frac{1}{2}	}}  +
	L	\langle			 (\xi		   D_2	+		   \Omega_2	)	\delta Q	,	\Delta \delta Q	\rangle_{\Hh^{	-\frac{1}{2}	}}  +
	L	\langle	\delta Q (\xi	\delta D	+	\delta \Omega	)				,	\Delta \delta Q	\rangle_{\Hh^{	-\frac{1}{2}	}}  +\\ +
	L	\langle	\delta Q (\xi		   D_2	+		   \Omega_2	)				,	\Delta \delta Q	\rangle_{\Hh^{	-\frac{1}{2}	}}
	\lesssim
	\|			\delta Q		\|_{\Hh^{	 \frac{1}{2}	}}
	\|	\nabla ( u_1,\, u_2)	\|_{		 L^2_x			 }
	\|	\Delta	\delta Q		\|_{\Hh^{	-\frac{1}{2}	}}
	\lesssim
	\|	\nabla	\delta Q		\|_{\Hh^{	-\frac{1}{2}	}}
	{\scriptstyle \times}\\{\scriptstyle \times}
	\|	\nabla ( u_1,\, u_2)	\|_{		 L^2_x			 }
	\|	\Delta	\delta Q		\|_{\Hh^{	-\frac{1}{2}	}}	
	\lesssim
	\|	\nabla ( u_1,\, u_2)	\|_{		 L^2_x			 }^2
	\|	\nabla	\delta Q		\|_{\Hh^{	-\frac{1}{2}	}}^2
	+C_{\Gamma, L}
	\|	\Delta	\delta Q		\|_{\Hh^{	-\frac{1}{2}	}}^2.	
\end{align*}
Moreover $La\Gamma \langle \delta Q,\,\Delta \delta Q \rangle_{\Hh^{-1/2}} \lesssim \| \delta Q \|_{\Hh^{-1/2}}^2 + C_{\Gamma, L}\|  \Delta \delta Q \|_{\Hh^{-1/2}}^2 $ and
\begin{align*}
	Lb\Gamma \langle Q_1 \delta Q + \delta Q Q_2,\,\Delta \delta Q \rangle_{	\Hh^{ -\frac{1}{2} }	} 
	\lesssim
	\| 	Q_1 \delta Q + \delta Q Q_2 \|_{\Hh^{-\frac{1}{2}}}
	\|	\Delta Q					\|_{\Hh^{-\frac{1}{2}}}
	\lesssim	
	\|	(Q_1,\,Q_2)					\|_{	L^2_x		  }
	\|			\delta Q			\|_{\Hh^{ \frac{1}{2}}}
	\|	\Delta 	\delta Q			\|_{\Hh^{-\frac{1}{2}}}\\
	\lesssim
	\|	(Q_1,\,Q_2)					\|_{	L^2_x		  }
	\|	\nabla \delta Q				\|_{\Hh^{-\frac{1}{2}}}
	\|	\Delta \delta Q				\|_{\Hh^{-\frac{1}{2}}}
	\lesssim
	\|	(Q_1,\,Q_2)					\|_{	L^2_x		  }^2
	\|	\nabla	\delta Q			\|_{\Hh^{-\frac{1}{2}}}^2
	+C_{\Gamma, L}
	\|	\Delta	\delta Q			\|_{\Hh^{-\frac{1}{2}}}^2.
\end{align*}
Furthermore, by a direct computation, we get
\begin{align*}
	Lc\Gamma \langle 	\delta Q 	\trc \{ 	Q_1^2  					 \}, 	\Delta \delta Q 	\rangle_{\Hh^{-\frac{1}{2}}} +
	Lc\Gamma \langle 		   Q_2 	\trc \{ Q_1\delta Q + \delta Q Q_2 	 \}, 	\Delta \delta Q 	\rangle_{\Hh^{-\frac{1}{2}}} 
	\lesssim
	\|	(Q_1^2,\,Q_2^2) 			\|_{L^2_x			  }
	\|			\delta Q			\|_{\Hh^{ \frac{1}{2}}}
	\|	\Delta 	\delta Q			\|_{\Hh^{-\frac{1}{2}}}\\
	\lesssim
	\|	(Q_1  ,\,Q_2  ) 			\|_{L^4_x			  }^2
	\|	\nabla	\delta Q			\|_{\Hh^{-\frac{1}{2}}}
	\|	\Delta 	\delta Q			\|_{\Hh^{-\frac{1}{2}}}
	\lesssim
	\|			(Q_1,\,Q_2)			\|_{L^2_x			  }^2
	\|	\nabla	(Q_1,\,Q_2)			\|_{L^2_x			  }^2
	\|	\nabla	\delta Q			\|_{\Hh^{-\frac{1}{2}}}^2
	+C_{\Gamma, L}
	\|	\Delta	\delta Q			\|_{\Hh^{-\frac{1}{2}}}^2	
\end{align*}
and
\begin{align*}
	L \langle 	\delta	 u		 \cdot \nabla 		 	 Q_1	,	\,&\Delta \delta Q 	\rangle_{\Hh^{-\frac{1}{2}}} +
	L \langle 			 u_2	 \cdot \nabla 	\delta	 Q 		,	\, \Delta \delta Q 	\rangle_{\Hh^{-\frac{1}{2}}} 
	\lesssim
	\|		(u_2, \nabla Q_1)					\|_{\Hh^{ \frac{3}{4}}}	
	\|		(\delta u,\, \nabla \delta Q )		\|_{\Hh^{-\frac{1}{4}}}
	\|					 \Delta \delta Q		\|_{\Hh^{-\frac{1}{2}}}\\
	&\lesssim
	\|		(u_2, \nabla Q_1)							\|_{	L^2_x		  }^{\frac{1}{4}}
	\|		(\nabla 	   u_2,	  	\Delta 	      Q_1)	\|_{	L^2_x		  }^{\frac{3}{4}}	
	\|		(		  \delta u,\,	\nabla \delta Q  )	\|_{\Hh^{-\frac{1}{2}}}^{\frac{3}{4}}
	\|		( \Delta  \delta u,\,	\Delta \delta Q  )	\|_{\Hh^{-\frac{1}{2}}}^{\frac{1}{4}}
	\|					 \Delta \delta Q				\|_{\Hh^{-\frac{1}{2}}}	\\
	&\lesssim
	\|		(u_2, \nabla Q_1)							\|_{	L^2_x		  }^{\frac{2}{3}}
	\|		(\nabla 	   u_2,	  	\Delta 	      Q_1)	\|_{	L^2_x		  }^2	
	\|		(		  \delta u,\,	\nabla \delta Q  )	\|_{\Hh^{-\frac{1}{2}}}^2	
	+ C_\nu
	\|	\nabla \delta u									\|_{\Hh^{-\frac{1}{2}}}^2
	+C_{\Gamma, L}
	\|	\Delta	\delta Q								\|_{\Hh^{-\frac{1}{2}}}^2.	
\end{align*}	
Moreover $a\xi \langle \delta Q Q_1,\,\nabla \delta u\rangle_{\Hh^{-1/2}} \lesssim \| \delta Q \|_{\Hh^{1/2}} \|Q_1 \|_{L^2} \|\nabla \delta u \|_{\Hh^{-1/2}} 
\lesssim \|Q_1\|_{L^2}^2 \| \nabla \delta Q\|_{\Hh^{-1/2}}^2 + C_{\nu} \|\nabla \delta u \|_{\Hh^{-1/2}}^2$,
\begin{align*}
	b\xi \langle \delta Q (Q_1^2 - &\trc\{ Q_1^2\}\frac{\Id}{2} ) , \, \nabla \delta u \rangle_{\Hh^{-\frac{1}{2}}} 
	\lesssim  
	\|			\delta 	Q 		\|_{	\Hh^{ \frac{1}{2}}	}
	\|					Q_1^2	\|_{		L^2_x			}
	\|	\nabla	\delta	u		\|_{	\Hh^{-\frac{1}{2}}	}
	=  
	\|	\nabla	\delta 	Q 		\|_{	\Hh^{-\frac{1}{2}}	}
	\|					Q_1		\|_{		L^4_x			}^{ 2	}
	\|	\nabla	\delta	u		\|_{	\Hh^{-\frac{1}{2}}	}\\	
	&\lesssim
	\|	\nabla	\delta 	Q 		\|_{	\Hh^{-\frac{1}{2}}	}
	\|					Q_1		\|_{		L^2_x			}
	\|	\nabla			Q_1		\|_{		L^2_x			}
	\|	\nabla	\delta	u		\|_{	\Hh^{-\frac{1}{2}}	}
	\lesssim
	\|					Q_1		\|_{		L^2_x			}^2
	\|	\nabla			Q_1		\|_{		L^2_x			}^2
	\|	\nabla	\delta 	Q 		\|_{	\Hh^{-\frac{1}{2}}	}^2
	+C_\nu
	\|	\nabla	\delta	u		\|_{	\Hh^{-\frac{1}{2}}	}^2
\end{align*}
and
\begin{align*}
	c\xi \langle \delta Q \trc (Q_1^2) Q_1,\,\nabla \delta u \rangle_{\Hh^{-\frac{1}{2}}}
	&\lesssim  
	\|			\delta 	Q 		\|_{	\Hh^{ \frac{1}{2}}	}
	\|					Q_1^2	\|_{		L^2_x			}
	\|					Q_1		\|_{		L^\infty		}
	\|	\nabla	\delta	u		\|_{	\Hh^{-\frac{1}{2}}	}\\
	&\lesssim
	\|					Q_1		\|_{		L^2_x			}^2
	\|	\nabla			Q_1		\|_{		L^2_x			}^2
	\|					Q_1		\|_{		H^2				}^2
	\|	\nabla	\delta 	Q 		\|_{	\Hh^{-\frac{1}{2}}	}^2
	+C_\nu
	\|	\nabla	\delta	u		\|_{	\Hh^{-\frac{1}{2}}	}^2	.
\end{align*}
Now, $a\xi \langle (Q_2 + \Id/2 )\delta Q, \nabla\delta u \rangle_{\Hh^{-1/2}} \lesssim (\|Q_2\|_{L^2_x}^2+1) \|\nabla \delta Q \|_{\Hh^{-1/2}}^2 + 
C_\nu \|\nabla \delta u\|_{\Hh^{-1/2}}^2  $ and
\begin{align*}
	b\xi &	\langle ( 	Q_2 + \frac{\Id}{2})(Q_1 \delta Q + \delta Q Q_2 )	,\, \nabla \delta u\rangle_{\Hh^{-\frac{1}{2}}}-
	b\xi 	\langle		Q_2 \trc\{Q_1 \delta Q + \delta Q Q_2 \}			,\, \nabla \delta u\rangle_{\Hh^{-\frac{1}{2}}}
	\lesssim\\
	&\lesssim
	(\|			 Q_2		\|_{		L^\infty_x		} + 1)
	 \|			(Q_1,\,Q_2)	\|_{		L^2_x			}
	 \|			\delta	Q	\|_{	\Hh^{ \frac{1}{2}}	}
	 \|	\nabla	\delta	u	\|_{	\Hh^{-\frac{1}{2}}	}
	 \lesssim
	 (\|			 Q_2	\|_{		H^2				} +  1)
	 \|			(Q_1,\,Q_2)	\|_{		L^2_x			}{\scriptstyle \times}\\&{\scriptstyle \times}
	 \|	\nabla	\delta	Q	\|_{	\Hh^{-\frac{1}{2}}	}
	 \|	\nabla	\delta	u	\|_{	\Hh^{-\frac{1}{2}}	}
	 \lesssim
	 (\|			 Q_2	\|_{		H^2				} + 1)^2
	 \|			(Q_1,\,Q_2)	\|_{		L^2_x			}^2
	 \|	\nabla	\delta	Q	\|_{	\Hh^{-\frac{1}{2}}	}^2
	 +C_\nu
	\|	\nabla	\delta	u		\|_{	\Hh^{-\frac{1}{2}}	}^2.
\end{align*}
Equivalently, we get
\begin{align*}
	c\xi &\langle ( Q_2 + \frac{\Id}{2} )	\delta 	Q	\trc\{ 			Q_1^2			   \}	,\,	\nabla \delta u	\rangle_{	\Hh^{-\frac{1}{2}}	} +
	c\xi \langle ( Q_2 + \frac{\Id}{2} )			Q_2	\trc\{ \delta Q Q_1 + Q_2 \delta Q \}	,\,	\nabla \delta u \rangle_{	\Hh^{-\frac{1}{2}}	}
	\lesssim		\\
	&\lesssim
	\|			Q_2				\|_{	L^\infty			}
	\|		  (	Q_1^2,\,Q_2^2)	\|_{	L^2_x				}
	\|	\delta 	Q				\|_{	\Hh^{ \frac{1}{2}}	}
	\|	\nabla	\delta	u		\|_{	\Hh^{-\frac{1}{2}}	}
	\lesssim
	\|			Q_2				\|_{	H^2					}
	\|		  (	Q_1,\,Q_2)		\|_{	L^4_x				}^2
	\|	\nabla	\delta 	Q		\|_{	\Hh^{ \frac{1}{2}}	}
	\|	\nabla	\delta	u		\|_{	\Hh^{-\frac{1}{2}}	}\\
	&\lesssim
	\|			Q_2				\|_{	H^2					}^2
	\|		   (Q_1,\,Q_2)		\|_{	L^2_x				}^2
	\|	\nabla (Q_1,\,Q_2)		\|_{	L^2_x				}^2
	\|	\nabla	\delta 	Q		\|_{	\Hh^{-\frac{1}{2}}	}^2
	+C_\nu
	\|	\nabla	\delta	u		\|_{	\Hh^{-\frac{1}{2}}	}^2		
\end{align*}
and moreover
\begin{align*}
	L\xi \langle	\delta Q	\Delta \delta 	Q	,\,  \nabla	\delta u \rangle_{	\Hh^{-\frac{1}{2}}	}+
	L\xi \langle	\delta Q	\Delta			Q_2	,\, &\nabla	\delta u \rangle_{	\Hh^{-\frac{1}{2}}	}
	\lesssim
	\|				\delta  Q	\|_{	\Hh^{ \frac{1}{2}}	}
	\|		\Delta	(Q_1,\,Q_2)	\|_{	L^2_x				}
	\|		\nabla	\delta  u	\|_{	\Hh^{-\frac{1}{2}}	}\\
	&\lesssim
	\|		\Delta	(Q_1,\,Q_2)	\|_{	L^2_x				}^2
	\|		\nabla	\delta  Q	\|_{	\Hh^{-\frac{1}{2}}	}^2
	+C_\nu
	\|	\nabla	\delta	u		\|_{	\Hh^{-\frac{1}{2}}	}^2.	
\end{align*}
We can similarly control the terms from  $-a\xi\langle Q_1 \delta Q, \nabla \delta u \rangle_{\Hh^{-1/2}}$ to $L\xi \langle \Delta Q_2, \delta Q,\,
\nabla \delta \rangle_{\Hh^{-1/2}}$ in \eqref{energy_uniq}, proceeding as in the previous estimates. Furthermore
\begin{align*}
	2a\xi \langle  &\delta Q   \trc\{ 			Q_1^2			   \} ,\,	\nabla \delta u	\rangle_{	\Hh^{-\frac{1}{2}}	} +
	2a\xi \langle		   Q_2 \trc\{ 		\delta Q Q_1		   \} ,\,	\nabla \delta u	\rangle_{	\Hh^{-\frac{1}{2}}	} +
	2a\xi \langle		   Q_2 \trc\{ 		Q_2 \delta Q		   \} ,\,	\nabla \delta u	\rangle_{	\Hh^{-\frac{1}{2}}	} 
	\lesssim \\
	&\lesssim 
	\|	\delta Q				\|_{	\Hh^{ \frac{1}{2}}	}
	\|  (Q_1^2,\,Q_2^2)			\|_{	L^2_x				}
	\|	\nabla	\delta u		\|_{	\Hh^{-\frac{1}{2}}	}
	\lesssim 
	\|  		(Q_1,\,Q_2)		\|_{	L^2_x				}^2
	\|  \nabla 	(Q_1,\,Q_2)		\|_{	L^2_x				}^2
	\|	\nabla	\delta Q		\|_{	\Hh^{-\frac{1}{2}}	}^2
	+C_\nu
	\|	\nabla	\delta u		\|_{	\Hh^{-\frac{1}{2}}	}^2,
\end{align*} 
\begin{align*}
	2b\xi   \langle	\delta	Q	\trc\{			Q_1^3								\},\,	\nabla \delta u	\rangle_{	\Hh^{-\frac{1}{2}}	} +
	2b\xi  &\langle			Q_2	\trc\{	\delta	Q			Q_2^2					\},\,	\nabla \delta u	\rangle_{	\Hh^{-\frac{1}{2}}	} +
	2b\xi	\langle			Q_2	\trc\{			Q_2	(\delta	Q Q_1 + Q_2 \delta Q )	\},\,	\nabla \delta u	\rangle_{	\Hh^{-\frac{1}{2}}	}  \\
	&\lesssim 
	\|  		(Q_1,\,Q_2)		\|_{	L^2_x					}^2
	\|  \nabla 	(Q_1,\,Q_2)		\|_{	L^2_x					}^2
	\|  		(Q_1,\,Q_2)		\|_{	H^2						}^2
	\|	\nabla	\delta Q		\|_{	\Hh^{-\frac{1}{2}}		}^2
	+C_\nu
	\|	\nabla	\delta u		\|_{	\Hh^{-\frac{1}{2}}		}^2
\end{align*}
and also
\begin{align*}
	2c\xi \langle  	\delta	Q  &\trc\{		   Q_1^2				\}^2										,\,	\nabla \delta u	\rangle_{	\Hh^{-\frac{1}{2}}	}+
	2c\xi \langle 			Q_2	\trc\{	\delta Q Q_1 +Q_2 \delta Q	\} 	\trc\{   	Q_1^2 \} 					,\,	\nabla \delta u	\rangle_{	\Hh^{-\frac{1}{2}}	}+\\&+
	2c\xi \langle 			Q_2	\trc\{	       Q_2^2				\}	\trc\{	\delta Q Q_1 + Q_2 \delta Q \}	,\,	\nabla \delta u	\rangle_{	\Hh^{-\frac{1}{2}}	}
	\lesssim
	\| 		(Q_1^4, \,Q_2^4) 	\|_{	L^2_x				}
	\|			\delta Q		\|_{	\Hh^{ \frac{1}{2}}	}
	\|	\nabla 	\delta u		\|_{	\Hh^{-\frac{1}{2}}	}\\
	&\lesssim
	\| 		(Q_1 , \,Q_2) 		\|_{	L^8_x				}^4
	\|	\nabla	\delta Q		\|_{	\Hh^{-\frac{1}{2}}	}
	\|	\nabla 	\delta u		\|_{	\Hh^{-\frac{1}{2}}	}
	\lesssim
	\| 			(Q_1 , \,Q_2) 		\|_{	L^2_x				}
	\| 	\nabla	(Q_1 , \,Q_2) 		\|_{	L^2_x				}^3
	\|	\nabla		\delta Q		\|_{	\Hh^{-\frac{1}{2}}	}
	\|	\nabla 		\delta u		\|_{	\Hh^{-\frac{1}{2}}	}\\
	&\lesssim
	\| 			(Q_1 , \,Q_2) 		\|_{	L^2_x				}^2
	\| 	\nabla	(Q_1 , \,Q_2) 		\|_{	L^2_x				}^6
	\|	\nabla		\delta Q		\|_{	\Hh^{-\frac{1}{2}}	}^2
	+C_\nu
	\|	\nabla 		\delta u		\|_{	\Hh^{-\frac{1}{2}}	}^2.
\end{align*}
Furthermore, we observe that
\begin{align*}
	2L\xi	\langle	\delta	Q	\trc\{	\delta	Q	\Delta	\delta	Q	\},\,  	\nabla \delta u	\rangle_{	\Hh^{-\frac{1}{2}}	}+
	2L\xi	\langle	\delta	Q	\trc\{	\delta	Q	&\Delta		  	Q_2	\},\,  \nabla \delta u	\rangle_{	\Hh^{-\frac{1}{2}}	}+
	2L\xi	\langle	\delta	Q	\trc\{			Q_2 \Delta	\delta	Q	\},\,	\nabla \delta u	\rangle_{	\Hh^{-\frac{1}{2}}	}+\\+
	2L\xi	\langle			Q_2	\trc\{	\delta	Q	\Delta	\delta	Q	\},\,	\nabla \delta u	\rangle_{	\Hh^{-\frac{1}{2}}	}+
	2L\xi	\langle	\delta	Q	\trc\{			&Q_2 \Delta			Q_2	\},\,	\nabla \delta u	\rangle_{	\Hh^{-\frac{1}{2}}	}+
	2L\xi	\langle			Q_2	\trc\{	\delta	Q	\Delta			Q_2	\},\,	\nabla \delta u	\rangle_{	\Hh^{-\frac{1}{2}}	}\\
	&\lesssim
	\|	\delta	Q			\|_{	\Hh^{ \frac{1}{2}}	}
	\|	\Delta 	(Q_1,\,Q_2)	\|_{	L^2_x				}
	\|			(Q_1,\,Q_2)	\|_{	L^\infty_x			}
	\|	\nabla	\delta	u	\|_{	\Hh^{-\frac{1}{2}}	}\\
	&\lesssim
	\|	\Delta 	(Q_1,\,Q_2)	\|_{	L^2_x				}^2
	\|			(Q_1,\,Q_2)	\|_{	H^2					}^2
	\|	\nabla	\delta	Q	\|_{	\Hh^{-\frac{1}{2}}	}^2
	+C_\nu
	\|	\nabla	\delta	u	\|_{	\Hh^{-\frac{1}{2}}	}^2
\end{align*}
and
\begin{align*}
	L	\langle	\nabla	\delta Q	\odot	\nabla		   Q_1,\,  	\nabla \delta u	\rangle_{	\Hh^{-\frac{1}{2}}	}&+
	L	\langle	\nabla		   Q_2	\odot	\nabla	\delta Q  ,\,  	\nabla \delta u	\rangle_{	\Hh^{-\frac{1}{2}}	}
	\lesssim
	\|	\nabla	\delta Q	\|_{	\Hh^{-\frac{1}{4}}	}
	\|	\nabla	(Q_1,\,Q_2)	\|_{	\Hh^{ \frac{3}{4}}	}
	\|	\nabla	\delta u	\|_{	\Hh^{-\frac{1}{2}}	}\\
	&\lesssim
	\|	\nabla	\delta Q	\|_{	\Hh^{-\frac{1}{2}}	}^{\frac{3}{4}}
	\|	\nabla	\delta Q	\|_{	\Hh^{ \frac{1}{2}}	}^{\frac{1}{4}}
	\|	\nabla	(Q_1,\,Q_2)	\|_{	L^2					}^{\frac{1}{4}}
	\|	\Delta	(Q_1,\,Q_2)	\|_{	L^2					}^{\frac{3}{4}}
	\|	\nabla	\delta u	\|_{	\Hh^{-\frac{1}{2}}	}\\
	&\lesssim
	\|	\nabla	\delta Q	\|_{	\Hh^{-\frac{1}{2}}	}^{\frac{3}{4}}
	\|	\Delta	\delta Q	\|_{	\Hh^{-\frac{1}{2}}	}^{\frac{1}{4}}
	\|	\nabla	(Q_1,\,Q_2)	\|_{	L^2					}^{\frac{1}{4}}
	\|	\Delta	(Q_1,\,Q_2)	\|_{	L^2					}^{\frac{3}{4}}
	\|	\nabla	\delta u	\|_{	\Hh^{-\frac{1}{2}}	}\\	
	&\lesssim
	\|	\nabla	(Q_1,\,Q_2)	\|_{	L^2					}^{\frac{2}{3}}
	\|	\Delta	(Q_1,\,Q_2)	\|_{	L^2					}^2
	\|	\nabla	\delta Q	\|_{	\Hh^{-\frac{1}{2}}	}^2
	+C_{\Gamma,L}
	\|	\Delta	\delta Q	\|_{	\Hh^{-\frac{1}{2}}	}^2
	+C_\nu
	\|	\nabla	\delta u	\|_{	\Hh^{-\frac{1}{2}}	}^2.	
\end{align*}
Moreover
\begin{align*}
	La \langle \delta Q 	 	   Q_1, \,	\nabla \delta u	\rangle_{	\Hh^{ -\frac{1}{2} } } + 
	 a \langle 		  Q_2 	\delta Q  , \,	\nabla \delta u	\rangle_{	\Hh^{ -\frac{1}{2} } } -
	 a \langle   	  Q_1  	\delta Q  , \,	\nabla \delta u	\rangle_{	\Hh^{ -\frac{1}{2} } } 
	\lesssim
	\|			\delta Q	\|_{	\Hh^{  \frac{1}{2} }	}
	\|	(Q_1,\,Q_2)			\|_{	L^2						}
	\|	\nabla	\delta u	\|_{	\Hh^{ -\frac{1}{2} }	}\\
	\lesssim
	\|	\nabla	\delta Q	\|_{	\Hh^{ -\frac{1}{2} }	}
	\|	(Q_1,\,Q_2)			\|_{	L^2						}
	\|	\nabla	\delta u	\|_{	\Hh^{ -\frac{1}{2} }	}
	\lesssim
	\|	\nabla	\delta Q	\|_{	\Hh^{ -\frac{1}{2} }	}^2
	\|	(Q_1,\,Q_2)			\|_{	L^2						}^2
	+C_\nu
	\|	\nabla	\delta u	\|_{	\Hh^{ -\frac{1}{2} }	}^2,
\end{align*}
\begin{align*}
 -Lb\langle \delta Q   (Q_1^2 					    - \trc\{ Q_1^2  					\}\frac{\Id}{3}),\,\nabla\delta u\rangle_{\Hh^{-\frac{1}{2}}}
 - b\langle		   Q_2 (Q_1\delta Q + \delta Q Q_2  - \trc\{ Q_1\delta Q + \delta Q Q_2 \}\frac{\Id}{3}),\,\nabla\delta u\rangle_{\Hh^{-\frac{1}{2}}}
 +\\
 + b\langle	(Q_1^2 					    - \trc\{ Q_1^2  					 \}\frac{\Id}{3})\delta	Q	,\,\nabla\delta u\rangle_{\Hh^{-\frac{1}{2}}}
 - b\langle (Q_1\delta Q + \delta Q Q_2  - \trc\{ Q_1\delta Q + \delta Q Q_2 \}\frac{\Id}{3})\delta	Q	,\,\nabla\delta u\rangle_{\Hh^{-\frac{1}{2}}}
 \lesssim\\
 \lesssim
 \|	 \delta	Q			\|_{	\Hh^{-\frac{1}{2}}	}
 \|	 (Q_1^2,\,Q_2^2)	\|_{	L^2					}
 \|	 \nabla \delta u	\|_{	\Hh^{-\frac{1}{2}}	}
 \lesssim
 \|	 \delta	Q			\|_{	\Hh^{ \frac{1}{2}}	}
 \|	 (Q_1,\,Q_2)		\|_{	L^4					}^2
 \|	 \nabla \delta u	\|_{	\Hh^{-\frac{1}{2}}	}
 \lesssim
 \|	 \nabla \delta	Q	\|_{	\Hh^{-\frac{1}{2}}	}
 \|	 (Q_1,\,Q_2)		\|_{	L^2					}{\scriptstyle \times}\\{\scriptstyle \times}
 \|	 \nabla (Q_1,\,Q_2)	\|_{	L^	2				}
 \|	 \nabla \delta 	u	\|_{	\Hh^{-\frac{1}{2}}	} 
 \lesssim
 \|	 \nabla \delta	Q	\|_{	\Hh^{-\frac{1}{2}}	}^2
 \|	 (Q_1,\,Q_2)		\|_{	L^2					}^2
 \|	 \nabla (Q_1,\,Q_2)	\|_{	L^	2				}^2
 +C_\nu
 \|	 \nabla \delta 	u	\|_{	\Hh^{-\frac{1}{2}}	}^2,
\end{align*}
and
\begin{align*}
	Lc \langle	\delta	Q			Q_1 \trc\{	Q_1^2	\}							,\,\nabla \delta u	\rangle_{	\Hh^{-\frac{1}{2}}	} +
	 c \langle	  		Q_2	\delta 	Q 	\trc\{	Q_1^2	\}							,\,\nabla \delta u	\rangle_{	\Hh^{-\frac{1}{2}}	} -
	 c \langle	  		Q_1	\delta 	Q 	\trc\{	Q_1^2	\}							,\,\nabla \delta u	\rangle_{	\Hh^{-\frac{1}{2}}	} -\\-
	 c \langle	\delta 	Q 			Q_2	\trc\{	Q_1^2	\}							,\,\nabla \delta u	\rangle_{	\Hh^{-\frac{1}{2}}	}	
	\lesssim
	\|	\delta Q		\|_{	\Hh^{ \frac{1}{2}}	}
	\|	(Q_1,\,Q_2)		\|_{	L^\infty_x			}
	\|	(Q_1^2,\,Q_2^2)	\|_{	L^2_x				}
	\|	\nabla \delta u	\|_{	\Hh^{-\frac{1}{2}}	}
	\lesssim
	\|	\nabla \delta Q	\|_{	\Hh^{-\frac{1}{2}}	}
	\|	(Q_1,\,Q_2)		\|_{	H^2					}{\scriptstyle \times}\\{\scriptstyle \times}
	\|	(Q_1,\,Q_2)		\|_{	L^4_x				}^2
	\|	\nabla \delta u	\|_{	\Hh^{-\frac{1}{2}}	}
	\lesssim
	\|	\nabla \delta Q		\|_{	\Hh^{-\frac{1}{2}}	}^2
	\|	(Q_1,\,Q_2)			\|_{	H^2					}^2
	\|	(Q_1,\,Q_2)			\|_{	L^2					}^2
	\|	\nabla(Q_1,\,Q_2)	\|_{	L^2					}^2
	+C_\nu
	\|	\nabla \delta u		\|_{	\Hh^{-\frac{1}{2}}	}^2.	
\end{align*}
Finally
\begin{align*}
	 \langle	u_2 &\cdot \nabla \delta u, 			\delta u \rangle_{	\Hh^{-\frac{1}{2}}	}
	=
	-\langle	u_2  \otimes 	 \delta u, \nabla		\delta u \rangle_{	\Hh^{-\frac{1}{2}}	}
	\lesssim
	\| 					u_2	\|_{	\Hh^{ \frac{1}{2}}	}
	\| 			\delta 	u	\|_{	L^2					}
	\| \nabla 	\delta 	u	\|_{	\Hh^{-\frac{1}{2}}	}\\
	&\lesssim
	\| 					u_2	\|_{	L^2					}^{\frac{1}{2}}
	\| 	\nabla			u_2	\|_{	L^2					}^{\frac{1}{2}}
	\| 			\delta 	u	\|_{	\Hh^{-\frac{1}{2}}	}^{\frac{1}{2}}
	\| \nabla 	\delta 	u	\|_{	\Hh^{-\frac{1}{2}}	}^{\frac{3}{2}}	
	\lesssim
	\| 					u_2	\|_{	L^2					}^2
	\| 	\nabla			u_2	\|_{	L^2					}^2
	\| 			\delta 	u	\|_{	\Hh^{-\frac{1}{2}}	}^2
	+C_\nu
	\| \nabla 	\delta 	u	\|_{	\Hh^{-\frac{1}{2}}	}^2		
\end{align*}
and $\langle \delta u \cdot \nabla u_1,\,\delta u\rangle_{\Hh^{-1/2}} \lesssim \| \delta u \|_{\Hh^{1/2}} \| \nabla u_1 \|_{L^2} 
\| \delta u \|_{\Hh^{-1/2}} \lesssim C_\nu\| \nabla \delta u \|_{\Hh^{-1/2}}^2 +  \| \nabla u_1 \|_{L^2}^2 \| \delta u \|_{\Hh^{-1/2}}^2$.

\subsubsection{Conclusion}
Recalling \eqref{energy_uniq} and summarizing all the the previous estimates, we conclude that there exists a function $\chi$ which belongs to 
$L^1_{loc}(\RR_+)$ such that
\begin{align*}
	\frac{\dd}{\dd t}\Phi(t)
	&+
	\nu \| \nabla  \delta u\|_{\Hh^{-\frac{1}{2}}}^2 + \Gamma L^2 \| \Delta Q \|_{\Hh^{-\frac{1}{2}}}^2
	\lesssim
	\chi (t) \mu (\Phi(t)) + 
	c_\nu\| \nabla  \delta u\|_{\Hh^{-\frac{1}{2}}}^2 
	+
	C_{\Gamma, L}
	\| \Delta Q \|_{\Hh^{-\frac{1}{2}}}^2
\end{align*}
where $\mu$ is the Osgood modulus of continuity defined in \eqref{Osgood_mc}. Hence, choosing $C_{\Gamma, L}$ and $C_\nu$ small enough from the beginning, 
we can absorb the last two terms on the right-hand side by the left-hand side, obtaining \eqref{uniqueness_inequality}. We deduce that $\Phi\equiv 0$, 
thanks to the Osgood Lemma and the null initial data $\Phi(0)=0$. Thus, $(\delta u, \nabla \delta Q)$ is identically zero and $\delta Q$ as well, since 
$\delta Q(t)$ decades to $0$ at infinity for almost every $t$.

\end{proof}

\appendix
\section{}

\begin{theorem}\label{theorem_product_homogeneous_sobolev_spaces}
	Let $s$ and $t$ be two real numbers such that $ |s|$ and $|t|$ belong to $[0,d/2)$. Let us assume that $s+t$ is positive, then for every 
	$a\in \Hh^s(\RR^d)$ and for every $b\in \Hh^t(\RR^d)$, the product $ab$ belongs to $\Hh^{s+t-d/2}$ and there exists a positive 
	constant (not dependent by $a$ and $b$) such that
	\begin{equation*}
		\|a b\|_{\Hh^{s+t-d/2}}\leq C\|a\|_{\Hh^{s}}\|b\|_{\Hh^t}
	\end{equation*}
\end{theorem}
\begin{proof}
	At first we identify the Sobolev Spaces $\Hh^s$ and $\Hh^t$ with the Besov Spaces $\BB_{2,2}^s$ and $\BB_{2,2}^t$ respectively. 
	We claim that $ab$ belongs to $\BB_{2,2}^{s+t-d/2}$ and 
	\begin{equation*}
		\|a b\|_{\BB_{2,2}^{s+t-\frac{d}{2}}}\leq C\|a\|_{\BB_{2,2}^{s}}\|b\|_{\BB_{2,2}^t},
	\end{equation*}
	for a suitable positive constant. 
	
	\noindent We decompose the product $ab$ through the Bony decomposition, namely $ab = \dot{T}_{a}b + \dot{T}_{b}a + R(a,b)$, where
	\begin{equation*}
		\dot{T}_{a}b	:= \sum_{q\in\ZZ}							\Dd_q 		a\, \Sd_{q-1}	b,\quad\quad
		\dot{T}_{b}a	:= \sum_{q\in\ZZ}							\Sd_{q-1}	a\,	\Dd_q 		b,\quad\quad
		\Rd(a,b)		:= \sum_{\substack{q\in\ZZ\\|\nu|\leq 1}}	\Dd_{q}		a\,	\Dd_{q+\nu} b.
	\end{equation*}
	For any $q\in \ZZ$, we have
	\begin{align*}
		2^{q(s+t-\frac{d}{2})}
		&\| (\Dd_q \dot{T}_ab ,\, \Dd_q \dot{T}_b a) \|_{L^2}
		\lesssim\\	&\lesssim
		\sum_{|q-q'|\leq 5}		2^{q's} 				\|\Dd_q		a \|_{L^2_x		} 	2^{q'(t-\frac{d}{2})} 	\|	\Sd_{q-1} b	\|_{L^\infty} +
		\sum_{|q-q'|\leq 5}		2^{q'(s-\frac{d}{2})} 	\|\Sd_{q-1} a \|_{L^\infty_x}	2^{q't} 				\| 	\Dd_q b 	\|_{L^2		},
	\end{align*}
	hence
	\begin{equation*}
	 	\| (\dot{T}_{a}b,\, \dot{T}_b a) \|_{\BB_{2,2}^{s+t-\frac{d}{2}}}\leq
	 	\| (\dot{T}_{a}b,\, \dot{T}_b a) \|_{\BB_{2,1}^{s+t-\frac{d}{2}}}\lesssim
	 	\| a \|_{\BB_{2, 2}^s}							\| b \|_{\BB_{\infty, 2}^{t-\frac{d}{2}}} + 
	 	\| a \|_{\BB_{\infty, 2}^{s-\frac{d}{2}}}		\| b \|_{\BB_{2, 2}^t}\lesssim
	 	\| a \|_{\BB_{2, 2}^s}							\| b \|_{\BB_{2, 2}^t},
	\end{equation*}	
	where we have used the embedding $\BB_{2,2}^{\sigma}\hookrightarrow \BB_{\infty, 2}^{\sigma -d/2}$, for any $\sigma\in\RR$ and moreover 
	the following norm-equivalence
	\begin{equation*}
		\| u \|_{\BB_{p,r}^{\tilde{\sigma}}} \approx \|(2^{\tilde{\sigma}}\|S_q u\|_{L^p_x})_{q\in\ZZ}\|_{l^r(\ZZ)},\quad\quad 
		u\in \BB_{p,r}^{\tilde{\sigma}},
	\end{equation*}	  
	for any $1\leq p,r\leq \infty$ and $\tilde{\sigma}<0$.
	
	\noindent In order to conclude the proof, we have to handle the rest $\Rd(a,b)$. By a direct computation, for any $q\in \ZZ$, 
	\begin{equation*}
		2^{(t+s)q} \| \Dd_q \Rd(a,b) \|_{L^1_x} \leq 
		\sum_{\substack{q'\geq q-5\\|\nu|\leq 1}} 2^{(q-q')(s+t)} 2^{q's} \|\Dd_{q'} a\|_{L^2_x} 2^{(q'+\nu)t}\|\Dd_{q'+\nu} a\|_{L^2_x}, 
	\end{equation*}
	so that, thanks to the Young inequality, we deduce
	\begin{equation*}
		\| \Rd(a,b) \|_{\BB_{2,2}^{s+t-\frac{d}{2}}} \lesssim
		\| \Rd(a,b) \|_{\BB_{1,1}^{s+t}} \lesssim 
		\| 	a		\|_{\BB_{2,2}^s}
		\| 	b		\|_{\BB_{2,2}^t},
	\end{equation*}
	where we have used the embedding $\BB_{1,1}^{s+t}\hookrightarrow \BB_{2,2}^{s+t-d/2}$ and moreover $\sum_{q\leq 5} 2^{q(s+t)}<\infty $  
	since $s+t$ is positive.	 
\end{proof}
\begin{theorem}\label{Thm_sqrt_N}
	Let $N$ be a positive real number and $f$ a function in $H^1$. Then $\Sd_N f$ belongs to $L^\infty_x$ and
	\begin{equation*}
		\|	\Sd_N f \|_{L^\infty_x} \lesssim \| f \|_{L^2_x} + \sqrt{N} \| \nabla f \|_{L^2_x}\lesssim (1+\sqrt{N})\|(f,\,\nabla f)\|_{L^2_x}.
	\end{equation*}
\end{theorem}
\begin{proof}
	We split $\Sd_N f$ into two parts, namely $\Sd_N f = \sum_{q<0} \Dd_q f + \sum_{0\leq q< N} \Dd_q f$. First we observe that
	\begin{equation*}
					\|	\sum_{q<0}					\Dd_q 	f \|_{L^\infty_x}	
		\leq			\sum_{q<0} 				\|	\Dd_q 	f \|_{L^\infty_x}	
		\lesssim 		\sum_{q<0} 	2^q 		\| 	\Dd_q 	f \|_{	L^2_x	}
		\lesssim \big(	\sum_{q<0}	2^q\,\big)	\| 			f \|_{	L^2_x	}.		
	\end{equation*}
	Similarly, considering the second term, we get
	\begin{align*}
					\|	\sum_{0<q\leq N}					\Dd_q 	f \|_{L^\infty_x}	
		&\leq				\sum_{0<q\leq N} 				\|	\Dd_q 	f \|_{L^\infty_x}	
		\lesssim 		\sum_{0<q\leq N} 	2^q 		\| 	\Dd_q 	f \|_{	L^2_x	}\\
		&\lesssim 		\sum_{0<q\leq N} 			\| 	\Dd_q \nabla 	f \|_{	L^2_x	}
		\lesssim
		\Big(
			\sum_{0<q\leq N}1
		\Big)^\frac{1}{2}
		\Big(
			\sum_{0<q\leq N}	\| 	\Dd_q \nabla 	f \|_{	L^2_x	}^2
		\Big)^\frac{1}{2}
		\lesssim 							\sqrt{N}	\| 			f \|_{	\Hh^1	},	
	\end{align*}
	which concludes the proof of the Theorem.
\end{proof}

The following Lemma plays a main role in the uniqueness result of Theorem \ref{thm: uniqueness}, more precisely inequality \eqref{appx_ineq1} is the key for the double-logarithmic estimate. 
\begin{lemma}\label{lemmaappx}
	There exist a positive constant $C$ such that for any $p\in [1,\infty)$ the following inequality is satisfied:
	\begin{equation}\label{appx_ineq1}
		\| f \|_{L^{2p}(\RR^2)} \leq C\sqrt{p} \| f \|_{L^2(\RR^2)}^\frac{1}{p} \| \nabla f \|_{L^2(\RR^2)}^{1-\frac{1}{p}}
	\end{equation}
\end{lemma}
\begin{proof}
	The proof of this lemma was presented in \cite{MR2155019} (lemma $4.3$) and we report it here, for the sake of simplicity. thanks to Sobolev embeddings, we have
	\begin{equation}\label{appx_ineq2}
		\| f \|_{L^{2p}(\RR^2)} \leq C\sqrt{p} \| f \|_{\dot H^{1-\frac{1}{p}}(\RR^2)}.
	\end{equation}
	Moreover, since $\dot H^{1-1/p}(\RR^2)$ is an interpolation space between $L^2(\RR^2)$ and $\dot H^1(\RR^2)$, the following inequality is satisfied:
	\begin{equation*}
		\| f \|_{\dot H^{1-\frac{1}{p}}(\RR^2)} \leq \| f \|_{L^2(\RR^2)}^\frac{1}{p}\| \nabla f \|_{L^2(\RR^2)}^{1-\frac{1}{p}},
	\end{equation*}
	which leads to \eqref{appx_ineq1}, together with \eqref{appx_ineq2}.
\end{proof}

\section{}
\begin{proposition}\label{prop:apriorismallstrainn} Let $(Q^{(n)},u^n)$ be a smooth solution of \eqref{approxsystem++} in dimension $d=2$ or $d=3$, with restriction \eqref{c+}, and smooth initial data $(\bar Q(x),\bar u(x))$, that decays fast enough at infinity so that we can integrate by parts in space (for any $t\ge 0$) without boundary terms. We assume that $|\xi|<\xi_0$ where $\xi_0$  is an explicitly computable constant, scale invariant, depending on $a,b,c,d,\Gamma,\nu,\lambda$.

\par    For $(\bar Q,\bar u)\in H^1\times L^2$,we have
\begin{equation}\label{apriorih1n}
	\|Q^{(n)}(t,\cdot)\|_{H^1}\le C_1+\bar C_1 e^{\bar C_1t}\|\bar Q\|_{H^1}, \forall t\ge 0
\end{equation} with $C_1,\bar C_1$ depending on $(a,b,c,d,\Gamma,L, \nu,\bar Q,\bar u)$. Moreover

\begin{equation}\label{apriorih2n}
	\|u^n(t,\cdot)\|_{L^2}^2+\nu\int_0^t\|\nabla u^n\|_{L^2}^2\le C_1
\end{equation}
\end{proposition}

\begin{proof}
We denote:
\begin{equation}
X^n_{\alpha\beta}\stackrel{\rm def}{=}L\Delta Q^{(n)}_{\alpha\beta}-cQ^{(n)}_{\alpha\beta}\textrm{tr}((Q^{(n)})^2),\,\alpha,\beta=1,2,3.
\end{equation}
\par Multiplying the first equation in \eqref{approxsystem++} by $-\lambda \bar{H}^n$ and the second one by $u^n$,taking the trace and integrating over $\mathbb{R}^d$, we get
\begin{equation}\label{energydecay+n}
\begin{aligned}
	&\frac{d}{dt}\int_{\mathbb{R}^d}\frac{1}{2}|u^n|^2+\frac{L\lambda}{2}|\nabla Q^{(n)}|^2+\lambda\big(\frac{a}{2}|Q^{(n)}|^2-
	\frac{b}{3}\textrm{tr}(Q^{(n)})^3+\frac{c}{4}|Q^{(n)}|^4\big)\,dx
	+\nu\|\nabla u^n\|_{L^2}^2	+\Gamma\lambda L^2	\|\Delta Q^{(n)}\|_{L^2}^2\\&+
	\Gamma\lambda c^2\|J_n(Q^{(n)}\trc\{Q^{(n)})\})\|_{L^2}^2 
	-2cL\Gamma\lambda\int_{\mathbb{R}^d}\Delta Q^{(n)}_{\alpha\beta}Q^{(n)}_{\alpha\beta}\trc \{(Q^{(n)})^2\}\,dx+
	a^2\Gamma\lambda\|Q^{(n)}\|_{L^2}^2\\&+b^2\Gamma\lambda\int_{\mathbb{R}^d}\trc\big\{ J_n\big((Q^{(n)})^2-
	\frac{\trc\{(Q^{(n)})^2\}}{d}\big)^2\big\}\,dx
 +\varepsilon\int_{\mathbb{R}^d} |R_\varepsilon u\nabla Q^{(n)}|^3\,dx+\varepsilon \int_{\mathbb{R}^d} |R_\varepsilon\nabla u^n|^4\,dx\\&
	\leq 
	2a\Gamma\lambda\underbrace{\int_{\mathbb{R}^d}\textrm{tr}\{X^nQ^{(n)}\}\,dx}_{\stackrel{\rm def}{=}\mathcal{I}_n}-
	2b\Gamma\lambda\underbrace{\int_{\mathbb{R}^d}\textrm{tr}\{X^n(Q^{(n)})^2\}\,dx}_{\stackrel{\rm def}
	{=}\mathcal{J}_n}\\&+2ab\Gamma\lambda\int_{\mathbb{R}^d}\textrm{tr}\{(Q^{(n)})^3\}\,dx
	 +\lambda\underbrace{\int_{\mathbb{R}^d}J_n\left(   R_\varepsilon u^n\cdot\nabla 
	  Q_{\alpha\beta}^{(n)}\right)J_n\left(bQ_{\alpha\gamma}^{(n)}Q_{\gamma\beta}^{(n)}-cQ_{\alpha\beta}^{(n)}
	  \big|Q^{(n)}\big|^2\right)\,dx}_{\stackrel{\rm def}{=}\mathcal{II}}\\
	 & +\lambda \int_{\mathbb{R}^d}J_n\left(-R_\varepsilon\Omega_{\alpha\gamma}^{(n)} Q_{\gamma\beta}^{(n)}+Q_{\alpha\gamma}^{(n)}R_\varepsilon\Omega_{\gamma\beta}^{(n)}\right)J_n\left(bQ_{\alpha\delta}^{(n)}Q_{\delta\beta}^{(n)}-cQ_{\alpha\beta}^{(n)}\big| Q^{(n)}\big|^2\right)\,dx 	
\end{aligned}
\end{equation}
\par Integrating by parts we have:
\begin{eqnarray}\label{eq:positivitycrosstermscn}
	-2cL\Gamma\lambda\int_{\mathbb{R}^d}\Delta	Q^{(n)}_{\alpha\beta}Q^{(n)}_{\alpha\beta}\textrm{tr}\{(Q^{(n)})^2\}dx=
	2cL\Gamma\lambda\int_{\mathbb{R}^d}Q^{(n)}_{\alpha\beta,k}
	Q^{(n)}_{\alpha\beta,k}\textrm{tr}\{(Q^{(n)})^2\}dx	\nonumber\\
	+2cL\Gamma\lambda\int_{\mathbb{R}^d}Q^{(n)}_{\alpha\beta,k}Q^{(n)}_{\alpha\beta}
	\partial_k\left(\textrm{tr}\{(Q^{(n)})^2\}\right)dx
	=2cL\Gamma\lambda\int_{\mathbb{R}^d}|\nabla Q^{(n)}|^2\textrm{tr}\{(Q^{(n)})^2\})\,dx\nonumber\\
	+cL\Gamma\lambda\int_{\mathbb{R}^d}|\nabla\left(\trc\{(Q^{(n)})^2\}\right)|^2\,dx\ge 0
\end{eqnarray}
(where for the last inequality we used the assumption (\ref{c+}) and $L,\Gamma,\lambda>0$). One can easily see that 
\begin{equation}
\mathcal{I}_n=-\frac{L}{2}\|\nabla Q^{(n)}\|_{L^2}^2-c\|Q^{(n)}\|_{L^4}^4
\end{equation}
and moreover
\begin{align*}
	\lambda \int_{\mathbb{R}^d}J_n\left(-R_\varepsilon\Omega_{\alpha\gamma}^{(n)} 
	Q_{\gamma\beta}^{(n)}+Q_{\alpha\gamma}^{(n)}R_\varepsilon\Omega_{\gamma\beta}^{(n)}\right)
	J_n\left(bQ_{\alpha\delta}^{(n)}Q_{\delta\beta}^{(n)}-cQ_{\alpha\beta}^{(n)}\big| Q^{(n)}\big|^2\right)\,dx \leq\\	
	\leq 
	\frac{\ee}{2}\int_{\mathbb{R}^d}|R_\varepsilon \nabla u^n|^4\,dx
	+C(\ee)\int_{\mathbb{R}^d}|Q^{(n)}|^4\,dx	
	+\frac{\Gamma c^2}{2}\int_{\mathbb{R}^d} |J_n(Q^{(n)}|Q^{(n)}|^2)|^2\,dx 	
\end{align*}
\par On the other hand, for any $\varepsilon>0$ and $\tilde C=\tilde C(\varepsilon,c)$ an explicitly computable constant, we have:
\begin{eqnarray}
\mathcal{J}_n=L\int_{\mathbb{R}^d}Q^{(n)}_{\alpha\beta,kk}Q^{(n)}_{\alpha\gamma}Q^{(n)}_{\gamma\beta}\,dx-c\int_{\mathbb{R}^d}\textrm{tr}\{(Q^{(n)})^2\}\textrm{tr}\{(Q)^{(n)})^3\}\,dx
\le-L\int_{\mathbb{R}^d}Q^{(n)}_{\alpha\beta,k}Q^{(n)}_{\alpha\gamma,k}Q^{(n)}_{\gamma\beta}\,dx\nonumber\\-L\int_{\mathbb{R}^d}Q^{(n)}_{\alpha\beta,k}Q^{(n)}_{\alpha\gamma}Q^{(n)}_{\gamma\beta,k}+ \int_{\mathbb{R}^d}\textrm{tr}\{(Q^{(n)})^2\}\left(\frac{\tilde C}\varepsilon\textrm{tr}\{(Q^{(n)})^2\}+\varepsilon\textrm{tr}^2\{(Q^{(n)})^2\}\right)\,dx\nonumber\\
\le L\varepsilon \int_{\mathbb{R}^d}|\nabla Q^{(n)}|^2\textrm{tr}\{(Q^{(n)})^2\}\,dx+\frac{\tilde C}{\varepsilon}\|\nabla Q^{(n)}\|_{L^2}^2+\int_{\mathbb{R}^d}\textrm{tr}\{(Q^{(n)})^2\}\left(\frac{\tilde C}\varepsilon\textrm{tr}\{(Q^{(n)})^2\}+\varepsilon\textrm{tr}^2\{(Q^{(n)})^2\}\right)\,dx
\end{eqnarray}
\par Using the last four relations in (\ref{energydecay+n}) and considering \eqref{est_II} we obtain:
\begin{align*}
\frac{d}{dt}\int_{\mathbb{R}^d}\frac{1}{2}|u^n|^2+\frac{L\lambda}{2}|\nabla Q^{(n)}|^2+\lambda\big(\frac{a}{2}|Q^{(n)}|^2-
	\frac{b}{3}\textrm{tr}(Q^{(n)})^3+\frac{c}{4}|Q^{(n)}|^4\big)\,dx
	+\nu\|\nabla u^n\|_{L^2}^2	+\Gamma\lambda L^2	\|\Delta Q^{(n)}\|_{L^2}^2\\+
	\frac{\Gamma\lambda c^2}{2}\|J_n(Q^{(n)}\trc\{Q^{(n)})\})\|_{L^2}^2
+a^2\Gamma\lambda\|Q^{(n)}\|_{L^2}^2
+2cL\Gamma\lambda\int_{\mathbb{R}^d}|\nabla Q^{(n)}|^2\textrm{tr}\{(Q^{(n)})^2\}\,dx\\+cL\Gamma\lambda\int_{\mathbb{R}^d}|\nabla\left(\textrm{tr}\{(Q^{(n)})^2\}\right)|^2\,dx
 + \frac{\varepsilon}{2}\int_{\mathbb{R}^d}  |R_\varepsilon u^n\cdot\nabla Q^{(n)}|^3\,dx+\frac{\varepsilon}{2}\int_{\mathbb{R}^d} |\nabla R_\varepsilon u^n|^4\,dx\\
\le 2|a|\Gamma\lambda(\frac{L}{2}\|\nabla Q^{(n)}\|_{L^2}^2+c\|Q^{(n)}\|_{L^4}^4)+2|b|\Gamma\lambda L\varepsilon \int_{\mathbb{R}^d}|\nabla Q^{(n)}|^2\textrm{tr}\{(Q^{(n)})^2\}\,dx+2|b|\Gamma\lambda\frac{\tilde C}{\varepsilon}
\|\nabla Q^{(n)}\|_{L^2}^2\\
+2|b|\Gamma\lambda\int_{\mathbb{R}^d}\textrm{tr}\{(Q^{(n)})^2\}\left(\frac{\tilde C}\varepsilon\textrm{tr}\{(Q^{(n)})^2\}+\varepsilon\textrm{tr}^2\{(Q^{(n)})^2\}\right)\,dx+2|ab|\Gamma\lambda(\varepsilon\|Q^{(n)}\|_{L^2}^2+ (C(\ee)+\frac{\tilde C}{\varepsilon})\|Q^{(n)}\|_{L^4}^4)	
\end{align*}
\par Taking $\varepsilon$ small enough we can absorb all the terms with an $\epsilon$ coefficient on the right into the left hand side, and  we are left with 
\begin{eqnarray}
\frac{d}{dt}\int_{\mathbb{R}^d}\frac{1}{2}|u^n|^2+\frac{L\lambda}{2}|\nabla Q^{(n)}|^2+\lambda\big(\frac{a}{2}|Q^{(n)}|^2-
	\frac{b}{3}\textrm{tr}(Q^{(n)})^3+\frac{c}{4}|Q^{(n)}|^4\big)\,dx\nonumber\\
+\nu \|\nabla u^n\|_{L^2}^2+\Gamma\lambda L^2\|\Delta Q^{(n)}\|_{L^2}^2+\frac{\Gamma\lambda c^2}{2}\|J_n(Q^{(n)}\trc\{Q^{(n)})\})\|_{L^2}^2+\Gamma\lambda a^2\|Q^{(n)}\|_{L^2}^2\nonumber\\+2cL\Gamma\lambda\int_{\mathbb{R}^d}|\nabla Q^{(n)}|^2\textrm{tr}\{(Q^{(n)})^2\}\,dx+cL\Gamma\lambda\int_{\mathbb{R}^d}|\nabla\left(\textrm{tr}\{(Q^{(n)})^2\}\right)|^2\,dx
\le \bar C\left(\|\nabla Q^{(n)}\|_{L^2}^2+\|Q^{(n)}\|_{L^4}^4\right)\nonumber
\end{eqnarray} with $\bar C=\bar C(a,b,c)$.

\par The last relation is not yet enough because there are no positive terms. However, let us note that, if $a>0$ we obtain the a-priori estimates  by using the inequality
$tr\{(Q^{(n)})^3\}\leq \frac 38 tr\{(Q^{(n)})^2)\}+tr\{(Q^{(n)})^2\}^2$.  If $a\leq 0$ we have to estimate separately $\|Q^{(n)}\|_{L^2}$ and this ask for a smallness condition for $\xi$. Indeed, proceeding as for proving \eqref{estimate_control_H1}, we get
 
 \begin{eqnarray}
 \frac{d}{dt}\Big[\int_{\mathbb{R}^d}\frac{1}{2}|u^n|^2+\frac{L\lambda}{2}|\nabla Q^{(n)}|^2+\lambda\big(\frac{a}{2}|Q^{(n)}|^2-
	\frac{b}{3}\textrm{tr}(Q^{(n)})^3+\frac{c}{4}|Q^{(n)}|^4\big)\,dx+M\|Q\|_{L^2}^2\big]\nonumber\\+\nu \|\nabla u^n\|_{L^2}^2+\Gamma\lambda L^2\|\Delta Q^{(n)}\|_{L^2}^2+\frac{\Gamma\lambda c^2}{2}\|J_n(Q^{(n)}\trc\{Q^{(n)})\})\|_{L^2}^2+a^2\|Q^{(n)}\|_{L^2}^2\nonumber\\
 +2cL\Gamma\lambda\int_{\mathbb{R}^d}|\nabla Q^{(n)}|^2\textrm{tr}\{(Q^{(n)})^2\}\,dx+cL\Gamma\lambda\int_{\mathbb{R}^d}|\nabla\left(\textrm{tr}\{(Q^{(n)})^2\}\right)|^2\,dx\nonumber\\+\frac{\varepsilon}{2}\int_{\mathbb{R}^d}  |R_\varepsilon u^n\cdot\nabla Q^{(n)}|^3\,dx+\frac{\varepsilon}{2}\int_{\mathbb{R}^d} |\nabla R_\varepsilon u^n|^4\,dx
\le \bar C\left(\|\nabla Q^{(n)}\|_{L^2}^2+\|Q^{(n)}\|_{L^4}^4\right)+\nonumber\\MC(d)\varepsilon\int_{\mathbb{R}^d} |\nabla u^n|^2\,dx+\frac{M|\xi|^2}{\varepsilon}\int_{\mathbb{R}^d}|Q^{(n)}|^2+|J_n(Q^{(n)}\trc\{(Q^{(n)})^2\}|^2\,dx+M\hat C\int_{\mathbb{R}^d}|Q^{(n)}|^2+|Q^{(n)}|^4\,dx
\end{eqnarray}

\par We chose $\varepsilon$ small enough so that $MC(d)\varepsilon<\nu$. Finally we make the assumption that $|\xi|$ is small enough, depending on $a,b,c,d,\nu$ so that 

$$\frac{M|\xi|^2}{\varepsilon}\le \Gamma\lambda c^2$$

\par Then taking into account that
\begin{equation*}
	 \frac{M}{2} \textrm{tr}\{(Q^{(n)})^2\}+\frac{c}{8}\textrm{tr}^2\{(Q^{(n)})^2\}\le(M+\frac{a}{2})
	 \textrm{tr}\{(Q^{(n)})^2\}-\frac{b}{3}\textrm{tr}\{(Q^{(n)})^3\}+\frac{c}{4}\textrm{tr}^2\{(Q^{(n)})^2\}
\end{equation*}
we obtain the claimed relation (\ref{apriorih1n}).

\end{proof}

\end{document}